\documentclass[a4paper,11pt]{article}
\usepackage{amsmath,amsthm,amssymb,enumitem,xcolor}
\usepackage[hang]{footmisc}

\usepackage[nosort,nocompress,noadjust]{cite}

\usepackage[bookmarks=false,hyperfootnotes=false,colorlinks,linktoc=all,
    linkcolor={red!60!black},
    citecolor={blue!50!black},
    urlcolor={blue!80!black}]{hyperref}

\renewcommand{\eqref}[1]{\hyperref[#1]{(\ref{#1})}}

\newlist{enumlist}{enumerate}{2}
\setlist[enumlist,1]{labelindent=0cm,label=(\roman*),ref=(\roman*),labelwidth=4.5ex,labelsep=1ex,leftmargin=5.5ex,align=right,topsep=0.5ex,itemsep=1ex,parsep=1ex}
\setlist[enumlist,2]{labelindent=0cm,label=\alph*),ref=\arabic*,labelwidth=5ex,labelsep=0.5ex,leftmargin=5.5ex,align=left,topsep=0.5ex,itemsep=1ex,parsep=1ex}

\newlist{enumRom}{enumerate}{1}
\setlist[enumRom,1]{labelindent=0cm,label=(\Roman*),ref=(\Roman*),labelwidth=4.5ex,labelsep=1ex,leftmargin=5.5ex,align=right,topsep=0.5ex,itemsep=1ex,parsep=1ex}

\newlist{itemlist}{itemize}{2}
\setlist[itemlist,1]{labelindent=0cm,label=$\bullet$,labelwidth=2.5ex,labelsep=0.5ex,leftmargin=3ex,align=left,topsep=0.5ex,itemsep=1ex,parsep=1ex}
\setlist[itemlist,2]{labelindent=0cm,label=$\circ$,labelwidth=2.5ex,labelsep=0.5ex,leftmargin=3ex,align=left,topsep=0.5ex,itemsep=1ex,parsep=1ex}

\numberwithin{equation}{section}

{\theoremstyle{definition}\newtheorem{definitiona}{Definition}[section]
\newtheorem{remarka}[definitiona]{Remark}
\newtheorem{questiona}[definitiona]{Question}
\newtheorem{examplea}[definitiona]{Example}
\newtheorem{examplesa}[definitiona]{Examples}}

\newtheorem{propositiona}[definitiona]{Proposition}
\newtheorem{lemmaa}[definitiona]{Lemma}
\newtheorem{theorema}[definitiona]{Theorem}
\newtheorem{corollarya}[definitiona]{Corollary}

\newtheorem{letterthma}{Theorem}
\renewcommand{\theletterthma}{\Alph{letterthma}}

{\theoremstyle{definition}
\newtheorem{letterremarka}[letterthma]{Remark}
\renewcommand{\theletterremarka}{\Alph{letterthma}}
\newtheorem{letterdefinitiona}[letterthma]{Definition}
\renewcommand{\theletterdefinitiona}{\Alph{letterdefinitiona}}
}

\newenvironment{definition}[1][]{\begin{definitiona}[#1]\setlist*[enumlist,1]{label=(\roman*),ref=\thedefinitiona(\roman*)}}{\end{definitiona}}
\newenvironment{remark}[1][]{\begin{remarka}[#1]\setlist*[enumlist,1]{label=(\roman*),ref=\theremarka(\roman*)}}{\end{remarka}}
\newenvironment{question}[1][]{\begin{questiona}[#1]\setlist*[enumlist,1]{label=(\roman*),ref=\thequestiona(\roman*)}}{\end{questiona}}

\newenvironment{examples}[1][]{\begin{examplesa}[#1]\setlist*[enumlist,1]{label=(\roman*),ref=\theexamplesa(\roman*)}}{\end{examplesa}}
\newenvironment{proposition}[1][]{\begin{propositiona}[#1]\setlist*[enumlist,1]{label=(\roman*),ref=\thepropositiona(\roman*)}}{\end{propositiona}}
\newenvironment{lemma}[1][]{\begin{lemmaa}[#1]\setlist*[enumlist,1]{label=(\roman*),ref=\thelemmaa(\roman*)}}{\end{lemmaa}}
\newenvironment{theorem}[1][]{\begin{theorema}[#1]\setlist*[enumlist,1]{label=(\roman*),ref=\thetheorema(\roman*)}}{\end{theorema}}
\newenvironment{corollary}[1][]{\begin{corollarya}[#1]\setlist*[enumlist,1]{label=(\roman*),ref=\thecorollarya(\roman*)}}{\end{corollarya}}
\newenvironment{letterthm}[1][]{\begin{letterthma}[#1]\setlist*[enumlist,1]{label=(\roman*),ref=\theletterthma(\roman*)}}{\end{letterthma}}

\newenvironment{letterdefinition}[1][]{\begin{letterdefinitiona}[#1]\setlist*[enumlist,1]{label=(\roman*),ref=\theletterdefinitiona(\roman*)}}{\end{letterdefinitiona}}

\renewcommand{\Im}{\operatorname{Im}}

\newcommand{\C}{\mathbb{C}}
\newcommand{\F}{\mathbb{F}}
\newcommand{\N}{\mathbb{N}}
\newcommand{\R}{\mathbb{R}}
\newcommand{\T}{\mathbb{T}}

\newcommand{\Z}{\mathbb{Z}}

\newcommand{\cA}{\mathcal{A}}
\newcommand{\cB}{\mathcal{B}}
\newcommand{\cC}{\mathcal{C}}

\newcommand{\cE}{\mathcal{E}}
\newcommand{\cF}{\mathcal{F}}
\newcommand{\cG}{\mathcal{G}}

\newcommand{\cI}{\mathcal{I}}
\newcommand{\cJ}{\mathcal{J}}
\newcommand{\cK}{\mathcal{K}}

\newcommand{\cM}{\mathcal{M}}
\newcommand{\cN}{\mathcal{N}}

\newcommand{\cR}{\mathcal{R}}
\newcommand{\cS}{\mathcal{S}}
\newcommand{\cT}{\mathcal{T}}
\newcommand{\cU}{\mathcal{U}}

\newcommand{\cW}{\mathcal{W}}
\newcommand{\cZ}{\mathcal{Z}}

\newcommand{\al}{\alpha}
\newcommand{\be}{\beta}
\newcommand{\eps}{\varepsilon}
\newcommand{\vphi}{\varphi}
\newcommand{\om}{\omega}
\newcommand{\Om}{\Omega}
\newcommand{\si}{\sigma}

\newcommand{\Ad}{\operatorname{Ad}}

\newcommand{\Aut}{\operatorname{Aut}}

\newcommand{\Char}{\operatorname{Char}}

\newcommand{\End}{\operatorname{End}}
\newcommand{\Fix}{\operatorname{Fix}}
\newcommand{\GL}{\operatorname{GL}}

\newcommand{\Hom}{\operatorname{Hom}}
\newcommand{\Inn}{\operatorname{Inn}}
\newcommand{\Irr}{\operatorname{Irr}}
\newcommand{\Ker}{\operatorname{Ker}}

\newcommand{\Out}{\operatorname{Out}}
\newcommand{\PGL}{\operatorname{PGL}}

\newcommand{\PSL}{\operatorname{PSL}}

\newcommand{\SL}{\operatorname{SL}}

\newcommand{\lspan}{\operatorname{span}}

\newcommand{\Tr}{\operatorname{Tr}}

\newcommand{\Atil}{\widetilde{A}}
\newcommand{\cAtil}{\widetilde{\mathcal{A}}}
\newcommand{\altil}{\widetilde{\alpha}}

\newcommand{\Dtil}{\widetilde{D}}

\newcommand{\Gtil}{\widetilde{G}}

\newcommand{\gammatil}{\widetilde{\gamma}}

\newcommand{\Omtil}{\widetilde{\Omega}}
\newcommand{\pitil}{\widetilde{\pi}}
\newcommand{\psitil}{\widetilde{\psi}}
\newcommand{\Stil}{\widetilde{S}}

\newcommand{\Ghat}{\widehat{G}}

\newcommand{\Jhat}{\widehat{J}}

\newcommand{\That}{\widehat{T}}

\newcommand{\ombar}{\overline{\omega}}
\newcommand{\Xbar}{\overline{X}}
\newcommand{\Ybar}{\overline{Y}}

\newcommand{\ot}{\otimes}
\newcommand{\otalg}{\otimes_{\text{\rm alg}}}
\newcommand{\id}{\mathord{\text{\rm id}}}
\newcommand{\ovt}{\mathbin{\overline{\otimes}}}
\newcommand{\actson}{\curvearrowright}

\newcommand{\dpr}{^{\prime\prime}}
\newcommand{\op}{^\text{\rm op}}
\newcommand{\Autgr}{\operatorname{Aut}_{\text{\rm gr}}}
\newcommand{\Autpmp}{\operatorname{Aut}_{\text{\rm pmp}}}
\newcommand{\otmin}{\otimes_{\text{\rm min}}}

\newcommand{\Autfield}{\operatorname{Aut}_{\text{\rm field}}}

\newcommand{\inv}{^{-1}}

\newcommand{\bim}[3]{\mathord{\raisebox{-0.4ex}[0ex][0ex]{\scriptsize $#1$}{#2}\hspace{-0.2ex}\raisebox{-0.4ex}[0ex][0ex]{\scriptsize $#3$}}}

\begin{document}

\begin{center}
{\boldmath\LARGE\bf W$^*$-superrigidity for discrete quantum groups}

\vspace{1ex}

{\sc by Milan Donvil\footnote{KU Leuven, Department of Mathematics, Leuven (Belgium), milan.donvil@kuleuven.be\\ Supported by PhD grant 1162024N funded by the Research Foundation Flanders (FWO)} and Stefaan Vaes\footnote{KU~Leuven, Department of Mathematics, Leuven (Belgium), stefaan.vaes@kuleuven.be\\ Supported by FWO research project G016325N of the Research Foundation Flanders and by Methusalem grant METH/21/03 –- long term structural funding of the Flemish Government.}}
\end{center}

\begin{abstract}\noindent
A discrete group $G$ is called W$^*$-superrigid if the group $G$ can be entirely recovered from the ambient group von Neumann algebra $L(G)$. We introduce an analogous notion for discrete quantum groups. We prove that this strengthened quantum W$^*$-superrigidity property holds for a natural family of co-induced discrete quantum groups. We also prove that, remarkably, most existing families of W$^*$-superrigid groups are not quantum W$^*$-superrigid.
\end{abstract}

\section{Introduction and main results}

A discrete group $G$ is called W$^*$-superrigid if $G$ can be entirely recovered from the group von Neumann algebra $L(G)$. This means that whenever $\Lambda$ is a discrete group such that the von Neumann algebras $L(G)$ and $L(\Lambda)$ are isomorphic, then $G \cong \Lambda$. The first W$^*$-superrigidity theorem for group von Neumann algebras $L(G)$ was obtained in \cite{IPV10}, using Popa's deformation/rigidity theory, for groups $G$ given by a generalized wreath product construction. In \cite{BV12}, it was proven that for many groups $\Gamma$, including all free groups and all free products $\Gamma = \Gamma_1 \ast \Gamma_2$ of nontrivial amenable groups with $|\Gamma_1| + |\Gamma_2| \geq 5$, the left-right wreath product group $G = (\Z/2\Z)^{(\Gamma)} \rtimes (\Gamma \times \Gamma)$ is W$^*$-superrigid.

Recently, several new degrees of W$^*$-superrigidity have been discovered. In \cite{CIOS21}, a family of property (T) groups was shown to be W$^*$-superrigid, providing the first positive result on Connes' rigidity conjecture, which remains wide open. In \cite{DV24a}, again for left-right wreath product groups $G$, it was shown that W$^*$-superrigidity also holds up to virtual isomorphisms: if $\Lambda$ is any discrete group such that $L(G)$ is virtually isomorphic to $L(\Lambda)$, then $G$ must be virtually isomorphic to $\Lambda$. Moreover, it was shown in \cite{DV24a} that W$^*$-superrigidity even holds for the cocycle twisted group von Neumann algebras $L_\mu(G)$. Finally in \cite{DV24b,CFQOT25}, the first families of W$^*$-superrigid groups with infinite center were obtained.

Not only discrete groups $G$ have a canonical group von Neumann algebra $L(G)$. Also discrete \emph{quantum} groups $\Lambda$ generate a von Neumann algebra that we might denote as $L(\Lambda)$. The W$^*$-superrigidity theorems mentioned so far do not prevent the existence of a discrete quantum group $\Lambda$ such that $L(G) \cong L(\Lambda)$ while $G \not\cong \Lambda$. Even more so, all the W$^*$-superrigidity theorems in \cite{IPV10,BV12,DV24a} are about variants of wreath products $G = (\Z/2\Z)^{(I)} \rtimes \Gamma$. As we will see, for all these wreath product groups $G$, there exist discrete quantum groups $\Lambda$ such that $L(G) \cong L(\Lambda)$ and $G \not\cong \Lambda$.

The ``discrete quantum group von Neumann algebras'' $L(\Lambda)$ mentioned above can more conveniently (and equivalently) be considered as the von Neumann algebras underlying \emph{compact quantum groups} in the sense of Woronowicz \cite{Wor87,Wor95,MVD98}. These are von Neumann algebras $A$ equipped with a comultiplication $\Delta_A : A \to A \ovt A$ that is co-associative and for which there exists an invariant Haar state (see Definition \ref{def.cqg}). The two classical examples of compact quantum groups are $(L^\infty(K),\Delta_K)$ when $K$ is a compact group and $(\Delta_K(F))(k_1,k_2) = F(k_1 k_2)$, and $(L(G),\Delta_G)$ when $G$ is a discrete group and $\Delta_G(u_g) = u_g \ot u_g$.

Two compact quantum groups $(A,\Delta_A)$ and $(B,\Delta_B)$ are called \emph{isomorphic} if there exists a von Neumann algebra isomorphism $\pi : A \to B$ satisfying $\Delta_B \circ \pi = (\pi \ot \pi) \circ \Delta_A$. We can then formalize the discussion above in the following way.

\begin{letterdefinition}\label{def.quantum-Wstar-superrigid}
We say that a compact quantum group $(A,\Delta_A)$ is \emph{quantum W$^*$-superrigid} if the following holds: if $(B,\Delta_B)$ is any compact quantum group such that $B \cong A$ as von Neumann algebras, then $(A,\Delta_A) \cong (B,\Delta_B)$ as compact quantum groups.

We say that a discrete group $G$ is \emph{quantum W$^*$-superrigid} if $(L(G),\Delta_G)$ is quantum W$^*$-superrigid in the sense above.
\end{letterdefinition}

If $G$ is quantum W$^*$-superrigid, then $G$ is also W$^*$-superrigid in the usual sense of the word: if $L(G) \cong L(\Lambda)$ for any discrete group $\Lambda$, then $G \cong \Lambda$. The converse however is not true. By \cite[Theorem B]{BV12}, for every countable group $\Gamma$ in the wide class $\cC$ (see Section \ref{sec.coarse-embeddings-co-induced}), the left-right wreath product group $(\Z/2\Z)^{(\Gamma)} \rtimes (\Gamma \times \Gamma)$ is W$^*$-superrigid. By Corollary \ref{cor.not-rigid}, none of these groups are quantum W$^*$-superrigid.

Whenever $\Gamma$ is a countable group and $\Gamma \actson^\be (A_0,\Delta_0)$ is an action by quantum group automorphisms of a compact quantum group $(A_0,\Delta_0)$ with Haar state $\vphi_0$, we consider the co-induced left-right Bernoulli action $\Gamma \times \Gamma \actson^\al (A,\vphi) = (A_0,\vphi_0)^\Gamma$ given by $\al_{(g,h)}(\pi_k(a)) = \pi_{gkh^{-1}}(\be_g(a))$. Then the crossed product $M = A \rtimes_\al (\Gamma \times \Gamma)$ has a natural compact quantum group structure (see Section \ref{sec.vanishing-and-nonvanishing}).

When $q$ is a prime power, we denote by $\F_q$ the unique finite field of order $q$. Our main result is the following.

\begin{letterthm}\label{thm.main}
For each of the following actions $\Gamma \actson^\be (A_0,\Delta_0)$, the co-induced left-right Bernoulli crossed product gives a quantum W$^*$-superrigid compact quantum group $(M,\Delta)$.
\begin{enumlist}
\item\label{thm.main.one} Take any torsion free amenable icc group $\Lambda$. Consider $(A_0,\Delta_0) = (L(\Lambda),\Delta_\Lambda)$ and let $\Gamma = \Lambda \ast \Z$ act on $(A_0,\Delta_0)$ by $\be_g = \Ad u_g$ for $g \in \Lambda$ and $\be_a = \id$ for $a \in \Z$.
\item\label{thm.main.two} Take $(A_0,\Delta_0) = (L^\infty(K),\Delta_K)$ where $K=\T^n$, $n \geq 3$, or $K$ is one of the following finite groups. Define $\Gamma = \F_{\Aut K}$ as the free group with free generators indexed by $\Aut K$, and define $\Gamma \actson^\be K$ by $\be_\al(a) = \al(a)$ for every $\al \in \Aut K$.
\begin{itemlist}
\item $\SL_2(\F_p)$ with $p$ prime and $p \geq 5$.
\item $\SL_n(\F_q)$, or $\F_q^n \rtimes \SL_n(\F_q)$, with $n \geq 3$, $q$ a prime power, $(n,q) \not\in \{(3,2),(3,4),(4,2)\}$,
\item the canonical double cover $\Atil_n$ of the alternating group $A_n$ with $n=5$ or $n \geq 8$.
\end{itemlist}
\end{enumlist}
\end{letterthm}

Note that the example in Theorem \ref{thm.main.one} and the first example in Theorem \ref{thm.main.two} with $K = \T^n$ produce quantum W$^*$-superrigid discrete groups: $(M,\Delta) \cong (L(G),\Delta_G)$. The other examples in Theorem \ref{thm.main.two} produce quantum W$^*$-superrigid compact quantum groups $(M,\Delta)$ that are neither commutative, nor co-commutative.

Theorem \ref{thm.main} is a consequence of the much more precise Theorem \ref{thm.generic-theorem} below, together with Examples \ref{ex.main}. In Remark \ref{rem.no-go-H2}, we also explain that without the restrictions on $n$, $p$ and $q$ in Theorem \ref{thm.main.two}, quantum W$^*$-superrigidity fails.

In order to prove a quantum W$^*$-superrigidity theorem such as Theorem \ref{thm.main}, we need to combine three different types of results, each having a quite different mathematical flavor. This need to combine three directions also explains the length of this paper. The three distinct parts of the paper can best be understood by considering the following potential obstructions to quantum W$^*$-superrigidity.

Consider one of the compact quantum groups $(M,\Delta)$ as in Theorem \ref{thm.main}.

{\bf Part 1.} Whenever $\Om \in M \ovt M$ is a unitary satisfying a $2$-cocycle relation (see Definition \ref{def.unitary-2-cocycle}), the $*$-homomorphism $\Delta_\Om : M \to M \ovt M : \Delta_\Om(a) = \Om \Delta(a) \Om^*$ defines another compact quantum group structure on the same underlying von Neumann algebra $M$. We thus need methods to analyze, and ultimately prove vanishing, of unitary $2$-cocycles on crossed product quantum groups. We prove such results in Section \ref{sec.quantum-groups-2-cocycles}. In Section \ref{sec.cohomological-obstructions-superrigidity}, we analyze the relation between nonvanishing of a unitary $2$-cocycle $\Om$ on $(M,\Delta)$ and the, in general stronger, property that the quantum groups $(M,\Delta)$ and $(M,\Delta_\Om)$ are nonisomorphic. Since we expect that several readers of this paper might be unfamiliar with Woronowicz' theory of compact quantum groups, we provide a fully self contained introduction to the theory in Section \ref{sec.cqg-basics}. In particular, we reprove several known results on unitary $2$-cocycles, mainly due to \cite{DeC10} and \cite{DMN21}.

{\bf Part 2.}  The construction of $(M,\Delta)$ in Theorem \ref{thm.main} is functorial in $\Gamma \actson^\be (A_0,\vphi_0)$. This implies the following. Assume that $\Gamma \actson^\gamma (A_1,\vphi_1)$ is another action by quantum group automorphisms. The same co-induced left-right Bernoulli crossed product gives a compact quantum group $(M_1,\Delta_1)$. If now the actions $\Gamma \actson^\be A_0$ and $\Gamma \actson^\gamma A_1$ are conjugate as actions on von Neumann algebras, forgetting about the quantum group structure, then $M \cong M_1$ and W$^*$-superrigidity may fail. We thus need from the start a rigidity property for $(A_0,\vphi_0)$, relative to this given action $\Gamma \actson^\be (A_0,\vphi_0)$. Roughly speaking: conjugacy of $\Gamma \actson^\be A_0$ and $\Gamma \actson^\gamma A_1$ as actions on von Neumann algebras should imply conjugacy as actions on quantum groups. This means that we should recover the quantum group structure $\Delta_0$ on $A_0$ by knowing the von Neumann algebra $A_0$ and knowing that each automorphism $\be_g$ is a quantum group automorphism. We introduce these notions of \emph{relative rigidity} in Section \ref{sec.relative-rigidity}.

The main part of Section \ref{sec.relative-rigidity} then consists of proving relative rigidity for various classes of examples. In particular, we prove that for every icc group $G$, the compact quantum group $(L(G),\Delta_G)$ is rigid relative to the action $G \actson (L(G),\Delta_G)$ given by conjugacy. In Section \ref{sec.examples-relative-rigidity-noncommutative}, we prove relative rigidity for several classes of finite groups $K$. Roughly speaking, we have to recover the group structure on $K$ by only knowing that a given group of permutations consists of group automorphisms. Surprisingly, this property holds for quite a few classes of finite groups, such as $\SL_n(\F_q)$ and the symmetric and alternating groups, but the proofs are rather involved and are, of course, purely combinatorial and finite group theoretic in nature. On the other hand, for non abelian connected compact groups, relative rigidity never holds; see Proposition \ref{prop.no-go-connected-Lie}. In Proposition \ref{prop.direct-product-relative-rigid}, we give further counterexamples to relative rigidity. However, in Question \ref{question.all-finite-simple-rigid}, we speculate that all finite simple groups are rigid relative to their automorphism group.

{\bf Part 3.} We need Popa's deformation/rigidity theory and that is the contents of Section \ref{sec.quantum-Wstar-superrigidity}. We use the comultiplication method introduced in \cite{PV09,IPV10}, and also used in \cite{BV12,DV24a}. If $\Delta_1 : M \to M \ovt M$ defines another compact quantum group structure on the same von Neumann algebra $M$, then $\Delta_1 : M \to M \ovt M$ is in particular an embedding of von Neumann algebras. While it is asking too much to classify all embeddings $M \to M \ovt M$, such comultiplication embeddings all have a property that we called \emph{coarseness} in \cite{DV24a}. Using the extensive arsenal of methods of \cite{IPV10,BV12,PV21,DV24a}, we classify in Section \ref{sec.coarse-embeddings-co-induced} all coarse embeddings $M \to M \ovt M$ for the II$_1$ factors $M$ appearing in Theorem \ref{thm.main}. Up to unitary conjugacy, they all have a ``canonical form''. The unitary realizing this unitary conjugacy will automatically be a unitary $2$-cocycle for $(M,\Delta)$ and the ``canonical form'' will say that $(M,\Delta_1)$ must arise from such other $\Gamma \actson^\gamma (A_1,\tau_1)$ as discussed in Part~2 above. This then allows us to prove in Section \ref{sec.proof-superrigidity-theorem} our most general quantum W$^*$-superrigidity Theorem \ref{thm.generic-theorem}.

{\bf Acknowledgment.} We thank Kenny De Commer for providing several references to the quantum group literature on $2$-cocycles and Galois objects.

\section{\boldmath Compact quantum groups, ergodic coactions and $2$-cocycles}\label{sec.quantum-groups-2-cocycles}

Since the context of this paper is mainly in the theory of von Neumann algebras and, in particular, in Popa's deformation/rigidity theory, we provide this rather lengthy and detailed introduction to the main aspects of Woronowicz' theory of compact quantum groups. We decided to make this section fully self-contained. In particular, we spend four pages on the container Theorem \ref{thm.cqg} that summarizes and proves all the basic properties of compact quantum groups of Kac type (i.e.\ where the Haar state is a trace) in an efficient way. We do this entirely in the von Neumann algebra context.

We need in this paper several existing and a few new results on unitary $2$-cocycles and coactions of compact quantum groups. In particular, we need that coactions of Kac type compact quantum groups on type I factors $B(K)$ are automatically trace preserving. That result was proven in \cite[Corollary 5.2]{DeC10} as a consequence of the more advanced theory of Galois co-objects. In \cite{DeC10}, the desire for an elementary proof was expressed. We provide such an elementary proof in this section.

Furthermore, a canonical way to produce a new compact quantum group structure with the same underlying von Neumann algebra is by twisting the comultiplication $\Delta : A \to A \ovt A$ to a new comultiplication $\Delta_\Om(a) = \Om \Delta(a) \Om^*$, where $\Om$ is a unitary $2$-cocycle on $(A,\Delta)$. It is therefore obvious that we need in this paper vanishing results for unitary $2$-cocycles. We prove these in Propositions \ref{prop.cocycle-reduce-to-core} and \ref{prop.cocycles-products}.

When $G$ is a discrete group, the group von Neumann algebra $L(G)$ carries the natural comultiplication
\begin{equation}\label{eq.LG-Delta}
\Delta_G : L(G) \to L(G) \ovt L(G) : \Delta_G(u_g) = u_g \ot u_g \quad\text{for all $g \in G$.}
\end{equation}
The canonical tracial state $\tau$ on $L(G)$, defined by $\tau(u_g) = 0$ for all $g \neq e$, satisfies the invariance properties
$$(\id \ot \tau)\Delta_G(a) = \tau(a) 1 = (\tau \ot \id)\Delta_G(a) \quad \text{for all $a \in L(G)$.}$$
In this way, $(L(G),\Delta_G)$ becomes a compact quantum group in the sense of Woronowicz (see Definition \ref{def.cqg}) and we view this as the dual of the discrete group $G$. Similarly, when $K$ is a compact group, we identify $L^\infty(K) \ovt L^\infty(K) = L^\infty(K \times K)$ and define the natural comultiplication
\begin{equation}\label{eq.Linfty-K-Delta}
\Delta_K : L^\infty(K) \to L^\infty(K) \ovt L^\infty(K) : \Delta_K(F)(k_1,k_2) = F(k_1 k_2) \quad\text{for all $k_1,k_2 \in K$.}
\end{equation}
Integration w.r.t.\ the Haar probability measure on $K$ provides an invariant tracial state on $(L^\infty(K),\Delta_K)$, which is thus a compact quantum group.

Since the presence of sufficiently nontrivial unitary $2$-cocycles is an obstruction to quantum W$^*$-superrigidity, it is natural to wonder why a similar obstruction to usual W$^*$-superrigidity for discrete groups $G$ did not appear before. The reason is that twisting the comultiplication $\Delta_G : L(G) \to L(G) \ovt L(G)$ by a unitary $2$-cocycle $\Om \in L(G) \ovt L(G)$ will typically only produce another group (and not quantum group) if $\Om$ is symmetric, meaning that $\si(\Om) = \Om$ where $\si$ is the flip. In \cite[Theorem 3.3]{IPV10}, it was proven that each such symmetric $2$-cocycle must be a coboundary. That result is used in all the W$^*$-superrigidity proofs in the literature.

\subsection{Compact quantum groups: definition and basic properties}\label{sec.cqg-basics}

The von Neumann algebra approach to compact quantum groups goes as follows.

\begin{definition}[{\cite{Wor87,Wor95}}]\label{def.cqg}
We call $(A,\Delta)$ a \emph{compact quantum group} if $A$ is a von Neumann algebra, $\Delta : A \to A \ovt A$ is a faithful normal unital $*$-homomorphism that is co-associative
$$(\Delta \ot \id) \Delta = (\id \ot \Delta) \Delta \; ,$$
and there exists a faithful normal state $\vphi$ on $A$ that is left and right invariant
$$(\id \ot \vphi)\Delta(a) = \vphi(a) 1 = (\vphi \ot \id)\Delta(a) \quad \text{for all $a \in A$.}$$
\end{definition}

Note that such an invariant state is necessarily unique: if also $\psi$ is invariant, we get that $\psi(a) = (\psi \ot \vphi)\Delta(a) = \vphi(a)$ for all $a \in A$. This unique invariant state is called the \emph{Haar state} of $(A,\Delta)$. If the Haar state is tracial, then $(A,\Delta)$ is said to be of \emph{Kac type}.

In \eqref{eq.LG-Delta} and \eqref{eq.Linfty-K-Delta}, we recalled the definition of the compact quantum groups $(L(G),\Delta_G)$ and $(L^\infty(K),\Delta_K)$ when $G$ is a discrete group and $K$ is a compact group.

An \emph{isomorphism} between compact quantum groups $(A,\Delta_A)$ and $(B,\Delta_B)$ is a $*$-isomorphism $\pi : A \to B$ satisfying $\Delta_B \circ \pi = (\pi \ot \pi)\circ \Delta_A$. By uniqueness of the Haar state, such an isomorphism is automatically Haar state preserving. We similarly define the notion of a \emph{quantum group automorphism} of $(A,\Delta_A)$ as an automorphism $\al$ of the von Neumann algebra $A$ satisfying $(\al \ot \al) \circ \Delta_A = \Delta_A \circ \al$.

A \emph{unitary corepresentation} of $(A,\Delta)$ on a Hilbert space $K$ is a unitary $X \in \cU(A \ovt B(K))$ satisfying $(\Delta \ot \id)(X) = X_{13} X_{23}$, where we make use of the tensor leg numbering notation. Let $X,X'$ be unitary corepresentations on $K,K'$. The \emph{direct sum} of $X$ and $X'$ is defined as the natural unitary $X \oplus X' \in A \ovt B(K \oplus K')$. The \emph{tensor product} of $X$ and $X'$ is defined as the unitary corepresentation $X_{12}X'_{13} \in A \ovt B(K \ot K')$. A bounded operator $T \in B(K,K')$ is called an \emph{intertwiner} if $(1 \ot T) X = X' (1 \ot T)$. If there exists a unitary intertwiner, then $X$ and $X'$ are said to be \emph{unitarily equivalent}. If the only intertwiners between $X$ and itself are the multiples of $1$, then $X$ is said to be \emph{irreducible}. If $P \in B(K)$ is an orthogonal projection and an intertwiner, then $X (1 \ot P)$ is a unitary corepresentation on the Hilbert space $P K$, and we call this a \emph{sub-corepresentation} of $X$. Since the intertwiners between $X$ and itself form a von Neumann algebra, it follows that $X$ is irreducible if and only if $X$ is not isomorphic to the direct sum of two unitary corepresentations.

A \emph{coaction} of a compact quantum group $(A,\Delta)$ on a von Neumann algebra $Q$ is a faithful normal unital $*$-homomorphism $\be : Q \to A \ovt Q$ satisfying $(\Delta \ot \id)\be = (\id \ot \be)\be$. One denotes by $Q^\be = \{b \in Q \mid \be(b) = 1 \ot b\}$ the fixed point algebra and says that $\be$ is \emph{ergodic} if $Q^\be = \C 1$. A normal state $\psi$ on $Q$ is said to be \emph{invariant} if $(\id \ot \psi)\be(b) = \psi(b) 1$ for all $b \in Q$. Given any faithful normal state $\psi_0$ on $Q$ and using the Haar state $\vphi$, it is easy to check that $\psi := (\vphi \ot \psi_0) \be$ defines a faithful normal invariant state on $Q$.

In the following container result, we summarize and prove the basic results of the theory of compact quantum groups of Kac type. We have a twofold motivation to include this proof. First, we expect that several readers of this paper are not fully familiar with the theory of compact quantum groups. Second, while the von Neumann algebraic Definition \ref{def.cqg} is highly natural and suitable in the context of this paper, it is not the most common one. This would make references to the literature for the following basic results rather difficult to state fully rigorously.

When $A$ is a von Neumann algebra with a faithful normal state $\vphi$, we call $(H,\xi_0)$ the \emph{GNS-construction} of $(A,\vphi)$ if $H$ is the completion of $A$ given by the scalar product $\langle a,b\rangle = \vphi(b^* a)$, $\xi_0 \in H$ is given by $1 \in A$ and $A$ is represented on $H$ by left multiplication.

\begin{theorem}[{\cite{Wor87,Wor95}}]\label{thm.cqg}
Let $(A,\Delta)$ be a compact quantum group of Kac type with Haar state $\vphi$ and GNS-construction $(H,\xi_0)$.
\begin{enumlist}
\item\label{cqg.1} If $\be : Q \to A \ovt Q$ is a coaction with faithful normal invariant state $\psi$ with GNS-construction $(K,\eta_0)$, there is a unique $X \in A \ovt B(K)$ satisfying $X^* (1 \ot b \eta_0) = \be(b) (1 \ot \eta_0)$ for all $b \in Q$. This $X$ is a unitary corepresentation of $(A,\Delta)$.
\item\label{cqg.2} There is a unique $W \in A \ovt B(H)$ satisfying $W^* (1 \ot a \xi_0) = \Delta(a) (1 \ot \xi_0)$ for all $a \in A$. This $W$ is a unitary corepresentation of $(A,\Delta)$.
\item\label{cqg.3} Every unitary corepresentation of $(A,\Delta)$ is a direct sum of irreducible unitary corepresentations. Every irreducible unitary corepresentation is finite dimensional.
\item\label{cqg.4} If $X \in A \ot M_n(\C)$ is a unitary corepresentation, then $\overline{X} \in A \ot M_n(\C)$ defined by $(\overline{X})_{ij} = X_{ij}^*$ is again a unitary corepresentation, which is called the contragredient of $X$.
\item\label{cqg.5} The linear span of all coefficients $X_{ij}$ of finite dimensional unitary corepresentations $X \in A \ot M_n(\C)$ is a $*$-strongly dense $*$-subalgebra $\cA \subset A$.
\item\label{cqg.6} If $\Irr$ is a complete set of inequivalent irreducible unitary corepresentations $X \in A \ot M_n(\C)$, where $n = d(X)$ is the dimension, then $\{\sqrt{d(X)} X_{ij} \xi_0 \mid X \in \Irr, 1 \leq i,j \leq d(X)\}$ is an orthonormal basis of $H$.
\item\label{cqg.7} There is a unique anti-unitary operator $\Jhat : H \to H$ satisfying $\Jhat (X_{ij} \xi_0) = X_{ji} \xi_0$ for all finite dimensional unitary corepresentations $X \in A \ot M_n(\C)$. Then $\Jhat A \Jhat = A$ and the $*$-anti-automorphism $S : A \to A : S(a) = \Jhat a^* \Jhat$ satisfies $S((\id \ot \om)(Y)) = (\id \ot \om)(Y^*)$ for every unitary corepresentation $Y \in A \ovt B(K)$ and all $\om \in B(K)_*$. Also, $S^2 = \id$.
\item\label{cqg.8} If $\be : Q \to A \ovt Q$ is a coaction with faithful normal invariant state $\psi$ and $(\si_t^\psi)_{t \in \R}$ is the modular automorphism group of $\psi$, then $\be \circ \si_t^\psi = (\id \ot \si_t^\psi) \circ \be$ for all $t \in \R$.
\item\label{cqg.9} The von Neumann algebra $A$ is commutative if and only if $(A,\Delta) \cong (L^\infty(K),\Delta_K)$ for some compact group $K$.
\item\label{cgq.10} We have that $(A,\Delta)$ is co-commutative, meaning that $\Delta = \si \circ \Delta$ where $\si$ is the flip on $A \ovt A$, if and only if $(A,\Delta) \cong (L(G),\Delta_G)$ for some discrete group $G$.
\end{enumlist}
\end{theorem}

\begin{proof}
(i) Since $(\id \ot \psi)\be(c^* b) = \psi(c^* b)1$ for all $b,c \in Q$, there is a unique $X \in A \ovt B(K)$ satisfying $X^* (1 \ot b \eta_0) = \be(b) (1 \ot \eta_0)$ for all $b \in Q$. Note that $X$ is a co-isometry, i.e.\ $XX^* = 1$. Then
\begin{align*}
(\Delta \ot \id)(X^*)(1 \ot 1 \ot b \eta_0) &= (\Delta \ot \id)(\be(b) (1 \ot \eta_0)) = (\id \ot \be)\be(b) (1 \ot 1 \ot \eta_0) \\
&= X^*_{23} (\be(b) (1 \ot \eta_0))_{13} = X^*_{23} X^*_{13} (1 \ot 1 \ot b \eta_0)
\end{align*}
for all $b \in Q$, so that $(\Delta \ot \id)(X) = X_{13} X_{23}$. Since $X$ is a co-isometry, we can define the projection $p = X^* X \in A \ovt B(K)$. Then
\begin{equation}\label{eq.comp}
(\Delta \ot \id)(p) = X^*_{23} X^*_{13} X_{13} X_{23} = X^*_{23} p_{13} X_{23} \leq X^*_{23} X_{23} = 1 \ot p \; .
\end{equation}
So, $1 \ot p - (\Delta \ot \id)(p) \geq 0$. Since
$$(\vphi \ot \vphi \ot \id)(1 \ot p - (\Delta \ot \id)(p)) = (\vphi \ot \id)(p) - (\vphi \ot \id)(p) = 0$$
and $\vphi$ is faithful, it follows that $(\Delta \ot \id)(p) = 1 \ot p$. Applying $\vphi \ot \id \ot \id$, we conclude that $p = 1 \ot p_0$, where $p_0 = (\vphi \ot \id)(p)$.

In \eqref{eq.comp}, we have seen that $(\Delta \ot \id)(p) = X^*_{23} p_{13} X_{23}$. Since $p = 1 \ot p_0$, we conclude that $1 \ot p_0 = X^* (1 \ot p_0) X$. Since $X$ is a co-isometry, it follows that $X (1 \ot p_0) X^* = 1 \ot p_0$. Since $p = 1 \ot p_0$ is the range projection of $X^*$, this means that $1 = X X^* = 1 \ot p_0 = p$. We conclude that $X$ is a unitary. We already proved that $(\Delta \ot \id)(X) = X_{13} X_{23}$, so that $X$ is a unitary corepresentation.

(ii) This is a special case of (i) because we may view $\Delta$ as a coaction of $(A,\Delta)$ on $A$ with invariant state $\vphi$.

(iii) Fix a unitary corepresentation $X \in A \ovt B(K)$. Define the coaction $\be : B(K) \to A \ovt B(K) : \be(b) = X^* (1 \ot b) X$ for all $b \in B(K)$ and denote by $Q = B(K)^\be$ its fixed point algebra. Since $\vphi$ is a faithful normal invariant state, the formula $E : B(K) \to Q : E(b) = (\vphi \ot \id)\be(b)$ defines a faithful normal conditional expectation of $B(K)$ onto $Q$. So, $Q$ is discrete and its identity $1$ can be written as a sum of minimal projections $p_k$. For every $k$, we have that $X(1 \ot p_k)$ is an irreducible unitary corepresentation on the Hilbert space $p_k K$.

To conclude the proof of (iii), we only need to prove that if $X \in A \ovt B(K)$ is an irreducible unitary corepresentation, then $K$ is finite dimensional. Use the same notation as in the previous paragraph. Choose a minimal projection $q \in B(K)$. By irreducibility, $\be$ is ergodic, so that $E(q)$ is a nonzero multiple $\al 1$ of $1$. Denote by $\Tr$ the trace on $B(K)$. Since $\vphi$ is a trace, we get that
$$\al \, \Tr(1) = \Tr(E(q)) = (\vphi \ot \Tr)\be(q) = (\vphi \ot \Tr)(X^* (1 \ot q) X) = (\vphi \ot \Tr)(1 \ot q) = 1 \; .$$
So, $\Tr(1) < +\infty$ and $K$ is finite dimensional.

(iv) Since $X \in A \ot M_n(\C)$ is a unitary corepresentation, we have that $\Delta(X_{ij}) = \sum_k X_{ik} \ot X_{kj}$. It follows that $\Delta(X_{ij}^*) = \sum_k X_{ik}^* \ot X_{kj}^*$, meaning that $(\Delta \ot \id)(\overline{X}) = \overline{X}_{13} \overline{X}_{23}$. So,
$$(\Delta \ot \id)({\overline{X}}^* {\overline{X}}) = {\overline{X}}^*_{23} ({\overline{X}}^* {\overline{X}})_{13} {\overline{X}}_{23} \; .$$
Applying $\vphi \ot \id \ot \id$ and defining $T = (\vphi \ot \id)({\overline{X}}^* {\overline{X}}) \in M_n(\C)$, we get that
\begin{equation}\label{eq.comp2}
1 \ot T = {\overline{X}}^* (1 \ot T) {\overline{X}} \; .
\end{equation}
Note that because $\vphi$ is a trace,
$$T_{ij} = \sum_k \vphi\bigl( ({\overline{X}}^*)_{ik} {\overline{X}}_{kj} \bigr) = \sum_k \vphi\bigl( X_{ki} X^*_{kj} \bigr) = \sum_k \vphi\bigl( X^*_{kj} X_{ki} \bigr) = \vphi((X^* X)_{ji}) = \delta_{i,j} \; .$$
So, $T = 1$ and it follows from \eqref{eq.comp2} that $\overline{X}$ is an isometry. Since $A \ovt M_n(\C)$ admits a faithful tracial state, this means that $\overline{X}$ is a unitary. We already proved that $(\Delta \ot \id)(\overline{X}) = \overline{X}_{13} \overline{X}_{23}$, so that $\overline{X}$ is a unitary corepresentation.

(v) Because we can take direct sums, tensor products and contragredients, $\cA$ is unital a $*$-subalgebra of $A$. Denote by $A_1 \subset A$ the strong$^*$-closure of $\cA$. We have to prove that $A_1 = A$. Denote by $E_1 : A \to A_1$ the unique $\vphi$-preserving conditional expectation. Choose $a \in A$. We have to prove that $a \in A_1$. Replacing $a$ by $a-E_1(a)$, it suffices to prove that every element $a \in A$ satisfying $E_1(a) = 0$ must be equal to $0$.

Define the unitary corepresentation $W$ as in (ii). By (iii), $W$ is unitarily equivalent to a direct sum of finite dimensional irreducible corepresentations. Using the vector functional $\om_{c \xi_0,b^* \xi_0} \in B(H)_*$ defined by $\om_{c \xi_0,b^* \xi_0}(T) = \langle T c \xi_0 , b^* \xi_0 \rangle$, it follows from the definition of $W$ that
$$(\id \ot \om_{c \xi_0,b^* \xi_0})(W) = (\id \ot \vphi)(\Delta(b)(1 \ot c)) \quad\text{for all $b,c \in A$.}$$
We conclude that $(\id \ot \vphi)(\Delta(b)(1 \ot c)) \in A_1$ for all $b,c \in A$. Since $a$ is orthogonal to $A_1$, it follows that
\begin{equation}\label{eq.comp3}
(\vphi \ot \vphi)((a^* \ot 1)\Delta(b)(1 \ot c)) = 0 \quad\text{for all $b,c \in A$.}
\end{equation}
Denote by $\sigma : A \ovt A \to A \ovt A$ the flip automorphism. By either reasoning in the same way as in the proof of (i), or by applying (i) to the coaction $\Delta\op = \sigma \circ \Delta$ of the compact quantum group $(A,\Delta\op)$ on $A$, we find a unitary $V \in B(H) \ovt A$ satisfying $V (b \xi_0 \ot 1) = \Delta(b) (\xi_0 \ot 1)$ for all $b \in A$. It follows that the linear span of $\Delta(b)(1 \ot c) (\xi_0 \ot \xi_0) = V(b \xi_0 \ot c \xi_0)$, $b,c \in A$, is dense in $H \ot H$. In \eqref{eq.comp3}, we may thus replace $\Delta(b) (1 \ot c)$ by $a \ot 1$ and conclude that $\vphi(a^* a) = 0$, so that $a= 0$.

(vi) Take irreducible unitary corepresentations $X \in A \ot M_n(\C)$ and $Y \in A \ot M_m(\C)$. For every $T_0 \in M_{n,m}(\C)$, the element $T := (\vphi \ot \id)(X^* (1 \ot T_0) Y)$ is an intertwiner between $X$ and $Y$. So if $X$ and $Y$ are inequivalent, $T=0$ for every choice of $T_0$. This applies in particular to the matrix unit $T_0 = e_{ik}$, from which it follows that the vectors $X_{ij} \xi_0$ and $Y_{kl} \xi_0$ are orthogonal for all $i,j,k,l$.

Similarly, $(\vphi \ot \id)(X^*(1 \ot T_0)X)$ is a multiple of $1$ for every $T_0 \in M_n(\C)$. We thus find scalars $\al(i,k)$ such that
\begin{equation}\label{eq.orthogonal-here}
(\vphi \ot \id)(X^*(1 \ot e_{ik})X) = \al(i,k) 1 \quad\text{so that}\quad \vphi(X_{ij}^* X_{kl}) = \al(i,k) \, \delta_{j,l} \; .
\end{equation}
Since $Y = X^*$ is an irreducible unitary corepresentation of $(A,\Delta\op)$, which has the same Haar state $\vphi$, we also find scalars $\be(i,k)$ such that
$$\be(i,k) \, \delta_{j,l} = \vphi(Y_{ij}^* Y_{kl}) = \vphi(X_{ji} X_{lk}^*) = \vphi(X_{lk}^* X_{ji}) = \al(l,j) \, \delta_{k,i}$$
for all $i,j,k,l$. So there exists a single scalar $\al$ such that $\al(i,k) = \al \, \delta_{i,k}$. Since $\sum_i X_{ij}^* X_{ij} = 1$ for all $j$, it follows from \eqref{eq.orthogonal-here} that $\al = n^{-1}$.

In combination with the previous paragraph, the vectors $\{\sqrt{d(X)} X_{ij} \xi_0 \mid X \in \Irr, 1 \leq i,j \leq d(X)\}$ are orthonormal. By (v), their linear span is dense, so that they form an orthonormal basis of $H$.

(vii) Let $\Irr$ be a complete set of inequivalent irreducible unitary corepresentations of $(A,\Delta)$. By (vi), there is a unique anti-unitary operator $\Jhat : H \to H$ satisfying $\Jhat X_{ij} \xi_0 = X_{ji} \xi_0$ for all $X \in \Irr$ and all $1 \leq i,j \leq \dim X$.

Since every finite dimensional unitary corepresentation is unitarily equivalent to a direct sum of irreducible unitary corepresentations, the equality $\Jhat X_{ij} \xi_0 = X_{ji} \xi_0$ holds for all finite dimensional unitary corepresentations $X$.

Let $X \in A \ot M_n(\C)$ and $Y \in A \ot M_m(\C)$ be finite dimensional unitary corepresentations. Denote by $Z = {\overline{X}}_{12} Y_{13}$ the tensor product of $\overline{X}$ and $Y$. Then,
$$\Jhat X_{ij}^* \Jhat (Y_{kl} \xi_0) = \Jhat X_{ij}^* (Y_{lk} \xi_0) = \Jhat( Z_{(i,l),(j,k)} \xi_0) = Z_{(j,k),(i,l)} \xi_0 = X_{ji}^* (Y_{kl} \xi_0) \; .$$
It follows that $\Jhat X_{ij}^* \Jhat = X_{ji}^*$. So, $\Jhat \cA \Jhat = \cA$. By (v), $\Jhat A \Jhat = A$ and the $*$-anti-automorphism $S : A \to A : S(a) = \Jhat a^* \Jhat$ satisfies $S(X_{ij}) = X_{ji}^*$ for all finite dimensional unitary corepresentations $X$. Since every unitary corepresentation $Y \in A \ovt B(K)$ is unitarily conjugate to a direct sum of irreducibles, also $S((\id \ot \om)(Y)) = (\id \ot \om)(Y^*)$ for all $\om \in B(K)_*$.

By definition, $\Jhat \circ \Jhat = \id$, so that $S^2 = \id$.

(viii) Denote by $(K,\eta_0)$ the GNS-construction of $\psi$. Denote by $S_\psi$ the modular conjugation of $\psi$, i.e.\ the closed antilinear operator on $K$ defined as the closure of $b \eta_0 \mapsto b^* \eta_0$, $b \in Q$. Denote by $X \in A \ovt B(K)$ the unitary corepresentation given by (i). By definition,
$$(\om \ot \id)(X^*) (b \eta_0) = ((\om \ot \id)\be(b)) \eta_0 \quad\text{for all $b \in Q$ and $\om \in A_*$.}$$
Since $\be$ is a $*$-homomorphism, it follows that $(\om \ot \id)(X^*) S_\psi \subset S_\psi (\ombar \ot \id)(X^*)$ for all $\om \in A_*$, where $\ombar(a) = \overline{\om(a^*)}$ for all $a \in A$. Taking adjoints, we also find that $(\om \ot \id)(X) S_\psi^* \subset S_\psi^* (\ombar \ot \id)(X)$ for all $\om \in A_*$.

By (vii), because $X$ is a unitary corepresentation, we have that $(\mu \ot \id)(X) = (\mu \circ S \ot \id)(X^*)$ for all $\mu \in A_*$. It follows that
\begin{align*}
(\om \ot \id)(X) S_\psi^* S_\psi \subset S_\psi^* (\ombar \ot \id)(X) S_\psi &= S_\psi^* (\ombar \circ S \ot \id)(X^*) S_\psi \subset S_\psi^* S_\psi (\om \circ S \ot \id)(X^*) \\ &= S_\psi^* S_\psi (\om \ot \id)(X)
\end{align*}
for all $\om \in A_*$. Since the modular operator $\Delta_\psi = S_\psi^* S_\psi$ is a positive, self-adjoint, nonsingular operator, it follows that $(\om \ot \id)(X) \Delta_\psi^{it} = \Delta_\psi^{it} (\om \ot \id)(X)$ for all $\om \in A_*$ and $t \in \R$. So, $X$ commutes with $1 \ot \Delta_\psi^{it}$. Since $\be(b) = X^* (1 \ot b) X$ for all $b \in Q$, the formula $\be \circ \si_t^\psi = (\id \ot \si_t^\psi) \circ \be$ follows.

(ix) If $(A,\Delta) \cong (L^\infty(K),\Delta_K)$, it is trivial that $A$ is commutative. Conversely, assume that $A$ is commutative. Denote by $B \subset A$ the operator norm closure of the dense $*$-subalgebra $\cA \subset A$. Since $B$ is an abelian unital C$^*$-algebra, we may identify $B$ with $C(K)$, where $K$ is a compact Hausdorff space. Since $\Delta(\cA) \subset \cA \otalg \cA$, we have $\Delta(B) \subset B \otmin B$ and find a continuous map $m : K \times K \to K$ such that $\Delta(F)(g,h) = F(m(g,h))$ for all $F \in B$ and $g,h \in K$.

In (vii), we defined the $*$-anti-automorphism $S : A \to A$. Since $S(\cA) = \cA$, also $S(B) = B$. We thus find a homeomorphism $I : K \to K$ such that $S(F) = F \circ I$ for all $F \in B$. Since $\Delta$ is co-associative, the multiplication $m$ is associative and we write $m(g,h) = g \cdot h$. For every unitary corepresentation $X \in A \ot M_n(\C)$, we have by definition
\begin{equation}\label{eq.making-inverses}
X_{ij}(I(g) \cdot g) = \Delta(X_{ij})(I(g),g) = \sum_k X_{ik}(I(g)) X_{kj}(g) = \sum_k \overline{X_{ki}(g)} X_{kj}(g) = \delta_{i,j}
\end{equation}
for all $g \in K$. In particular, $F(I(g) \cdot g) = F(I(h) \cdot h)$ for all $F \in \cA$ and all $g,h \in K$. So, the same holds for all $F \in B = C(K)$ and we find an element $e \in K$ such that $I(g) \cdot g = e$ for all $g \in K$. Then \eqref{eq.making-inverses} says that $X_{ij}(e) = \delta_{i,j}$ for all unitary corepresentations $X \in A \ot M_n(\C)$.

Evaluating $\Delta(X_{ij}) = \sum_k X_{ik} \ot X_{kj}$ in $(g,e)$, resp.\ $(e,g)$, we find that $X_{ij}(g \cdot e) = X_{ij}(g) = X_{ij}(e \cdot g)$ for all unitary corepresentations $X \in A \ot M_n(\C)$. So, $g \cdot e = g = e \cdot g$ for all $g \in K$. We already proved that $I(g) \cdot g = e$. Since by (vii), $S^2 = \id$, also $I(I(g)) = g$ for all $g \in K$. So, $K$ is a compact group with unit $e$ and inverse $g^{-1} := I(g)$. By uniqueness of the Haar state, the identification $C(K) = B$ extends to an identification $(L^\infty(K),\Delta_K) = (A,\Delta)$.

(x) If $(A,\Delta) \cong (L(G),\Delta_G)$, it is trivial that $\si \circ \Delta = \Delta$. Conversely, assume that $\si \circ \Delta = \Delta$. Define $G \subset \cU(A)$ as the subgroup of unitaries $a \in A$ satisfying $\Delta(a) = a \ot a$. By definition, $G$ consists of the one dimensional unitary corepresentations of $(A,\Delta)$. When $a \in G$ and $a \neq 1$, the one dimensional, irreducible corepresentations $a$ and $1$ are inequivalent. So by (vi), we have that $\vphi(a) = 0$ for all $a \in G \setminus \{1\}$. We thus get a unique normal trace preserving $*$-homomorphism $\pi : L(G) \to A$ satisfying $\pi(u_a) = a$ for all $a \in G$. By construction, $\Delta \circ \pi = (\pi \ot \pi) \circ \Delta_G$.

It remains to prove that $\pi$ is surjective. Since the $*$-algebra $\cA \subset A$ defined in (v) is dense, it suffices to fix an irreducible unitary corepresentation $X \in A \ot M_n(\C)$ and prove that $X$ is one dimensional. Define the linear map $\theta : A_* \to M_n(\C) : \theta(\om) := (\om \ot \id)(X)$. Since $X$ is a corepresentation, we get that $\theta(\om) \theta(\om') = \theta(\om \ast \om')$, where $\om \ast \om' = (\om \ot \om') \circ \Delta$. Since $\Delta$ is co-commutative, the image of $\theta$ is a commutative subalgebra of $M_n(\C)$. Take $\om \in A_*$. It follows that $X (1 \ot \theta(\om)) = (1 \ot \theta(\om))X$, because applying $\om' \ot \id$ to the left or the right hand side gives the same result for every $\om' \in A_*$. By irreducibility of $X$, we conclude that $\theta(\om) \in \C 1$ for all $\om \in A_*$. It follows that $X (1 \ot T) = (1 \ot T)X$ for all $T \in M_n(\C)$, because applying $\om \ot \id$ to the left or the right hand side gives the same result for every $\om \in A_*$. So, $M_n(\C) = \C 1$, i.e.\ $n = 1$.
\end{proof}

\subsection{\boldmath Coactions and unitary $2$-cocycles on compact quantum groups}

We now give an elementary proof for \cite[Corollary 5.2]{DeC10} saying that a coaction of a Kac type compact quantum group $(A,\Delta)$ on a type I factor $B(K)$ automatically preserves the trace.

\begin{proposition}[{Corollary 5.2 in \cite{DeC10}}]\label{prop.invariant-trace}
Let $(A,\Delta)$ be a Kac type compact quantum group, $K$ a Hilbert space and $\be : B(K) \to A \ovt B(K)$ a coaction.
\begin{enumlist}
\item\label{action-BK.1} The trace $\Tr$ is invariant: $\Tr((\om \ot \id)\be(b)) = \om(1) \, \Tr(b)$ for all $b \in B(K)^+$ and $\om \in A_*^+$.
\item\label{action-BK.2} The fixed point algebra $B(K)^\be$ is discrete.
\item\label{action-BK.3} If $\be$ is ergodic, then $K$ is finite dimensional.
\end{enumlist}
\end{proposition}
\begin{proof}
First assume that $\be : B(K) \to A \ovt B(K)$ is an ergodic coaction. We prove that $K$ is finite dimensional and that the trace $\Tr$ is invariant.

Let $\vphi$ be the Haar state of $(A,\Delta)$. By ergodicity, we can define the faithful normal invariant state $\psi$ on $B(K)$ such that $(\vphi \ot \id)\be(b) = \psi(b) 1$ for all $b \in B(K)$. Since $\psi$ is a faithful normal state on $B(K)$, we find a positive trace class operator $Q$ on $K$ such that $Q$ has trivial kernel and $\psi(b) = \Tr(b Q)$ for all $b \in B(K)$. Diagonalizing $Q$, we can write
$$\psi(b) = \sum_{k=1}^n \al_k \, \Tr(p_k b p_k) \quad\text{for all $b \in B(K)$,}$$
where $n \in \{+\infty\} \cup \{1,2,3,\ldots\}$, the $\al_k$ are strictly decreasing, strictly positive real numbers and the $p_k$ are nonzero finite rank projections summing up to $1$. We prove that $n=1$.

Assume that $n \geq 2$. When $b \in p_k B(K) p_r$, we have that $\si_t^\psi(b) = (\al_k / \al_r)^{it} b$ for all $t \in \R$. Since the sequence $\al_k$ is strictly decreasing, we have that $\al_1 / \al_r \geq 1$ for all $r$. So the following statement holds: whenever $b \in B(K)$ and $0 < \al < 1$ satisfy $\si_t^\psi(b) = \al^{it} b$, we get that $p_1 b = 0$.

Fix $r \geq 2$ and $b \in p_r B(K) p_1$. Put $\al = \al_r / \al_1$, so that $0 < \al < 1$ and $\si_t^\psi(b) = \al^{it} b$. By Theorem \ref{cqg.8}, we have that $(\id \ot \si_t^\psi)\be(b) = \be(\si_t^\psi(b)) = \al^{it} \be(b)$. By the statement in the previous paragraph, we get that $(1 \ot p_1)\be(b) = 0$ for all $b \in p_r B(K) p_1$. So, $(1 \ot p_1)\be(b c^*) = 0$ for all $b,c \in p_r B(K) p_1$. Since $p_r$ belongs to the linear span of $p_r B(K) p_1 B(K) p_r$, we get that $(1 \ot p_1) \be(p_r) = 0$ for all $r \geq 2$. So, $\be(p_r) \leq 1 \ot (1-p_1)$ for all $r \geq 2$, meaning that $1 \ot p_1 \leq \be(p_1)$.

Since $\be(p_1) - 1 \ot p_1 \geq 0$ and $(\vphi \ot \psi)(\be(p_1) - 1 \ot p_1) = (\vphi \ot \psi)\be(p_1) - \psi(p_1) = 0$, we conclude that $\be(p_1) = 1 \ot p_1$. Since $\be$ is ergodic and $p_1$ is a nonzero projection, we get that $p_1 = 1$. This means that $n=1$, contradicting the assumption that $n \geq 2$.

So, we have proven that $n=1$. Since $p_1$ is a finite rank projection, $K$ is finite dimensional. Also, $\psi(b) = \al_1 \Tr(b)$ for all $b \in B(K)$. Since $\psi$ is an invariant state, also $\Tr$ is invariant.

Next, let $\be : B(K) \to A \ovt B(K)$ be any coaction. Then, $b \mapsto (\vphi \ot \id)\be(b)$ is a faithful normal conditional expectation of $B(K)$ onto $B(K)^\be$. So, the fixed point algebra $B(K)^\be$ is a discrete von Neumann algebra and we find a family of minimal projections $p_k \in B(K)^\be$ such that $\sum_k p_k = 1$. By construction, for every $k$, the restriction of $\be$ to $p_k B(K) p_k = B(p_k K)$ defines an ergodic coaction $\be_k$. By the first part of the proof, the trace $\Tr_k$ on $B(p_k K)$ is invariant under $\be_k$.

Take $b \in B(K)^+$ and $\om \in A_*^+$. Then,
\begin{align*}
\Tr((\om \ot \id)\be(b)) &= \sum_k \Tr(p_k (\om \ot \id)\be(b) p_k) = \sum_k \Tr_k((\om \ot \id)\be_k(p_k b p_k)) \\ &= \sum_k \om(1) \, \Tr(p_k b p_k) = \om(1) \, \Tr(b) \; .\qedhere
\end{align*}
\end{proof}

As explained in the beginning of this section, twisting the comultiplication by a unitary $2$-cocycle can provide nonisomorphic compact quantum group structures with the same underlying von Neumann algebra. We now give precise definitions and then prove a number of vanishing results for such unitary $2$-cocycles.

\begin{definition}\label{def.unitary-2-cocycle}
A \emph{unitary $2$-cocycle} on a compact quantum group $(A,\Delta)$ is a unitary $\Om \in \cU(A \ovt A)$ satisfying
$$(\Om \ot 1) (\Delta \ot \id)(\Om) = (1 \ot \Om)(\id \ot \Delta)(\Om) \; .$$
Two unitary $2$-cocycles $\Om$ and $\Om'$ are said to be \emph{cohomologous} if there exists a unitary $v \in \cU(A)$ such that $\Om = (v \ot v) \Om' \Delta(v^*)$. One says that $\Om$ is a \emph{coboundary} if $\Om$ is cohomologous to $1$.
A \emph{unitary $\Om$-corepresentation} on a Hilbert space $K$ is a unitary $X \in \cU(A \ovt B(K))$ satisfying $(\Om \ot 1)(\Delta \ot \id)(X) = X_{13} X_{23}$. One says that $X$ is \emph{irreducible} if every $a \in B(K)$ satisfying $(1 \ot a) X = X (1 \ot a)$ is a multiple of $1$.
\end{definition}

If $\Om$ is a unitary $2$-cocycle on $(A,\Delta)$, the map $\Delta_\Om : A \to A \ovt A : \Delta_\Om(a) = \Om \Delta(a) \Om^*$ is again co-associative. In our next result, we give an elementary proof for \cite[Proposition 5.1]{DeC10}, saying that if $(A,\Delta)$ is of Kac type, then $(A,\Delta_\Om)$ is again a compact quantum group of Kac type. The nontrivial point is to prove that $(A,\Delta_\Om)$ again admits an invariant state, which may fail without the Kac type assumption.

If $X \in A \ovt B(K)$ is a unitary $\Om$-corepresentation on $K$, then
$$\be : B(K) \to A \ovt B(K) : \be(b) = X^* (1 \ot b) X \quad\text{for all $b \in B(K)$,}$$
is a coaction. Note that $\be$ is ergodic if and only if $X$ is irreducible. In the next result, we also include a proof for \cite[Proposition 3.1.9 and Corollary 3.1.13]{DMN21}, saying that the converse holds: every coaction of a Kac type compact quantum group on a type I factor is implemented by an $\Om$-corepresentation for some unitary $2$-cocycle $\Om$. Our proof is very similar to the proof of \cite{DMN21}, but to keep this section self-contained and because we need the result in this paper, we decided to include it.

\begin{proposition}\label{prop.coc}
Let $(A,\Delta)$ be a Kac type compact quantum group.
\begin{enumlist}
\item\label{coc.1} (\cite[Proposition 3.1.9 and Corollary 3.1.13]{DMN21}) For every Hilbert space $K$ and every coaction $\be : B(K) \to A \ovt B(K)$, there exists a unitary $2$-cocycle $\Om$ on $(A,\Delta)$ and a unitary $\Om$-corepresentation $X$ on $K$ such that $\be(b) = X^* (1 \ot b) X$ for all $b \in B(K)$.

\item\label{coc.2} (\cite[Propositions 4.3 and 5.1]{DeC10}) For every unitary $2$-cocycle $\Om$ on $(A,\Delta)$, there exist irreducible, finite dimensional $\Om$-corepresentations.

\item\label{coc.3} (\cite[Proposition 5.1]{DeC10}) For every unitary $2$-cocycle $\Om$ on $(A,\Delta)$, the twisted comultiplication $\Delta_\Om$ turns $(A,\Delta_\Om)$ into a compact quantum group of Kac type.
\end{enumlist}
\end{proposition}

\begin{proof}
(i) By Proposition \ref{action-BK.2}, the fixed point algebra $B(K)^\be$ is discrete. Choose a minimal projection $p \in B(K)^\be$. The restriction of $\be$ to $p B(K) p = B(pK)$ defines an ergodic coaction on $B(pK)$. By Proposition \ref{action-BK.3}, $pK$ is finite dimensional, so that $B(pK) \cong M_n(\C)$ for some $n \geq 1$. Choose a system of matrix units $e_{ij}$ for $B(K)$ in such a way that the first $e_{ij}$ with $1 \leq i,j \leq n$ form a system of matrix units for $B(pK)$.

We have that $A \ot B(pK)$ is a finite von Neumann algebra and denote by $E_\cZ$ its center valued trace, i.e.\ the unique trace preserving conditional expectation $E_\cZ : A \ot B(pK) \to \cZ(A)$. Since the projections $\be(e_{ii})$, $i=1,\ldots,n$, are all equivalent via the partial isometries $\be(e_{ij})$ and since they sum up to $1 \ot p$, we get that $E_\cZ(\be(e_{11})) = n^{-1} 1 = E_\cZ(1 \ot e_{11})$. So, the projections $\be(e_{11})$ and $1 \ot e_{11}$ are equivalent. Choose a partial isometry $V \in A \ovt B(K)$ with $VV^* = 1 \ot e_{11}$ and $V^*V = \be(e_{11})$. Then, $X = \sum_i (1 \ot e_{i1})V \be(e_{1i})$ is a unitary in $A \ovt B(K)$ satisfying $\be(b) = X^* (1 \ot b) X$ for all $b \in B(K)$.

Since $\be$ is a coaction, it follows that
\begin{align*}
(\Delta \ot \id)(X)^* (1 \ot 1 \ot b) (\Delta \ot \id)(X) &= (\Delta \ot \id)\be(b) = (\id \ot \be)\be(b) \\ &= X^*_{23} X^*_{13} (1 \ot 1 \ot b) X_{13} X_{23}
\end{align*}
for all $b \in B(K)$. This means that $X_{13} X_{23} = (\Om \ot 1)(\Delta \ot \id)(X)$ for some unitary $\Om \in A \ovt A$. From this, it follows that $\Om$ is a unitary $2$-cocycle on $(A,\Delta)$, so that $X$ is an $\Om$-corepresentation.

(ii) Denote by $\si \in \Aut(A \ovt A)$ the flip map. Define $\Delta\op = \si \circ \Delta$ and $\Omtil = \si(\Om)$. Applying \ref{cqg.2} to the compact quantum group $(A,\Delta\op)$, we define the unitary corepresentation $\cW \in A \ovt B(H)$ of $(A,\Delta\op)$ by $\cW^* (1 \ot a \xi_0) = \Delta\op(a) (1 \ot \xi_0)$ for all $a \in A$. Then, $\cW^*$ is a unitary corepresentation of $(A,\Delta)$. Also, $\Delta\op(a) = \cW^* (1 \ot a) \cW$ for all $a \in A$.

We prove that $Z = \Omtil \cW^*$ is a unitary $\Om$-corepresentation of $(A,\Delta)$. Indeed,
\begin{align*}
(\Om \ot 1) (\Delta \ot \id)(Z) &= ((1 \ot \Om) (\id \ot \Delta)(\Om))_{312} \cW^*_{13} \cW^*_{23} = ((\Om \ot 1) (\Delta \ot \id)(\Om))_{312} \cW^*_{13} \cW^*_{23} \\
& = \Omtil_{13} ((\Delta\op \ot \id)(\Om) (\cW^* \ot 1))_{132} \cW^*_{23} = \Omtil_{13} ((\cW^* \ot 1) (1 \ot \Om))_{132} \cW^*_{23} \\
& = \Omtil_{13} \cW^*_{13} \Omtil_{23} \cW^*_{23} = Z_{13} Z_{23} \; .
\end{align*}
Then, $\be : B(H) \to A \ovt B(H) : \be(b) = Z^* (1 \ot b) Z$ defines a coaction of $(A,\Delta)$ on $B(H)$. By Proposition \ref{action-BK.2}, we can choose a minimal projection $p$ in the fixed point algebra $B(H)^\be$. The restriction of $\be$ to $p B(H) p$ is ergodic, so that by Proposition \ref{action-BK.3}, the projection $p \in B(H)$ has finite rank. Then, $X = (1 \ot p) Z = Z (1 \ot p)$ defines an irreducible, finite dimensional $\Om$-corepresentation.

(iii) Since $\Om$ is a unitary $2$-cocycle, the twisted comultiplication is co-associative. It thus suffices to prove that the tracial state $\vphi$ remains invariant under $\Delta_\Om$. Since $\vphi$ is tracial, we get that
$$(\vphi \ot \vphi)\Delta_\Om(a) = (\vphi \ot \vphi)(\Om \Delta(a) \Om^*) = (\vphi \ot \vphi)(\Delta(a)) = \vphi(a) \; ,$$
so that $(\vphi \ot \vphi)\Delta_\Om = \vphi$. By Lemma \ref{lem.Haar-1} below, it thus suffices to prove that $(A \ot 1)\Delta_\Om(A)$ and $\Delta_\Om(A) (1 \ot A)$ span weakly dense subspaces of $A \ovt A$.

By (ii), we can fix a finite dimensional unitary $\Om$-corepresentation $X \in A \ot M_n(\C)$. We denote by $\be : M_n(\C) \to A \ot M_n(\C) : \be(b) = X^* (1 \ot b) X$ the corresponding coaction. By Proposition \ref{action-BK.1}, the trace $\Tr$ on $M_n(\C)$ is invariant under $\be$. So, $(\id \ot \Tr)\be(e_{ij}) = \delta_{i,j} 1$, meaning that $\sum_{k=1}^n X^*_{ik} X_{jk} = \delta_{i,j} 1$. Defining the element $\Xbar \in A \ot M_n(\C)$ by ${\Xbar}_{ij} = X^*_{ij}$, this means that ${\Xbar} \, {\Xbar}^* = 1$. Since $A \ot M_n(\C)$ is a finite von Neumann algebra, it follows that $\Xbar$ is a unitary.

Since $X$ is an $\Om$-corepresentation, we have
$$\Om \Delta(X_{ij}) = \sum_{k=1}^n X_{ik} \ot X_{kj} \; .$$
Taking the adjoint, it follows that $(\Delta \ot \id)(\Xbar) (\Om^* \ot 1) = \Xbar_{13} \Xbar_{23}$.

So, whenever $Y \in A \ot M_m(\C)$ is a finite dimensional unitary corepresentation of $(A,\Delta)$, we can consider the unitary $Z = X_{12} Y_{13} \Xbar_{14} \in A \ot M_{nmn}(\C)$ and note that $(\Delta_\Om \ot \id)(Z) = Z_{13} Z_{23}$.

We also note that $\Xbar_{12} X_{13} \in A \ot M_{n^2}(\C)$ is a unitary corepresentation of $(A,\Delta)$. Therefore, the linear span of all the coefficients of the unitaries $Z$ constructed in the previous paragraph form a $*$-subalgebra $\cAtil \subset A$. By Theorem \ref{cqg.5}, the coefficients of the finite dimensional unitary corepresentations $Y$ of $(A,\Delta)$ span a dense $*$-subalgebra $\cA$ of $A$. Then also $\cAtil$ is dense in $A$.

Since $(\Delta_\Om \ot \id)(Z) = Z_{13} Z_{23}$ and $Z$ is unitary, we get that
$$\sum_{r=1}^n (Z^*_{ri} \ot 1) \Delta_\Om(Z_{rj}) = \sum_{r,k=1}^n Z^*_{ri} Z_{rk} \ot Z_{kj} = 1 \ot Z_{ij} \; .$$
It follows that $1 \ot \cAtil \subset \lspan (\cAtil \ot 1)\Delta_\Om(\cAtil)$ and thus $\cAtil \ot \cAtil \subset \lspan (\cAtil \ot 1)\Delta_\Om(\cAtil)$. By density of $\cAtil$ in $A$, we get that $(A \ot 1)\Delta_\Om(A)$ spans a weakly dense subspace of $A \ovt A$. Similarly, $\Delta_\Om(A) (1 \ot A)$ spans a weakly dense subspace of $A \ovt A$.
\end{proof}

\begin{corollary}\label{cor.vanishing-char}
Let $(A,\Delta)$ be a Kac type compact quantum group. Then the following are equivalent.
\begin{enumlist}
\item Every unitary $2$-cocycle on $(A,\Delta)$ is a coboundary.
\item For every $n \in \N$ and every coaction $\be : M_n(\C) \to A \ot M_n(\C)$, there exists a unitary corepresentation $X \in A \ot M_n(\C)$ such that $\be(b) = X^* (1 \ot b) X$ for all $b \in M_n(\C)$.
\item For every $n \in \N$ and every ergodic coaction $\be : M_n(\C) \to A \ot M_n(\C)$, there exists a unitary corepresentation $X \in A \ot M_n(\C)$ such that $\be(b) = X^* (1 \ot b) X$ for all $b \in M_n(\C)$.
\end{enumlist}
\end{corollary}
\begin{proof}
(i) $\Rightarrow$ (ii). By Proposition \ref{coc.1}, we can take a unitary $2$-cocycle $\Om$ on $(A,\Delta)$ and a unitary $\Om$-corepresentation $X \in A \ot M_n(\C)$ such that $\be(b) = X^* (1 \ot b) X$ for all $b \in M_n(\C)$. By the assumption in (i), we can take a unitary $v \in \cU(A)$ such that $\Om = (v \ot v) \Delta(v^*)$. Replacing $X$ by $(v \ot 1)X$, we have found the required unitary corepresentation.

(ii) $\Rightarrow$ (iii) is trivial.

(iii) $\Rightarrow$ (i). Take a unitary $2$-cocycle $\Om$ on $(A,\Delta)$. By Proposition \ref{coc.2}, we can choose $n \in \N$ and an irreducible unitary $\Om$-corepresentation $X \in A \ot M_n(\C)$. Then $\be : M_n(\C) \to A \ot M_n(\C) : \be(b) = X^* (1 \ot b) X$ is an ergodic coaction of $(A,\Delta)$ on $M_n(\C)$. By the assumption in (iii), we can choose a unitary corepresentation $Y \in A \ot M_n(\C)$ such that $\be(b) = Y^* (1 \ot b) Y$ for all $b \in M_n(\C)$. Then $XY^*$ commutes with $1 \ot M_n(\C)$ and we find a unitary $v \in \cU(A)$ such that $X = (v \ot 1)Y$. Since $X$ is an $\Om$-corepresentation and $Y$ is a corepresentation, it follows that $\Om \Delta(v) = v \ot v$, so that $\Om$ is a coboundary.
\end{proof}

The proof of the following lemma is essentially contained in \cite[Lemma 4.3]{MVD98}.

\begin{lemma}\label{lem.Haar-density}
Let $A$ be a von Neumann algebra and $\Delta : A \to A \ovt A$ a unital faithful normal $*$-homomorphism that is co-associative. Assume that $\psi$ is a faithful normal state on $A$ such that $(\psi \ot \psi)\Delta = \psi$.
\begin{enumlist}
\item\label{lem.Haar-1} If both $(A \ot 1) \Delta(A)$ and $\Delta(A)(1 \ot A)$ span weakly dense subspaces of $A \ovt A$, then $\psi$ is an invariant state and $(A,\Delta)$ is a compact quantum group with Haar state $\psi$.
\item\label{lem.Haar-2} If $(A,\Delta)$ is a compact quantum group, then $\psi$ is the Haar state.
\end{enumlist}
\end{lemma}
\begin{proof}
Fix $a \in A$. Define $b = (\id \ot \psi)\Delta(a)$. Since $(\psi \ot \psi)\Delta = \psi$, we get that $(\id \ot \psi)\Delta(b) = b$. From this it follows that
$$(\psi \ot \psi)((\Delta(b) - b \ot 1)^* (\Delta(b) - b \ot 1)) = (\psi \ot \psi)\Delta(b^* b) - \psi(b^* b) = 0 \; .$$
Since $\psi \ot \psi$ is faithful, we conclude that $\Delta(b) = b \ot 1$.

(i) By the introductory paragraph, for all $a \in A$,
$$(\id \ot \psi)\Delta(a) \ot 1 = (\id \ot \id \ot \psi)(\Delta \ot \id)\Delta(a) = (\id \ot (\id \ot \psi)\Delta)(\Delta(a)) \; .$$
This implies that
$$(\psi \ot \psi)((b \ot 1)\Delta(a)) \ot 1 = (\psi \ot (\id \ot \psi)\Delta)((b \ot 1)\Delta(a))$$
for all $a,b \in A$. Taking linear combinations, by continuity and the assumption that $(A \ot 1) \Delta(A)$ spans a weakly dense subspace of $A \ovt A$, we may replace $(b \ot 1)\Delta(a)$ by $1 \ot d$, for $d \in A$ arbitrary, and conclude that $(\id \ot \psi)\Delta(d) = \psi(d) 1$ for all $d \in A$.

We similarly obtain right invariance: $(\psi \ot \id)\Delta(d) = \psi(d)1$ for all $d \in A$. So, $(A,\Delta)$ is a compact quantum group with Haar state $\psi$.

(ii) By the introductory paragraph, whenever $b = (\id \ot \psi)\Delta(a)$, we have that $\Delta(b) = b \ot 1$. Let $\vphi$ be the Haar state of $(A,\Delta)$. Then, $\vphi(b) 1 = (\id \ot \vphi)\Delta(b) = b$. So, $b$ is a multiple of $1$ and thus, $b = \psi(b) 1$. Since $\psi(b) = (\psi \ot \psi)\Delta(a) = \psi(a)$, we have proven that $(\id \ot \psi)\Delta(a) = \psi(a) 1$ for all $a \in A$. Applying $\vphi$, it follows that $\vphi(a) = \psi(a)$ for all $a \in A$.
\end{proof}

\subsection{\boldmath Vanishing and nonvanishing of unitary $2$-cocycles}\label{sec.vanishing-and-nonvanishing}

In \cite[Proposition 7.3]{DeC10}, it was proven that every unitary $2$-cocycle on $(L(G),\Delta_G)$ is a coboundary if $G$ is a torsion free discrete group. We generalize this to crossed product compact quantum groups in Proposition \ref{prop.cocycle-reduce-to-core.one} below. So we first introduce this concept.

If $\Gamma \actson^\al (A,\Delta_A)$ is an action of a discrete group $\Gamma$ by quantum group automorphisms of $(A,\Delta_A)$, the crossed product von Neumann algebra $B = A \rtimes_\al \Gamma$ carries a unique comultiplication
$$\Delta_B : B \to B \ovt B : \Delta_B(a u_g) = \Delta_A(a) (u_g \ot u_g) \quad\text{for all $a \in A$ and $g \in \Gamma$.}$$
It is easy to see that $\Delta_B$ is co-associative. The Haar state $\vphi_A$ of $(A,\Delta)$ has a canonical extension to a faithful normal state $\vphi_B$ on $B$ satisfying $\vphi_B(au_g) = 0$ for all $a \in A$ and $g \in \Gamma \setminus \{e\}$. It is easy to check that $\vphi_B$ is an invariant state on $(B,\Delta_B)$, which thus is a compact quantum group. Note that $(B,\Delta_B)$ is of Kac type if and only if $(A,\Delta_A)$ is of Kac type.

\begin{proposition}\label{prop.cocycle-reduce-to-core}
Let $\Gamma$ be a discrete group and $\Gamma \actson^\al (A,\Delta)$ an action of $\Gamma$ on the Kac type compact quantum group $(A,\Delta)$. Denote the crossed product as $(B,\Delta)$.
\begin{enumlist}
\item\label{prop.cocycle-reduce-to-core.one} For every unitary $2$-cocycle $\Om$ on $(B,\Delta)$, there exists a finite subgroup $\Lambda < \Gamma$ such that $\Om$ is cohomologous to a $2$-cocycle $\Om_0$ that belongs to $(A \rtimes \Lambda) \ovt (A \rtimes \Lambda)$.
\item\label{prop.cocycle-reduce-to-core.two} Let $\Om_1$ and $\Om_2$ be unitary $2$-cocycles for $(A,\Delta)$. Then $\Om_1$ is cohomologous with $\Om_2$ as unitary $2$-cocycles for $(B,\Delta)$ if and only if there exists a $g \in \Gamma$ such that $\Om_1$ is cohomologous with $(\al_g \ot \al_g)(\Om_2)$ as unitary $2$-cocycles for $(A,\Delta)$.
\item\label{prop.cocycle-reduce-to-core.three} Nontrivial unitary $2$-cocycles for $(A,\Delta)$, resp.\ $(L(\Gamma),\Delta)$, remain nontrivial as unitary $2$-cocycles for $(B,\Delta)$.
\end{enumlist}
\end{proposition}
\begin{proof}
(i) Let $\Om$ be a unitary $2$-cocycle on $(B,\Delta)$. By Proposition \ref{coc.2}, we can choose an irreducible unitary $\Om$-corepresentation $X \in B \ot M_n(\C)$ and define the associated ergodic coaction $\be : M_n(\C) \to B \ot M_n(\C) : \be(b) = X^*(1 \ot b) X$.

Denote by $K$ the Hilbert space $M_n(\C)$ on which the scalar product is given by the trace $\Tr$. By Proposition \ref{action-BK.1}, the trace $\Tr$ is invariant under $\be$. By Theorem \ref{cqg.1}, we can define the unitary corepresentation $Z \in B \ot B(K)$ of $(B,\Delta)$ by $Z^* (1 \ot b) = \be(b)$ for all $b \in K = M_n(\C)$.

Below we analyze which irreducible unitary corepresentations of $(B,\Delta)$ are unitarily conjugate to a sub-corepresentation of $Z$. We actually introduce there an ad hoc approach to the theory of spectral subspaces due to \cite{Boc92}. Before doing that, we fix a specific complete set of inequivalent irreducible unitary corepresentations of $(B,\Delta)$.

Fix a complete set $\Irr(A,\Delta)$ of inequivalent irreducible unitary corepresentations of $(A,\Delta)$. For every $X \in \Irr(A,\Delta)$ and $g \in \Gamma$, we have that $X(u_g \ot 1)$ is an irreducible unitary corepresentation of $(B,\Delta)$. By Theorem \ref{cqg.6}, the coefficients of $X \in \Irr(A,\Delta)$ form an orthogonal basis of the GNS-Hilbert space of $(A,\vphi)$. So, the coefficients of $X(u_g \ot 1)$, $X \in \Irr(A,\Delta)$, $g \in \Gamma$, form an orthogonal basis of the GNS-Hilbert space of $(B,\vphi)$. By Theorem \ref{cqg.6}, $\Irr(B,\Delta) := \{X(u_g \ot 1) \mid X \in \Irr(A,\Delta), g \in \Gamma\}$ must be a complete set of inequivalent irreducible unitary corepresentations of $(B,\Delta)$.

Define $\cI \subset \Irr(B,\Delta)$ as the set of $Y \in \Irr(B,\Delta)$ such that $\Ybar$ is equivalent to a sub-corepresen\-tation of $Z$. Since $K$ is finite dimensional, $\cI$ is a finite set. Note that if $Y \in \Irr(B,\Delta)$ is $k$-dimensional, then $v \in B(\C^k,K)$ is an intertwiner between $\Ybar$ and $Z$ if and only if $\sum_{i=1}^k v(e_i) \ot e_i$ belongs to
$$\cS_Y = \{S \in M_n(\C) \ot \C^k \mid (\be \ot \id)(S) = Y_{13} S_{23} \}\; .$$
It follows that $Y \in \cI$ if and only if $\cS_Y \neq \{0\}$. If $S \in \cS_Y$, we have that $\be(S^*S) = 1 \ot S^* S$, so that $S^* S$ is a multiple of $1$ by ergodicity of $M_n(\C)$. We conclude that $Y \in \cI$ if and only if $\cS_Y$ contains an isometry.

If $\sum_i S_i \ot e_i$ belongs to $\cS_Y$, we have that $\sum_i S_i^* \ot e_i$ belongs to $\cS_{\Ybar}$. So, $\overline{\cI} = \cI$.

If $Y,Y' \in \cI$, of dimensions $k,k'$, we can choose isometries $S \in \cS_Y$ and $S' \in \cS_{Y'}$. Then $S\dpr := S_{12} S'_{13} \in M_n(\C) \ot \C^k \ot \C^{k'}$ is also an isometry. Denote by $Y\dpr = Y_{12} Y'_{13}$ the tensor product of $Y$ and $Y'$. Then, $(\be \ot \id)(S\dpr) = Y\dpr_{13} S\dpr_{23}$. Since $Y\dpr$ can be written as a direct sum of irreducible unitary corepresentations and $S\dpr$ is nonzero, we can find $U \in \Irr(B,\Delta)$ of dimension $s$ and an intertwiner $w \in B(\C^k \ot \C^{k'},\C^s)$ from $Y\dpr$ to $U$ such that $T:=(1 \ot w) S\dpr \neq 0$. By construction, $T \in \cS_U$. We have thus proven that for all $Y,Y' \in \cI$, there exists a sub-corepresentation $U$ of the tensor product of $Y$ and $Y'$ with $U \in \cI$.

For every $g \in \Gamma$, define $\Irr_g \subset \Irr(B,\Delta)$ as the set of irreducible unitary corepresentations of the form $X (u_g \ot 1)$ with $X \in \Irr(A,\Delta)$. Since $\Gamma \actson (A,\Delta)$ is an action by quantum group automorphisms, it is easy to check that $\overline{\Irr_g} = \Irr_{g^{-1}}$ and that every irreducible sub-corepresentation of the tensor product of $Y \in \Irr_g$ and $Y' \in \Irr_h$ belongs to $\Irr_{gh}$.

Define the finite subset $\Lambda \subset \Gamma$ of $g \in \Gamma$ such that $\Irr_g \cap \cI \neq \emptyset$. Since $\overline{\cI} = \cI$ and $\overline{\Irr_g} = \Irr_{g^{-1}}$, we find that $\Lambda = \Lambda^{-1}$. If $g,h \in \Lambda$, we can choose $Y \in \Irr_g \cap \cI$ and $Y' \in \Irr_h \cap \cI$. We have seen above that there exists an irreducible sub-corepresentation $U$ of the tensor product of $Y$ and $Y'$ that belongs to $\cI$. Then also $U \in \Irr_{gh}$, so that $gh \in \Lambda$.

We have thus proven that $\Lambda$ is a finite subgroup of $\Gamma$. Define $B_0 = A \rtimes \Lambda$. Note that $(B_0,\Delta)$ is itself a compact quantum group of Kac type. By definition of $\Lambda$, all irreducible sub-corepresentations of $Z$ are unitary corepresentations of $(B_0,\Delta)$. This means that $Z \in B_0 \ot B(K)$, so that $\be(M_n(\C)) \subset B_0 \ot M_n(\C)$.

By Proposition \ref{coc.1}, there exists a unitary $2$-cocycle $\Om_0$ on $(B_0,\Delta)$ and a unitary $\Om_0$-corepresentation $X_0 \in B_0 \ot M_n(\C)$ such that $\be(b) = X_0^* (1 \ot b) X_0$ for all $b \in M_n(\C)$. Then $XX_0^*$ commutes with $1 \ot M_n(\C)$ and we find a unitary $v \in \cU(B)$ such that $X = (v \ot 1) X_0$. Since $X$ is an $\Om$-corepresentation, while $X_0$ is an $\Om_0$-corepresentation, it follows that $\Om = (v \ot v)\Om_0 \Delta(v^*)$, so that $\Om$ is cohomologous to $\Om_0$.

(ii) First assume that $\Om_1$ is cohomologous with $\Om_2$ as unitary $2$-cocycles for $(B,\Delta)$. Take a unitary $v \in B$ such that $\Om_1 \Delta(v) = (v \ot v) \Om_2$. Denote by $v = \sum_{g \in \Gamma} a_g u_g$, with $a_g \in A$, the Fourier decomposition of $v \in A \rtimes \Gamma$. Since $\Om_1 \Delta(v) = (v \ot v) \Om_2$, we get that
$$\sum_{g \in \Gamma} \Om_1 \Delta(a_g) (u_g \ot u_g) = \sum_{g \in \Gamma} (v \ot a_g u_g) \Om_2 = \sum_{g \in \Gamma} (v \ot a_g) (\id \ot \al_g)(\Om_2) (1 \ot u_g) \; .$$
Consider the Fourier decomposition in the second component of the tensor product, it follows that
$$\Om_1 \Delta(a_g) (u_g \ot 1) = (v \ot a_g) (\id \ot \al_g)(\Om_2) \quad\text{for all $g \in \Gamma$.}$$
This means that
$$\Om_1 \Delta(a_g) (\al_g \ot \al_g)(\Om_2^*) = vu_g^* \ot a_g \quad\text{for all $g \in \Gamma$.}$$
Fix a $g \in \Gamma$ such that $a_g \neq 0$. By the previous equality, $v u_g^* \ot a_g \in A \ovt A$. Since $a_g \neq 0$, this means that $v u_g^* \in A$. We denote this unitary as $v_0 \in \cU(A)$. We have proven that $v = v_0 u_g$. Since $\Om_1  = (v \ot v) \Om_2 \Delta(v^*)$, this implies that
\begin{equation}\label{eq.this-is-what-we-find}
\Om_1 = (v_0 \ot v_0) (\al_g \ot \al_g)(\Om_2) \Delta(v_0^*) \; .
\end{equation}
So, $\Om_1$ is cohomologous with $(\al_g \ot \al_g)(\Om_2)$ as unitary $2$-cocycles for $(A,\Delta)$.

Conversely, if $\Om_1$ is cohomologous with $(\al_g \ot \al_g)(\Om_2)$ as unitary $2$-cocycles for $(A,\Delta)$, we can take a unitary $v_0 \in A$ such that \eqref{eq.this-is-what-we-find} holds. Defining $v = v_0 u_g$, we have found a unitary $v \in B$ such that $\Om_1 = (v \ot v) \Om_2 \Delta(v^*)$, so that $\Om_1$ and $\Om_2$ are cohomologous as unitary $2$-cocycles for $(B,\Delta)$.

(iii) It follows from (ii) that every nontrivial unitary $2$-cocycle for $(A,\Delta)$ also is a nontrivial unitary $2$-cocycle for $(B,\Delta)$. To conclude the proof of (iii), assume that $\Om$ is a unitary $2$-cocycle for $(L(\Gamma),\Delta)$ that is a coboundary as a unitary $2$-cocycle for $(B,\Delta)$. We have to prove that it is already a coboundary as a unitary $2$-cocycle for $(L(\Gamma),\Delta)$.

By Proposition \ref{coc.2}, we can choose a unitary $\Om$-corepresentation $X \in L(\Gamma) \ot M_n(\C)$. Denote by $\be : M_n(\C) \to L(\Gamma) \ot M_n(\C) : \be(b) = X^* (1 \ot b) X$ the associated coaction. Since $\Om$ is a coboundary as a unitary $2$-cocycle for $(B,\Delta)$, we can take a unitary $v \in B$ such that $\Om = (v^* \ot v^*)\Delta(v)$. Write $Y = (v \ot 1)X \in B \ot M_n(\C)$. Then $Y$ is a unitary corepresentation of $(B,\Delta)$ and $\be(b) = Y^* (1 \ot b) Y$ for all $b \in M_n(\C)$.

Denote by $\cA \subset A$ the $*$-algebra given by Theorem \ref{cqg.5}, spanned by the coefficients of the finite dimensional unitary corepresentations of $(A,\Delta)$. We similarly define $\cB \subset B$ and note that $\cB$ equals the algebraic crossed product $\cB = \cA \rtimes_{\text{\rm alg}} \Gamma$. Denote by $\eps : \cA \to \C$ the co-unit, i.e.\ the unique $*$-homomorphism satisfying $\eps(Z_{ij}) = \delta_{i,j}$ for every unitary corepresentation $Z \in A \ot M_k(\C)$ of $(A,\Delta)$. Then denote by $\psi : \cB \to \C[\Gamma]$ the unique $*$-homomorphism satisfying $\psi(a u_g) = \eps(a) u_g$ for all $a \in \cA$, $g \in \Gamma$. Note that $(\psi \ot \psi) \circ \Delta = \Delta \circ \psi$ on $\cB$. Since $Y \in \cB \ot M_n(\C)$, we get that $Z := (\psi \ot \id)(Y)$ is a well-defined unitary corepresentation of $(L(\Gamma),\Delta)$.

Since we can view $\be$ as a unitary corepresentation on the Hilbert space $M_n(\C)$, we have that $\be(M_n(\C)) \subset \C[\Gamma] \ot M_n(\C)$. Then $(\psi \ot \id)(\be(b)) = \be(b)$ for all $b \in M_n(\C)$ and applying $\psi \ot \id$ to the equality $\be(b) = Y^* (1 \ot b) Y$, we conclude that $\be(b) = Z^* (1 \ot b) Z$ for all $b \in M_n(\C)$. This means that $Z = (w \ot 1)X$, where $w \in L(\Gamma)$ is a unitary. Since $Z$ is a corepresentation, while $X$ is an $\Om$-corepresentation, it follows that $\Om = (w^* \ot w^*)\Delta(w)$, so that $\Om$ is a coboundary in $(L(\Gamma),\Delta)$.
\end{proof}

We next turn to unitary $2$-cocycles on direct products. If $(A_k,\Delta_k)$ is a family of compact quantum groups, with Haar states $\vphi_k$, we may define the von Neumann algebra $A$ with a faithful normal state $\vphi$ as the tensor product
$$(A,\vphi) = \ovt_{k} (A_k,\vphi_k) \; .$$
For every $k$, we have the canonical embedding $\pi_k : A_k \to A$ as the $k$'th tensor factor. There is a unique comultiplication $\Delta : A \to A \ovt A$ satisfying $\Delta \circ \pi_k = (\pi_k \ot \pi_k) \circ \Delta_k$ for all $k$. It is easy to check that $(A,\Delta)$ is again a compact quantum group with Haar state $\vphi$.

In the following proposition, we characterize when for such a tensor product quantum group $(A,\Delta)$ all unitary $2$-cocycles are a coboundary. It is possible to give a complete description of all unitary $2$-cocycles. In the context of Hopf algebras, this has been done in \cite[Theorem 3.5.5]{Sch02}.

\begin{proposition}\label{prop.cocycles-products}
Let $(A_k,\Delta_k)$, $k \in J$, be a finite or infinite family of Kac type compact quantum groups. Denote by $(A,\Delta)$ their direct product, as above. Then every unitary $2$-cocycle on $(A,\Delta)$ is a coboundary if and only if the following holds.
\begin{enumlist}
\item For every $k \in J$, every unitary $2$-cocycle on $(A_k,\Delta_k)$ is a coboundary.
\item If $k,l \in J$ are distinct and $Z \in \cU(A_k \ovt A_l)$ is a bicharacter, meaning that
$$(\Delta_k \ot \id)(Z) = Z_{13} Z_{23} \quad\text{and}\quad (\id \ot \Delta_l)(Z) = Z_{13} Z_{12} \; ,$$
then $Z=1$.
\end{enumlist}
\end{proposition}
\begin{proof}
First assume that conditions (i) and (ii) hold. Let $\be : M_n(\C) \to A \ot M_n(\C)$ be a coaction. By Corollary \ref{cor.vanishing-char}, it suffices to construct a unitary corepresentation $X \in A \ot M_n(\C)$ such that $\be(b) = X^* (1 \ot b) X$ for all $b \in M_n(\C)$.

Denote by $\vphi_k$ the Haar state on $(A_k,\Delta_k)$. Choose a complete set $\Irr_k$ of inequivalent irreducible unitary corepresentations of $(A_k,\Delta_k)$. By Theorem \ref{cqg.6}, the coefficients $X_{ij}$, $X \in \Irr_k$ form an orthogonal basis for the GNS Hilbert space of $(A_k,\vphi_k)$. Denote by $\cA_k \subset A_k$ the linear span of all $X_{ij}$, $X \in \Irr_k$. If $k_1,\ldots,k_s \in J$ are distinct and $X_t \in \Irr_{k_t}$ has dimension $n_t$, then
$$(\pi_{k_1} \ot \id)(X_1)_{12} \, (\pi_{k_2} \ot \id)(X_2)_{13} \, \cdots \, (\pi_{k_s} \ot \id)(X_s)_{1,s+1} \in A \ot M_{n_1\cdots n_s}(\C)$$
is an irreducible corepresentation of $(A,\Delta)$. By construction, the coefficients of these irreducible corepresentations form an orthogonal basis of the GNS Hilbert space of $(A,\Delta)$. By Theorem \ref{cqg.6}, these irreducible corepresentations of $(A,\Delta)$ form a complete set of irreducible unitary corepresentations of $(A,\Delta)$. We conclude that the dense $*$-subalgebra $\cA \subset A$ spanned by the coefficients of all finite dimensional unitary corepresentations of $(A,\Delta)$ is generated by $\pi_k(\cA_k)$, $k \in J$, meaning that $\cA$ is the algebraic tensor product of the $*$-algebras $\cA_k$.

As in the beginning of the proof of Proposition \ref{prop.cocycle-reduce-to-core}, we may view $\be$ as a unitary corepresentation of $(A,\Delta)$ of dimension $n^2$. This means that $\be(M_n(\C)) \subset \cA \ot M_n(\C)$.

For every $k \in J$, denote by $\eps_k : \cA_k \to \C$ the co-unit, i.e.\ the unique $*$-homomorphism satisfying $\eps_k(X_{ij}) = \delta_{i,j}$ for all $X \in \Irr_k$. Since $\cA$ is the algebraic tensor product of the $\cA_k$, we can uniquely define the $*$-homomorphisms $\theta_k : \cA \to \cA_k$ satisfying $\theta_k(\pi_s(a)) = \eps_s(a) 1$ for all $s \neq k$, $a \in \cA_s$, and $\theta_k(\pi_k(a)) = a$ for all $a \in \cA_k$.

Note that $\Delta_k \circ \theta_k = (\theta_k \ot \theta_k) \circ \Delta$ on $\cA$. So, $\be_k = (\theta_k \ot \id) \circ \be$ is a well-defined coaction of $(A_k,\Delta_k)$ on $M_n(\C)$. By condition (i) and Corollary \ref{cor.vanishing-char}, we find unitary corepresentations $X_k \in \cA_k \ot M_n(\C)$ such that $\be_k(b) = X_k^* (1 \ot b) X_k$ for all $b \in M_n(\C)$ and $k \in J$.

We prove that for distinct $k,l \in J$, the unitaries $(\pi_k \ot \id)(X_k)$ and $(\pi_l \ot \id)(X_l)$ in $A \ot M_n(\C)$ commute. Denote by $\si : \cA_l \otalg \cA_k \to \cA_k \otalg \cA_l$ the flip isomorphism. Note that $(\theta_k \ot \theta_l)\Delta(a) = \si((\theta_l \ot \theta_k)\Delta(a))$ for all $a \in \cA$. Therefore,
$$(\id \ot \be_l) \be_k = (\theta_k \ot \theta_l \ot \id)(\Delta \ot \id) \be = (\si \ot \id)(\theta_l \ot \theta_k \ot \id)(\Delta \ot \id)\be = (\si \ot \id)(\id \ot \be_k)\be_l \; .$$
This means that $(X_l^*)_{23} (X_k^*)_{13} (1 \ot 1 \ot b) (X_k)_{13} (X_l)_{23} = (X_k^*)_{13} (X_l^*)_{23} (1 \ot 1 \ot b) (X_l)_{23} (X_k)_{13}$ for all $b \in M_n(\C)$. There thus exists a unitary $Z \in \cA_k \otalg \cA_l$ such that
$$(Z \ot 1) (X_k)_{13} (X_l)_{23} = (X_l)_{23} (X_k)_{13} \; .$$
Applying $\Delta_k \ot \id \ot \id$, it follows that $(\Delta_k \ot \id)(Z) = Z_{13} Z_{23}$. Applying $\id \ot \Delta_l \ot \id$, it follows that $(\id \ot \Delta_l)(Z) = Z_{13} Z_{12}$. By condition (ii), $Z=1$. So, $(X_k)_{13}$ commutes with $(X_l)_{23}$, which is the same as saying that $(\pi_k \ot \id)(X_k)$ and $(\pi_l \ot \id)(X_l)$ commute in $A \ot M_n(\C)$.

Since $M_n(\C)$ is finite dimensional and $\be(M_n(\C)) \subset \cA \ot M_n(\C)$, we can take a finite subset $J_0 \subset J$ such that $\be(M_n(\C)) \subset \cA_{J_0} \ot M_n(\C)$, where $\cA_{J_0}$ is the algebraic tensor product of all $\cA_k$, $k \in J_0$. Since the unitaries $(\pi_k \ot \id)(X_k)$, $k \in J_0$, all commute, their product is a unitary corepresentation $X \in A \ot M_n(\C)$. By construction, $\be(b) = X^* (1 \ot b) X$ for all $b \in M_n(\C)$.

Conversely assume that every unitary $2$-cocycle on $(A,\Delta)$ is a coboundary. To prove that (i) holds, by Corollary \ref{cor.vanishing-char}, it suffices to show that for every $k \in J$, every coaction $\be_k : M_n(\C) \to A_k \ot M_n(\C)$ is implemented by a unitary corepresentation of $(A_k,\Delta_k)$. As above, note that $\be_k(M_n(\C)) \subset \cA_k \ot M_n(\C)$. Viewing $\be_k$ as a coaction of $(A,\Delta)$, by Corollary \ref{cor.vanishing-char}, we find a unitary corepresentation $X \in \cA \ot M_n(\C)$ satisfying $\be_k(b) = X^* (1 \ot b) X$ for all $b \in M_n(\C)$. With the notation of the first part of the proof, $X_k := (\theta_k \ot \id)(X)$ is a unitary corepresentation for $(A_k,\Delta_k)$ and $\be_k(b) = X_k^* (1 \ot b) X_k$ for all $b \in M_n(\C)$.

Finally, take an element $Z$ as in condition (ii). By Lemma \ref{lem.bichar} below, we find $n \in \N$ and unitary corepresentations $X_k \in \cA_k \ot M_n(\C)$ and $X_l \in \cA_l \ot M_n(\C)$ of $(A_k,\Delta_k)$, resp.\ $(A_l,\Delta_l)$, such that
\begin{equation}\label{eq.commute-up-to-Z}
(Z \ot 1) (X_k)_{13} (X_l)_{23} = (X_l)_{23} (X_k)_{13} \; .
\end{equation}
Define the coactions $\be_k(b) = X_k^* (1 \ot b) X_k$ and $\be_l(b) = X_l^* (1 \ot b) X_l$. Denoting by $\sigma$ the flip isomorphism, it follows from \eqref{eq.commute-up-to-Z} that $(\id \ot \be_l)\be_k = (\si \ot \id)(\id \ot \be_k)\be_l$. Using the natural embedding $A_k \ovt A_l \to A$, we may thus view $(\id \ot \be_l)\be_k$ as a coaction of $(A,\Delta)$ on $M_n(\C)$.

By Corollary \ref{cor.vanishing-char}, we find a unitary corepresentation $X \in \cA \ot M_n(\C)$ of $(A,\Delta)$ such that $\be(b) = X^*(1 \ot b) X$ for all $b \in M_n(\C)$. Define the unitary corepresentation $X'_k := (\theta_k \ot \id)(X)$ of $(A_k,\Delta_k)$. Similarly define $X'_l$. Applying $\theta_k \ot \theta_l \ot \id$, resp.\ $\theta_l \ot \theta_k \ot \id$, to the corepresentation property of $X$, we get that $(X'_k)_{13}$ commutes with $(X'_l)_{23}$.

Since both $X_k$ and $X'_k$ are unitary corepresentations that implement the same coaction $\be_k$, we can take a unitary $v_k \in \cA_k$ satisfying $X_k = (v_k \ot 1)X'_k$. We similarly find a unitary $v_l \in \cA_l$ satisfying $X_l = (v_l \ot 1) X'_l$. Since $(X'_k)_{13}$ commutes with $(X'_l)_{23}$, it follows that $(X_k)_{13}$ commutes with $(X_l)_{23}$. By \eqref{eq.commute-up-to-Z}, it follows that $Z=1$.
\end{proof}

\begin{lemma}\label{lem.bichar}
Let $(A_1,\Delta_1)$ and $(A_2,\Delta_2)$ be compact quantum groups of Kac type. Let $Z \in \cU(A_1 \ovt A_2)$ be a bicharacter:
$$(\Delta_1 \ot \id)(Z) = Z_{13} Z_{23} \quad\text{and}\quad (\id \ot \Delta_2)(Z) = Z_{13} Z_{12} \; .$$
There then exist $n \in \N$ and unitary corepresentations $X_i \in A_i \ot M_n(\C)$ of $(A_i,\Delta_i)$ for $i = 1,2$, such that
$$(Z \ot 1) (X_1)_{13} (X_2)_{23} = (X_2)_{23} (X_1)_{13} \; .$$
In particular, if $\cA_i \subset A_i$ are the $*$-algebras spanned by the coefficients of the finite dimensional unitary corepresentations, we have that $Z \in \cA_1 \otalg \cA_2$.
\end{lemma}
\begin{proof}
We denote by $(H_i,\xi_i)$ the GNS construction for $(A_i,\vphi_i)$, where $\vphi_i$ is the Haar state on $(A_i,\Delta_i)$. Applying Theorem \ref{cqg.2} to $(A_i,\Delta_i\op)$, we can define unitaries $\cW_i \in A_i \ovt B(H_i)$ such that $\cW_i^* (1 \ot a \xi_i) = \Delta_i\op(a) (1 \ot \xi_i)$ for all $a \in A_i$. Then $\cW_i^*$ is a unitary corepresentation of $(A_i,\Delta_i)$ on $H_i$. Also, $\Delta_i\op(a) = \cW_i^*(1 \ot a)\cW_i$ for all $a \in A_i$.

Note that in $A_2 \ovt A_1 \ovt B(H_2)$, we have the equalities
$$Z_{23} \cW^*_{2,13} Z_{23}^* = Z_{23} ((\id \ot \Delta_2\op)(Z^*))_{213} \cW^*_{2,13}  = Z^*_{21} \cW^*_{2,13} \; .$$
Since $\cW^*_{2,13}$ is a unitary corepresentation of $(A_2,\Delta_2)$ on $H_1 \ot H_2$, it follows that also $Y_2 := Z^*_{21} \cW^*_{2,13} \in A_2 \ovt B(H_1 \ot H_2)$ is a unitary corepresentation of $(A_2,\Delta_2)$ on $H_1 \ot H_2$.

Next, $Y_1 := \cW^*_{1,12}$ is a unitary corepresentation of $(A_1,\Delta_1)$ on $H_1 \ot H_2$. Note that
$$Z_{12} \, \cW^*_{1,13} \, Z^*_{32} \cW^*_{2,24} = Z_{12} \, ((\Delta_1\op \ot \id)(Z^*))_{132} \, \cW^*_{1,13} \, \cW^*_{2,24} = Z^*_{32} \, \cW^*_{1,13} \, \cW^*_{2,24} = Z^*_{32} \cW^*_{2,24} \, \cW^*_{1,13} \; . $$
This says that the unitary corepresentations $Y_1$ and $Y_2$ commute up to $Z$. Therefore, the coactions
$$\be_i : B(H_1 \ot H_2) \to A_i \ovt B(H_1 \ot H_2) : \be_i(b) = Y_i^* (1 \ot b) Y_i$$
satisfy $(\id \ot \be_1)\be_2 = (\si \ot \id)(\id \ot \be_2)\be_1$, where $\si$ denotes the flip isomorphism.

Denote by $(A,\Delta)$ the tensor product of $(A_1,\Delta_1)$ and $(A_2,\Delta_2)$. Then $\be := (\id \ot \be_2)\be_1$ defines a coaction of $(A,\Delta)$ on $B(H_1 \ot H_2)$. By Proposition \ref{action-BK.2}, we can choose a minimal projection $p \in B(H_1 \ot H_2)^\be$ and restrict $\be$ to an ergodic coaction on $p B(H_1 \ot H_2) p = B(p(H_1 \ot H_2))$. By Proposition \ref{action-BK.3}, $p (H_1 \ot H_2)$ is finite dimensional.

We claim that $\be_i(p) = 1 \ot p$ for all $i = 1,2$. Since $\be = (\id \ot \be_2)\be_1$ and $\be(p) = 1 \ot 1 \ot p$, the element $a := (\vphi_1 \ot \id)\be_1(p)$ satisfies $\be_2(a) = 1 \ot p$. Then,
$$1 \ot \be_2(p) = (\id \ot \be_2)\be_2(a) = (\Delta_2 \ot \id)\be_2(a) = (\Delta_2 \ot \id)(1 \ot p) = 1 \ot 1 \ot p \; ,$$
so that $\be_2(p)= 1 \ot p$. By symmetry, also $\be_1(p) = 1 \ot p$. Defining $X_i = Y_i (1 \ot p)$, we have found the required finite dimensional unitary corepresentations of $(A_i,\Delta_i)$ that commute up to $Z$.
\end{proof}

For completeness, we include a short proof of the following result.

\begin{proposition}[{Theorem 3.3 in \cite{IPV10}}]\label{prop.on-LG-all-symmetric-coboundary}
Let $G$ be a discrete group. If $\Om$ is a unitary $2$-cocycle for $(L(G),\Delta_G)$ that is symmetric, in the sense that $\si(\Om) = \Om$, where $\si$ is the flip automorphism of $L(G) \ovt L(G)$, then $\Om$ is a coboundary.
\end{proposition}
\begin{proof}
By Proposition \ref{coc.2}, we can choose an irreducible finite dimensional unitary $\Om$-corepresentation $X \in L(G) \ot M_n(\C)$. This means that
\begin{equation}\label{eq.here-Om-corep}
X_{13} X_{23} = (\Om \ot \id) (\Delta_G \ot \id)(X) \; .
\end{equation}
The right hand side of \eqref{eq.here-Om-corep} is invariant under $\si \ot \id$. So also the left hand side is invariant under $\si \ot \id$, which means that $X_{13}$ commutes with $X_{23}$. So, $1 \ot (\mu \ot \id)(X)$ commutes with $X$ for every $\mu \in L(G)_*$.

Since $X$ is irreducible, it follows that $(\mu \ot \id)(X)$ is a multiple of $1$ for every $\mu \in L(G)_*$. This means that $X = v \ot 1$, where $v \in \cU(L(G))$. Since $X$ is irreducible, we get that $n=1$. Now \eqref{eq.here-Om-corep} says that $\Om = (v \ot v)\Delta_G(v^*)$, so that $\Om$ is a coboundary.
\end{proof}

The following result is mentioned without proof in the remarks after \cite[Corollary 7.4]{DeC10}. Since we need the result in our paper, we include a proof here.

\begin{lemma}\label{lem.2-cocycle-up-to-scalar}
Let $M$ be a von Neumann algebra and $\Delta: M \to M \ovt M$ a unital, normal, co-associative $*$-homomorphism. Let $\Om \in M \ovt M$ be a unitary and $\nu \in \T$ satisfying
\begin{equation}\label{eq.2-cocycle-up-to-constant}
(\Om \ot 1)(\Delta \ot \id)(\Om) = \nu (1 \ot \Om)(\id \ot \Delta)(\Om) \; .
\end{equation}
Then $\nu = 1$.
\end{lemma}
\begin{proof}
We write $\Delta^{(2)} = (\Delta \ot \id) \Delta = (\id \ot \Delta)\Delta$. Applying $\Delta \ot \id \ot \id$ to \eqref{eq.2-cocycle-up-to-constant} and multiplying at the left with $\Om \ot 1 \ot 1$, we get that
\begin{equation}\label{eq.interm}
((\Om \ot 1)(\Delta \ot \id)(\Om) \ot 1)(\Delta^{(2)} \ot \id)(\Om) = \nu (\Om \ot \Om)(\Delta \ot \Delta)(\Om) \; .
\end{equation}
Using four times \eqref{eq.2-cocycle-up-to-constant}, the left hand side of \eqref{eq.interm} equals
\begin{align*}
\nu (1 \ot \Om \ot 1)(\id \ot \Delta \ot \id) & \bigl((\Om \ot 1)(\Delta \ot \id)(\Om)\bigr) \\ &= \nu^2 (1 \ot \Om \ot 1) (1 \ot (\Delta \ot \id)(\Om))(\id \ot \Delta^{(2)})(\Om) \\
&= \nu^3 (1 \ot 1 \ot \Om)(\id \ot \id \ot \Delta)\bigl((1 \ot \Om)(\id \ot \Delta)(\Om)\bigr) \\
&= \nu^2 (1 \ot 1 \ot \Om)(\id \ot \id \ot \Delta)\bigl((\Om \ot 1)(\Delta \ot \id)(\Om)\bigr) \\
&= \nu^2 (\Om \ot \Om)(\Delta \ot \Delta)(\Om) \; .
\end{align*}
Comparing with the right hand side of \eqref{eq.interm}, we get that $\nu^2 = \nu$, so that $\nu=1$.
\end{proof}

\subsection{\boldmath Cohomological obstructions to quantum W$^*$-superrigidity}\label{sec.cohomological-obstructions-superrigidity}

By Proposition \ref{coc.3}, every $2$-cocycle twist of a Kac type compact quantum group is again a Kac type compact quantum group, with the same underlying von Neumann algebra. Often, this twisted quantum group is not isomorphic to the original quantum group. We prove the following precise result.

\begin{proposition}\label{prop.not-rigid}
Let $G$ be an icc group, $G_0 < G$ a subgroup and $\Om_0 \in L(G_0) \ovt L(G_0)$ a nontrivial unitary $2$-cocycle for $(L(G_0),\Delta_0)$. View $\Om_0$ as a unitary $2$-cocycle $\Om$ for $(L(G),\Delta)$. Then, $\Delta_\Om$ is not symmetric.

In particular, $\Om_0$ remains nontrivial as a $2$-cocycle on $(L(G),\Delta)$, we have that $(L(G),\Delta_\Om) \not\cong (L(G),\Delta)$ and $(L(G),\Delta_G)$ is not quantum W$^*$-superrigid.
\end{proposition}
\begin{proof}
Denote by $\si : L(G) \ovt L(G) \to L(G) \ovt L(G) : \si(a \ot b) = b \ot a$ the flip map. Assume that $\si \circ \Delta_\Om = \Delta_\Om$. We prove that $\Om_0$ is a coboundary as a $2$-cocycle for $(L(G_0),\Delta_0)$.

Define $X = \Om^* \si(\Om)$. Since $\si \circ \Delta = \Delta$, we find that $X$ commutes with $\Delta(L(G))$. So, $X$ commutes with all unitaries $u_g \ot u_g$, $g \in G$. Since $G$ is icc, we find $\mu \in \T$ such that $X = \mu 1$. So, $\si(\Om) = \mu \Om$. Since $\Om = \Om_0$, we also get that $\si(\Om_0) = \mu \Om_0$ as unitaries in $L(G_0) \ovt L(G_0)$. By \cite[Theorem 3.3]{IPV10} (see also Lemma \ref{lem.2-cocycle-up-to-scalar} and Proposition \ref{prop.on-LG-all-symmetric-coboundary}), we get that $\mu = 1$ and that $\Om_0$ is a coboundary as a $2$-cocycle for $(L(G_0),\Delta_0)$.
\end{proof}

\begin{corollary}\label{cor.not-rigid}
If $G$ is an icc group that admits a finite abelian subgroup $G_0 < G$ such that $H^2(\Ghat_0,\T) \neq 1$, then $G$ is not quantum W$^*$-superrigid: there exists a Kac type compact quantum group $(B,\Delta)$ such that $B \cong L(G)$, but $(B,\Delta) \not\cong (L(G),\Delta)$.
\end{corollary}

In particular, none of the W$^*$-superrigid groups of the form $G = (\Z/2\Z)^{(I)} \rtimes \Gamma$ as considered in \cite{IPV10,BV12,DV24a} remain W$^*$-superrigid in the larger category of discrete quantum groups, because $H^2(\Z/2\Z \times \Z/2\Z,\T) \neq 1$.

\begin{proof}
Take a $2$-cocycle $\om_0 \in Z^2(\Ghat_0,\T)$ that is not a coboundary. Under the isomorphism $L^\infty(\Ghat_0) = L(G_0)$, we may view $\om_0$ as a unitary $2$-cocycle $\Om_0$ for $(L(G_0),\Delta_0)$ that is not a coboundary as a $2$-cocycle for $(L(G_0),\Delta_0)$. The result then follows from Proposition \ref{prop.not-rigid}.
\end{proof}

For the formulation of the next proposition, recall that a trace preserving group action $\Gamma \actson (A,\vphi)$ is called weakly mixing if the diagonal action $\Gamma \actson A \ovt A$ is ergodic.

\begin{proposition}\label{prop.crossed-product-not-rigid}
Let $\Gamma$ be an icc group and $\Gamma \actson^\al (A,\Delta)$ an action of $\Gamma$ by quantum group automorphisms of the Kac type compact quantum group $(A,\Delta)$. Denote the crossed product as $(B,\Delta)$, as defined before Proposition \ref{prop.cocycle-reduce-to-core}. Assume that the action $\Gamma \actson^\al A$ is weakly mixing. Assume that for every automorphism $\theta$ of the von Neumann algebra $B = A \rtimes_\al \Gamma$, there exists a unitary $v \in B$, an automorphism $\delta \in \Aut \Gamma$ and a character $\om : \Gamma \to \T$ such that $v \theta(u_g) v^* = \om(g) u_{\delta(g)}$ for all $g \in \Gamma$.

If $(A,\Delta)$ or $(L(\Gamma),\Delta)$ admits a nontrivial unitary $2$-cocycle, there exists a unitary $2$-cocycle $\Om$ on $(B,\Delta)$ such that $(B,\Delta_\Om) \not\cong (B,\Delta)$. In particular, $(B,\Delta)$ is not quantum W$^*$-superrigid.
\end{proposition}
\begin{proof}
By Proposition \ref{prop.cocycle-reduce-to-core.three}, nontrivial unitary $2$-cocycles on $(A,\Delta)$ or $(L(\Gamma),\Delta)$ give rise to nontrivial unitary $2$-cocycles on $(B,\Delta)$. It thus suffices to prove the following: if $\Om$ is a unitary $2$-cocycle for $(B,\Delta)$ and $(B,\Delta_\Om) \cong (B,\Delta)$, then $\Om$ is a coboundary.

Assume that $\theta : B \to B$ is an isomorphism satisfying $(\theta \ot \theta) \circ \Delta = \Delta_\Om \circ \theta$. By our assumption, there exists a unitary $v \in B$, an automorphism $\delta \in \Aut \Gamma$ and a character $\om : \Gamma \to \T$ such that $v \theta(u_g) v^* = \om(g) u_{\delta(g)}$ for all $g \in \Gamma$. Replacing $\theta$ by $(\Ad v) \circ \theta$ and replacing $\Om$ by the cohomologous $2$-cocycle $(v \ot v)\Om \Delta(v^*)$, we may assume that $\theta(u_g) = \om(g) u_{\delta(g)}$ for all $g \in \Gamma$. Then,
\begin{align*}
\om(g) \, \Om(u_{\delta(g)} \ot u_{\delta(g)}) \Om^* &= \Om \Delta(\theta(u_g)) \Om^* = \Delta_\Om(\theta(u_g)) = (\theta \ot \theta)\Delta(u_g) \\
&= \theta(u_g) \ot \theta(u_g) = \om(g)^2 \, u_{\delta(g)} \ot u_{\delta(g)} \; ,
\end{align*}
for all $g \in \Gamma$. It follows that
$$(u_g \ot u_g) \Om^* (u_g^* \ot u_g^*) = \om(\delta^{-1}(g)) \Om^* \quad\text{for all $g \in \Gamma$.}$$
Since $\Gamma$ is icc and $\Gamma \actson^\al A$ is weakly mixing, it follows that $\Om = \nu \cdot 1$ for some $\nu \in \T$. But then $\Om$ is the coboundary of $\nu \cdot 1$.
\end{proof}

\subsection{Characters and translation automorphisms of compact quantum groups}\label{sec.quantum-group-like-aut}

The \emph{group like isomorphisms} $\pi : L(G) \to L(\Lambda)$ between group von Neumann algebras $L(G)$ and $L(\Lambda)$ are the isomorphisms of the form $\pi = \pi_\delta \circ \pi_\om$ where $\delta : G \to \Lambda$ is a group isomorphism and $\pi_\delta(u_g) = u_{\delta(g)}$, and where $\om : G \to \T$ is a character and $\pi_\om(u_g) = \om(g) u_g$.

When we consider more generally Kac type compact quantum groups $(A,\Delta_A)$ and $(B,\Delta_B)$, the part $\pi_\delta$ precisely corresponds to the quantum group isomorphisms, i.e.\ the isomorphisms $\pi_1 : A \to B$ satisfying $\Delta_B \circ \pi_1 = (\pi_1 \ot \pi_1) \circ \Delta_A$. The automorphisms $\pi_\om$ precisely correspond to the \emph{translation automorphisms} of $(A,\Delta_A)$ that we define in this section.

First note that characters on a group $G$ precisely correspond to unital $*$-homomorphisms $\C[G] \to \C$. When $\cA \subset A$ is the dense $*$-subalgebra defined in Theorem \ref{cqg.5}, we thus define $\Char(A,\Delta)$ as the set of unital $*$-homomorphisms $\om : \cA \to \C$. For every $\om \in \Char(A,\Delta)$, we define the \emph{left translation automorphism} $\lambda_\om : A \to A$, formally as $\lambda_\om = (\om \ot \id)\Delta_A$ and more precisely as the unique von Neumann algebra automorphism satisfying $\lambda_\om(X_{ij}) = \sum_k \om(X_{ik}) X_{kj}$ whenever $X \in A \ot M_n(\C)$ is a unitary corepresentation. Using the orthonormal basis of Theorem \ref{cqg.6}, it is easy to see that $\lambda_\om$ uniquely extends to a Haar state preserving automorphism of $A$.

\begin{definition}\label{def.quantum-group-like-aut}
Let $(A,\Delta_A)$ and $(B,\Delta_B)$ be Kac type compact quantum groups. We call a von Neumann algebra isomorphism $\pi : A \to B$ \emph{quantum group like} if $\pi$ is of the form $\pi_0 \circ \lambda_\om$, where $\pi_0 : (A,\Delta_A) \to (B,\Delta_B)$ is a quantum group isomorphism and $\lambda_\om \in \Aut A$ is the left translation automorphism given by a character $\om \in \Char(A,\Delta)$.
\end{definition}

Since $\Delta_A$ need not be symmetric, every $\om \in \Char(A,\Delta)$ similarly gives rise to a right translation automorphism $\rho_\om : A \to A$ defined by $\rho_\om(a) = (\id \ot \om)\Delta_A(a)$ for all $a \in \cA$. For the following reason, using right instead of left translation automorphisms in Definition \ref{def.quantum-group-like-aut} does not change the concept of a quantum group like isomorphism.

First note that $\Char(A,\Delta)$ is a group with product $\om \cdot \om' := (\om \ot \om') \circ \Delta$, inverse $\om^{-1} := \om \circ S$, where $S$ is the antipode given by Theorem \ref{cqg.7}, and identity element given by the co-unit $\eps : \cA \to \C$ satisfying $\eps(X_{ij}) = \delta_{i,j}$ for every finite dimensional unitary corepresentation $X \in A \ot M_n(\C)$. Since $(\om \ot \om^{-1})\Delta(a) = \eps(a)$ for all $a \in \cA$, we get that $\pi_\om := \lambda_\om \circ \rho_{\om^{-1}}$ is a quantum group automorphism, for every $\om \in \Char(A,\Delta)$. So, we can switch between left and right translation automorphisms by composing with a quantum group automorphism.

When $\al \in \Aut(A,\Delta)$ is a quantum group automorphism, we have $\al(\cA) = \cA$, so that for every $\om \in \Char(A,\Delta)$, also $\om \circ \al \in \Char(A,\Delta)$. We say that $\om$ is \emph{$\al$-invariant} if $\om \circ \al = \om$. Noting that $\lambda_\om \circ \al = \al \circ \lambda_{\om \circ \al}$, we get that $\om$ is $\al$-invariant if and only if $\lambda_\om$ commutes with $\al$.

\begin{remark}
In the commutative case, when $(A,\Delta_A) = (L^\infty(K),\Delta_K)$, the left translation automorphisms are precisely the automorphism of the form $F(\cdot) \mapsto F(k_1 \cdot)$ for some $k_1 \in K$, which justifies the terminology. This can be seen as follows. First, if $k_1 \in K$, then $\om : \cA \to \C : \om(X_{ij}) = X_{ij}(k_1)$ is a well-defined character and the corresponding $\lambda_\om$ is given by left translation by $k_1$. Conversely, if $\om \in \Char(L^\infty(K),\Delta_K)$, then $\lambda_\om$ is a normal automorphism of $L^\infty(K)$ satisfying $\lambda_\om(\cA) = \cA$. Taking the norm closure, we find that $\lambda_\om$ is an automorphism of the C$^*$-algebra $C(K)$ and thus of the form $\lambda_\om(F(\cdot)) = F(\theta(\cdot))$ where $\theta$ is a homeomorphism of $K$. Since $\Delta_K \circ \lambda_\om = (\lambda_\om \ot \id)\circ \Delta_K$, the homeomorphism $\theta$ commutes with all right translations. It must therefore be a left translation: there exists a $k_1 \in K$ such that $\theta(k) = k_1 k$ for all $k \in K$. Then also $\om(X_{ij}) = X_{ij}(k_1)$ for all finite dimensional unitary representations $X : K \to \cU(n)$.
\end{remark}

\section{Relative rigidity of compact quantum groups}\label{sec.relative-rigidity}

The co-induced left-right Bernoulli construction in Theorem \ref{thm.main} is a canonical construction with input data given by an action $\Gamma \actson^\be (A_0,\Delta_0)$ by quantum group automorphisms of a compact quantum group $(A_0,\Delta_0)$. Given the functoriality of the construction, we can only expect that the output is quantum W$^*$-superrigid, if the input already is: we need to recover the quantum group structure on $A_0$ from the von Neumann algebra $A_0$ and the extra knowledge that each of the von Neumann algebra automorphisms $\be_g$, $g \in \Gamma$, actually is a quantum group automorphism.

That brings us to the notion of rigidity \emph{relative} to a group of automorphisms. We define this concept in this section and then prove that it holds for several families of compact groups.

\subsection{Definition of relative rigidity}

\begin{definition}\label{def.relative-rigidity-cqg}
Let $(A,\Delta_A)$ be a Kac type compact quantum group with Haar state $\vphi_A$. Let $\cG < \Aut(A,\Delta_A)$ be a subgroup or, more generally, let $\cG \actson^\al (A,\Delta_A)$ be an action of a group $\cG$ by quantum group automorphisms.
\begin{enumlist}
\item\label{def.relative-rigidity-cqg.one} We say that $(A,\Delta_A)$ is \emph{strictly rigid relative to $\cG < \Aut(A,\Delta_A)$} if the following holds: if $(B,\Delta_B)$ is any Kac type compact quantum group with Haar state $\vphi_B$ and $\pi : A \to B$ is a state preserving von Neumann algebra isomorphism such that $\pi \circ \al \circ \pi^{-1}$ is a quantum group automorphism for all $\al \in \cG$, then $\pi$ can be written as $\pi = \pi_0 \circ \lambda_\om$, where $\pi_0 : (A,\Delta_A) \to (B,\Delta_B)$ is a quantum group isomorphism and $\lambda_\om \in \Aut A$ is the left translation automorphism given by a $\cG$-invariant character $\om$ (see Section \ref{sec.quantum-group-like-aut}).
\end{enumlist}
We also need the following weaker notion.
\begin{enumlist}[resume]
\item\label{def.relative-rigidity-cqg.two} We say that $(A,\Delta_A)$ is \emph{rigid relative to $\cG \actson^\al (A,\Delta_A)$} if the following holds: if $(B,\Delta_B)$ is any Kac type compact quantum group with its Haar state, $\cG \actson^\be (B,\Delta_B)$ is an action by quantum group automorphisms and $\pi : A \to B$ is a state preserving von Neumann algebra isomorphism such that $\be_g \circ \pi = \pi \circ \al_g$ for all $g \in \cG$, there exists a quantum group isomorphism $\pi_0 : (A,\Delta_A) \to (B,\Delta_B)$ and a group automorphism $\zeta : \cG \to \cG$ such that $\be_{\zeta(g)} \circ \pi_0 = \pi_0 \circ \al_g$ for all $g \in  \cG$.
\end{enumlist}
\end{definition}

One could of course define a notion of strict rigidity relative to $\cG \actson^\al (A,\Delta_A)$, but this would anyway only depend on the image $\al(\cG) < \Aut(A,\Delta_A)$. While the notion of rigidity relative to an action may depend on the precise choice of action $\cG \actson^\al (A,\Delta_A)$, for strict rigidity the following natural property holds.

\begin{lemma}
Let $(A,\Delta_A)$ be a Kac type compact quantum group that is strictly rigid relative to a subgroup $\cG < \Aut(A,\Delta_A)$. If $\cG < \cG' < \Aut(A,\Delta_A)$ is a larger subgroup, $(A,\Delta_A)$ is also strictly rigid relative to $\cG'$.
\end{lemma}
\begin{proof}
Take a Kac type compact quantum group $(B,\Delta_B)$, equip $A$ and $B$ with the respective Haar states, and assume that $\pi : A \to B$ is a state preserving von Neumann algebra isomorphism such that $\pi \circ \al \circ \pi^{-1} \in \Aut(B,\Delta_B)$ for all $\al \in \cG'$. Since $(A,\Delta_A)$ is strictly rigid relative to $\cG$, we can write $\pi = \pi_0 \circ \lambda_\om$, where $\pi_0 : (A,\Delta_A) \to (B,\Delta_B)$ is a quantum group isomorphism and $\om \in \Char(A,\Delta)$. Since $\pi_0$ is a quantum group isomorphism, we get that $\lambda_\om \circ \al \circ \lambda_\om^{-1} \in \Aut(A,\Delta_A)$ for every $\al \in \cG'$. Since $\lambda_\om \circ \al \circ \lambda_\om^{-1} = \al \circ \lambda_{(\om \circ \al) \om^{-1}}$, also $\lambda_{(\om \circ \al) \om^{-1}} \in \Aut(A,\Delta_A)$ for every $\al \in \cG'$. To conclude the proof of the lemma, it thus suffices to prove the following statement: if $\mu \in \Char(A,\Delta_A)$ and $\lambda_\mu \in \Aut(A,\Delta_A)$, then $\mu = \eps$.

By the definition of $\lambda_\mu$, we get that $\Delta_A \circ \lambda_\mu = (\lambda_\mu \ot \id) \circ \Delta_A$. If $\lambda_\mu \in \Aut(A,\Delta_A)$, we also have that $\Delta_A \circ \lambda_\mu = (\lambda_\mu \ot \lambda_\mu)\circ \Delta_A$. It follows that $(\id \ot \lambda_\mu) \circ \Delta_A = \Delta_A$. This implies that $\mu = \eps$.
\end{proof}

\begin{remark}\label{rem.strictly-rigid-no-invariant-characters}
Let $(A,\Delta_A)$ be a Kac type compact quantum group that is strictly rigid relative to $\cG < \Aut(A,\Delta_A)$. If the co-unit $\eps$ is the only $\cG$-invariant character on $(A,\Delta_A)$, then every von Neumann algebra isomorphism $\pi$ as in Definition \ref{def.relative-rigidity-cqg.one} is automatically a quantum group isomorphism.
\end{remark}

\subsection{Examples of relative rigidity: the co-commutative case}\label{sec.examples-relative-rigidity-cocommutative}

We first prove rigidity of two very natural families of compact quantum groups w.r.t.\ natural groups of automorphisms: in Theorem \ref{thm.rigid-Kn}, we prove this for $\SL_n(\Z) \actson K_0^n$ whenever $n \geq 3$ and $K_0$ is a connected compact abelian group, e.g.\ $K_0= \T$, while in Theorem \ref{thm.rigid-LG-icc}, we prove this for $(L(G),\Delta_G)$ relative to the inner automorphisms $(\Ad u_g)_{g \in G}$ whenever $G$ is an icc group.

In both cases, these compact quantum groups are \emph{co-commutative}, meaning that $\si \circ \Delta = \Delta$, where $\si$ denotes the flip automorphism. In Theorem \ref{thm.generic-theorem} and Examples \ref{ex.main}, this will then provide discrete groups $G$ that are quantum W$^*$-superrigid, i.e.\ such that $(L(G),\Delta_G)$ is quantum W$^*$-superrigid in the sense of Definition \ref{def.quantum-Wstar-superrigid}.

In order to give examples of genuine quantum groups $(M,\Delta)$ that are quantum W$^*$-superrigid, we need examples of relative rigidity for \emph{noncommutative} compact groups $K$. We give such examples in the next Section \ref{sec.examples-relative-rigidity-noncommutative}.

When $\cT$ is a compact second countable group, we denote by $\Autgr(\cT)$ the group of continuous group automorphisms of $\cT$ and we denote by $\Autpmp(\cT)$ the group of Haar measure preserving automorphisms of the standard probability space $\cT$, where two such automorphisms are identified when they are equal a.e. Then $\Autpmp(\cT)$ is a Polish group and $\Autgr(\cT) < \Autpmp(\cT)$ is a closed subgroup.

\begin{theorem}\label{thm.rigid-Kn}
Let $K$ be a second countable connected compact abelian group and $n \geq 3$ an integer. Equip $K^n$ with the Haar measure. Writing the group operation in $K$ multiplicatively, view $\SL(n,\Z)$ as the subgroup of $\Autgr(K^n)$ defined by
\begin{equation}\label{eq.action-SLn}
(A \cdot a)_i = \prod_{j=1}^n a_j^{A_{ij}} \quad\text{for all $A \in \SL(n,\Z)$ and $a \in K^n$.}
\end{equation}
\begin{enumlist}
\item\label{thm.rigid-Kn.one} The compact group $K^n$ is strictly rigid relative to $\SL(n,\Z) < \Autgr(K^n)$.
\item\label{thm.rigid-Kn.two} If $\pi \in \Autpmp(K^n)$ commutes with $\SL(n,\Z)$, there exists a $\delta \in \Autgr(K)$ such that $\pi = \delta \times \cdots \times \delta$ a.e.
\end{enumlist}
\end{theorem}
\begin{proof}
(i) Write $\cG = \SL(n,\Z)$. For $A \in \cG$, we denote the automorphism $a \mapsto A \cdot a$ on $K^n$ by $\al_A$. Since $e$ is the only element of $K^n$ that is fixed by all these automorphisms, translation automorphisms will not appear. So by Theorem \ref{cqg.9}, we have to prove the following statement: if $\cT$ is a compact second countable group and $\pi : K^n \to \cT$ is a pmp isomorphism such that for all $A \in \cG$, $\pi \circ \al_A \circ \pi^{-1}$ is a.e.\ equal to a group automorphism of $\cT$, then $\pi$ itself is a.e.\ equal to a group isomorphism.

Define $\cG_1 \cong \SL(n-1,\Z)$ as the subgroup of matrices $A$ such that $A_{11} = 1$ and $A_{1i} = 0 = A_{i1}$ for all $i \geq 2$. The group $\cG_1$ acts naturally on $K^{n-1}$ and the restriction of the action $\cG \actson K^n$ to $\cG_1$ is the product $\cG_1 \actson K \times K^{n-1}$ of the trivial action on $K$ and the natural action on $K^{n-1}$. Note that the Pontryagin dual $\widehat{K}$ is torsion free, because $K$ is connected. Dualizing $\cG_1 \actson K^{n-1}$ gives an action $\cG_1 \actson {\widehat{K}}^{n-1}$, which has infinite orbits because $\widehat{K}$ is torsion free and $n \geq 3$. So, the action $\cG_1 \actson K^{n-1}$ is weakly mixing.

Denote by $B_1 \subset L^\infty(K^n)$ the von Neumann subalgebra of functions that only depend on the first variable. By the discussion in the previous paragraph, $B_1 = L^\infty(K^n)^{\cG_1}$. Also, w.r.t.\ the diagonal action $\cG_1 \actson K^n \times K^n$, we get that $L^\infty(K^n \times K^n)^{\cG_1} = B_1 \ovt B_1$.

For every $A \in \cG$, we define $\be_A \in \Autgr(\cT)$ such that $\be_A = \pi \circ \al_A \circ \pi^{-1}$ a.e. We also consider the von Neumann algebra isomorphism $\pi_* : L^\infty(K^n) \to L^\infty(\cT) : \pi_*(F) = F \circ \pi^{-1}$. By construction, $\pi_* \circ \al_A = \be_A \circ \pi_*$ for all $A \in \cG$. Define $D_1 = \pi_*(B_1)$. Then, $D_1 = L^\infty(\cT)^{\be_{\cG_1}}$. Denote by $\Delta_\cT : L^\infty(\cT) \to L^\infty(\cT \times \cT)$ the comultiplication of the compact group $\cT$. Since every $\be_A$ is a group automorphism, using the diagonal $\be$-action on $L^\infty(\cT \times \cT)$, we get that
$$\Delta_{\cT}(D_1) \subset L^\infty(\cT \times \cT)^{\be_{\cG_1}} = (\pi_* \ot \pi_*)(L^\infty(K^n \times K^n)^{\al_{\cG_1}}) = D_1 \ovt D_1 \; .$$
By Lemma \ref{lem.invariant-subalg-quotient} below, we find a compact group $T$ and a quotient homomorphism $\theta_1 : \cT \to T$ such that $D_1 = \{F \circ \theta_1 \mid F \in L^\infty(T)\}$. So, we find a pmp isomorphism $\delta : K \to T$ such that $\theta_1(\pi(a)) = \delta(a_1)$ for a.e.\ $a \in K^n$.

Fix $i \geq 2$ and define $\si_i \in \cG$ such that
$$(\si_i \cdot a)_k = \begin{cases} a_i &\quad\text{if $k=1$,} \\ a_1^{-1} &\quad\text{if $k=i$,} \\ a_k &\quad\text{if $k \not\in \{1,i\}$.}\end{cases}$$
Since $\be_{\si_i}$ is a group automorphism of $\cT$, $\theta_i := \theta_1 \circ \be_{\si_i} : \cT \to T$ is a quotient homomorphism. Since $\be_{\si_i} \circ \pi = \pi \circ \al_{\si_i}$ a.e., we get that $\theta_i(\pi(a)) = \delta(a_i)$ for a.e.\ $a \in K^n$.

Since the tautological map $a \mapsto (a_1,\ldots,a_n)$ is a pmp isomorphism, also the continuous group homomorphism $\theta : \cT \to T^n : \theta(b) = (\theta_1(b),\ldots,\theta_n(b))$ induces a pmp isomorphism. This implies that $\theta$ is an isomorphism of compact groups. By construction, $\theta \circ \pi = \delta \times \cdots \times \delta$ a.e.

We prove that $\delta$ is a.e.\ equal to a group isomorphism $K \cong T$. Since also $\theta$ is a group isomorphism, it then follows that $\pi$ is a.e.\ equal to a group isomorphism.

Define $\eta \in \cG$ such that $(\eta \cdot a)_1 = a_1 a_2$ and $(\eta \cdot a)_i = a_i$ for all $i \geq 2$. We then define $\phi \in \Autgr(T^n)$ such that $\phi$ is a.e.\ equal to $(\theta \circ \pi) \circ \al_\eta \circ (\theta \circ \pi)^{-1}$. Since $\theta \circ \pi = \delta \times \cdots \times \delta$ a.e., it follows that for a.e.\ $b \in T^n$,
$$\phi(b)_1 = \delta(\delta^{-1}(b_1) \delta^{-1}(b_2)) \quad\text{and}\quad \phi(b)_i = b_i \;\;\text{for all $i \geq 2$.}$$
By continuity of $\phi$, we get that $\phi(b)_i = b_i$ for all $b \in T^n$ and all $i \geq 2$. By a similar continuity argument, the map $b \mapsto \phi(b)_1$ only depends on the coordinates $b_1,b_2$. We thus find pmp homeomorphisms $\phi_b$ of $T$ such that
$$\phi(b) = (\phi_{b_2}(b_1),b_2,\ldots,b_n) \quad\text{for all $b \in T^n$.}$$
In particular, the map $(b,c) \mapsto \phi_c(b)$ is continuous. By construction,
\begin{equation}\label{eq.good-phi}
\phi_c(b) = \delta(\delta^{-1}(b) \delta^{-1}(c)) \quad\text{for a.e.\ $(b,c) \in T^2$.}
\end{equation}
Since $\phi$ is a group automorphism, we get that
\begin{equation}\label{eq.double-mult}
\phi_{c_1 c_2}(b_1 b_2) = \phi_{c_1}(b_1) \phi_{c_2}(b_2) \quad\text{for all $(b_1,c_1)$, $(b_2,c_2)$ in $T^2$.}
\end{equation}
In particular, $\phi_e \in \Autgr(T)$. Taking $c_2 = b_1 = e$, and next taking $c_1 = b_2 = e$, we find that
$$\phi_c(b) = \phi_c(e) \phi_e(b) \quad\text{and}\quad \phi_c(b) = \phi_e(b) \phi_c(e) \quad\text{for all $b,c \in T$.}$$
We write $\psi(c) := \phi_c(e)$ and note that $\psi(c)$ belongs to the center of $T$ for all $c \in T$. Since $\phi_c(b) = \psi(c) \phi_e(b) = \phi_e(b) \psi(c)$ and $\phi_e$ is a group automorphism, it follows from \eqref{eq.double-mult} that $\psi : T \to \cZ(T)$ is a continuous group homomorphism.

By \eqref{eq.good-phi},
\begin{equation}\label{eq.better-phi}
\phi_e(\delta(a)) \psi(\delta(b)) = \phi_{\delta(b)}(\delta(a)) = \delta(ab) \quad\text{for a.e.\ $(a,b) \in K^2$.}
\end{equation}
Since $K$ is abelian, $\delta(ab) = \delta(ba)$ and we conclude that $\phi_e(a) \psi(b) = \phi_e(b) \psi(a)$ for a.e.\ $a,b \in T$. By continuity, the equality holds for all $a,b \in T$. Taking $b = e$, it follows that $\phi_e(a) = \psi(a)$ for all $a \in T$. So, $T$ is abelian and $\psi = \phi_e$.

Then consider the group automorphism $\phi^2 = \phi \circ \phi$ of $T^n$. On the one hand,
$$\phi^2(b) = \phi(\phi_e(b_1) \phi_e(b_2),b_2,\ldots,b_n) = (\phi_e^2(b_1) \phi_e^2(b_2) \phi_e(b_2),b_2,\ldots,b_n) \; .$$
Since $(\eta^2 \cdot a)_1 = a_1 a_2^2$ and $(\eta^2 \cdot a)_i = a_i$ for all $i \geq 2$, we have on the other hand that
$$\phi^2(\delta(a_1),\ldots,\delta(a_n)) = (\delta(a_1 a_2^2) , \delta(a_2),\ldots,\delta(a_n))$$
for a.e.\ $a \in K^n$. It follows that
\begin{equation}\label{eq.tussenstap}
\phi_e^2(\delta(a)) \phi_e^2(\delta(b)) \phi_e(\delta(b)) = \delta(a b^2) \quad\text{for a.e.\ $(a,b) \in K^2$.}
\end{equation}
We claim that $K \to K : b \mapsto b^2$ is a measure preserving factor map. Since this map is a continuous group homomorphism, we only need to prove that it is surjective. If it would not be surjective, we find a character $\om \in \widehat{K}$ such that $\om(b^2)=1$ for all $b \in K$. Writing the group operation in $\widehat{K}$ additively, this means that $\om + \om = 0$, contradicting the fact that $\widehat{K}$ is torsion free, because $K$ is connected.

Since by \eqref{eq.better-phi}, $\delta(ab) = \phi_e(\delta(a)) \phi_e(\delta(b))$ for a.e.\ $(a,b) \in K^2$ and since $b \mapsto b^2$ is a measure preserving factor map, we get that $\delta(ab^2) = \phi_e(\delta(a)) \phi_e(\delta(b^2))$ for a.e.\ $(a,b) \in K^2$. In combination with \eqref{eq.tussenstap}, this means that
$$\phi_e^2(\delta(a)) \phi_e^2(\delta(b)) \phi_e(\delta(b)) = \phi_e(\delta(a)) \phi_e(\delta(b^2))$$
for a.e.\ $(a,b) \in K^2$. By the Fubini theorem, we find a $b \in K$ such that the equality holds for a.e.\ $a \in K$. Writing $c = \phi_e(\delta(b^2)) (\phi_e^2(\delta(b))\phi_e(\delta(b)))^{-1}$, we get that $\phi_e^2(\delta(a)) = \phi_e(\delta(a)) c$ for a.e.\ $a \in K$. By continuity, this implies that $\phi_e^2(a) = \phi_e(a) c$ for all $a \in T$. Taking $a=e$, we find that $c=e$. We next conclude that $\phi_e = \id$.

We have thus proven that $\delta(ab) = \delta(a) \delta(b)$ for a.e.\ $(a,b) \in K^2$, which means that $\delta$ is a.e.\ equal to a group isomorphism $K \cong T$.

(ii) Assume that $\pi \in \Autpmp(K^n)$ commutes with $\cG$. Repeating the first paragraphs of the proof of (i), we find a pmp isomorphism $\delta : K \to K$ such that $\pi(a) = (\delta(a_1),\ldots,\delta(a_n))$ for a.e.\ $a \in K^n$. Expressing that $\pi$ commutes with the element $\eta \in \cG$ that we used in the proof of (i), it follows that $\delta(ab) = \delta(a) \delta(b)$ for a.e.\ $(a,b) \in K^2$. So, $\delta$ is a.e.\ equal to an automorphism of $K$.
\end{proof}

The following lemma is certainly well-known, but in order to keep this paper self-contained, we include a short proof.

\begin{lemma}\label{lem.invariant-subalg-quotient}
Let $K$ be a compact second countable group, write $A = L^\infty(K)$ and consider the comultiplication $\Delta_K : A \to A \ovt A$. If $B \subset A$ is a von Neumann subalgebra satisfying $\Delta(B) \subset B \ovt B$, there is a unique closed normal subgroup $T \lhd K$ such that $B = L^\infty(K/T)$.
\end{lemma}
\begin{proof}
Denote by $\Delta_B$ the restriction of $\Delta_K$ to $B$. The restriction of the Haar state remains invariant. So $(B,\Delta_B)$ is a compact quantum group in the sense of Definition \ref{def.cqg}. Choose a complete set $\cJ$ of inequivalent irreducible unitary corepresentations of $(B,\Delta_B)$. By Theorem \ref{cqg.5}, the linear span $\cB$ of all coefficients of $\pi \in \cJ$ is dense in $B$.

Since $B \subset L^\infty(K)$, every $\pi \in \cJ$ can be viewed as an irreducible unitary corepresentation of $(L^\infty(K),\Delta_K)$, and thus as an irreducible unitary representation of $K$. Define the closed normal subgroup $T < K$ as $T = \bigcap_{\pi \in \cJ} \Ker \pi$. By construction, $\cB \subset L^\infty(K/T)$. We may view $\cJ$ is a family of irreducible unitary representations of the compact group $K/T$ that, by definition, are separating the points of $K/T$. So, the coefficients of all $\pi \in \cJ$ span a dense $*$-subalgebra of $L^\infty(K/T)$. This means that $\cB$ is dense in $L^\infty(K/T)$ and thus $B = L^\infty(K/T)$.

The uniqueness of $T$ is obvious.
\end{proof}

While we will see below that relative rigidity of compact groups $K$ can only be shown in specific cases using rather ad hoc methods, relative rigidity holds very generally for duals of discrete groups.

\begin{theorem}\label{thm.rigid-LG-icc}
Let $G$ be an icc group and consider the compact quantum group $(L(G),\Delta_G)$, where $\Delta_G(u_g) = u_g \ot u_g$ for all $g \in G$. We equip $L(G)$ with the Haar tracial state $\tau$.

If $\cG < \Aut(G)$ is a countable subgroup that contains all inner automorphisms, then $(L(G),\Delta_G)$ is strictly rigid relative to the corresponding $\cG < \Aut(L(G),\Delta_G)$.
\end{theorem}
\begin{proof}
We only need to prove the following statement: if $(A,\Delta)$ is a Kac type compact quantum group with Haar state $\vphi$ and $\pi : L(G) \to A$ is a state preserving von Neumann algebra isomorphism such that $\pi \circ \Ad u_g \circ \pi^{-1} \in \Aut(A,\Delta)$ for all $g \in G$, there exists a quantum group isomorphism $\pi_0 : L(G) \to A$ and a character $\om : G \to \T$ such that $\pi = \pi_0 \circ \pi_\om$, where $\pi_\om(u_g) = \om(g) u_g$ for all $g \in G$.

Define $v_g = \pi(u_g)$ and $\be_g = \Ad v_g$. By assumption, $\be_g \in \Aut(A,\Delta)$ for all $g \in G$. Since $\pi$ is state preserving, $\vphi$ is tracial and $(A,\Delta)$ is of Kac type. Denote by $E : A \ovt A \to \Delta(A)$ the unique trace preserving conditional expectation. Denote $\cE = \Delta^{-1} \circ E : A \ovt A \to A$.

Fix $g \in G$. Since $(v_g \ot v_g) \Delta(a) = ((\be_g \ot \be_g)\Delta(a)) (v_g \ot v_g) = \Delta(\be_g(a)) (v_g \ot v_g)$, we find that the element $d_g \in A$ defined by $d_g := \cE(v_g \ot v_g)$ satisfies $d_g a = \be_g(a) d_g$ for all $a \in A$. Since $L(G)$ is a factor, also $A$ is a factor and we conclude that $d_g$ is a (potentially zero) multiple of $v_g$. We thus uniquely define $\eta_g \in \C$ such that $d_g = \eta_g v_g$.

Fix a complete set $\Irr$ of inequivalent irreducible unitary corepresentations $X$ of $(A,\Delta)$ and denote by $d(X)$ their dimension. By Theorem \ref{cqg.6}, $\{d(X)^{1/2} X_{ij} \mid X \in \Irr , 1 \leq i,j \leq d(X)\}$ is an orthonormal basis of the GNS Hilbert space of $(A,\vphi)$. Every element $a \in A$ can thus be uniquely written as
$$a = \sum_{X \in \Irr} \sum_{i,j=1}^{d(X)} (a)^X_{ij} X_{ij} \; ,$$
with convergence in the $2$-norm given by $\vphi$. Note that $(a)^X_{ij} = d(X) \, \vphi(a X_{ij}^*)$. Define the matrices $P_g^X \in M_{d(X)}(\C)$ as $(P_g^X)_{ij} = (v_g)^X_{ij}$.

We thus find that
\begin{align*}
\eta_g \, (P_g^X)_{ij} &= (d_g)^X_{ij} = d(X) \, \vphi(\cE(v_g \ot v_g) X_{ij}^*) = d(X) \, (\vphi \ot \vphi)((v_g \ot v_g) \Delta(X_{ij}^*)) \\
& = d(X) \sum_{k=1}^{d(X)} \vphi(v_g X_{ik}^*) \, \vphi(v_g X_{kj}^*) = d(X)^{-1} \sum_{k=1}^{d(X)} (P_g^X)_{ik} \, (P_g^X)_{kj} \; .
\end{align*}
This means that
\begin{equation}\label{eq.key-formula-PgX}
(P_g^X)^2 = d(X) \eta_g \, P_g^X \quad\text{for all $g \in G$ and $X \in \Irr$.}
\end{equation}
By Theorem \ref{cqg.7}, the compact quantum group $(A,\Delta)$ of Kac type has the antipode $S : A \to A$, which is a $*$-anti-automorphism of $A$ that, by construction, commutes with all quantum group automorphisms. In particular, $\be_g(S(a)) = S(\be_g(a))$ for all $g \in G$, $a \in A$. Since $S$ is a $*$-anti-automorphism, it follows that $\Ad S(v_g^*) = \Ad(v_g)$ for all $g \in G$. There thus exist $\nu_g \in \T$ such that $S(v_g^*) = \nu_g v_g$ for all $g \in G$.

Since $S(X_{ij}) = X_{ji}^*$, we find that $\overline{(P_g^X)_{ji}} = \nu_g \, (P_g^X)_{ij}$. This means that $(P_g^X)^* = \nu_g P_g^X$. In combination with \eqref{eq.key-formula-PgX}, we get that
\begin{equation}\label{eq.second-key-formula-PgX}
d(X) \nu_g \eta_g \, P_g^X = (P_g^X)^* \, P_g^X \quad\text{for all $g \in G$ and $X \in \Irr$.}
\end{equation}
If for some $g \in G$, $\eta_g = 0$, it follows that $P_g^X = 0$ for all $X \in \Irr$, so that $v_g = 0$, which is absurd because $v_g$ is a unitary. So, $\eta_g \neq 0$ for all $g \in G$. Since the right hand side of \eqref{eq.second-key-formula-PgX} is self-adjoint, while $(P_g^X)^* = \nu_g P_g^X$, we also get that $\overline{\eta_g} = \nu_g \eta_g$.

It then follows that $Q_g^X = d(X)^{-1} \eta_g^{-1} P_g^X$ is a self-adjoint projection in $M_{d(X)}(\C)$. Take distinct $g,h \in G$. Since $\tau(u_g u_h^*) = 0$ and $\pi$ is state preserving, we find that
$$0 = \vphi(v_g v_h^*) = \sum_{X \in \Irr} d(X)^{-1} \Tr(P_g^X (P_h^X)^*) = \eta_g \overline{\eta_h} \sum_{X \in \Irr} d(X) \Tr(Q_g^X Q_h^X) \; .$$
Since $\eta_g \neq 0$ and $\eta_h \neq 0$, we conclude that $\sum_{X \in \Irr} d(X) \Tr(Q_g^X Q_h^X)=0$. Since $Q_g^X$ and $Q_h^X$ are orthogonal projections, $\Tr(Q_g^X Q_h^X) \geq 0$ for all $X \in \Irr$. We conclude that
\begin{equation}\label{eq.orthog-Q}
Q_g^X Q_h^X = 0 \quad\text{for all $g \neq h$ and $X \in \Irr$.}
\end{equation}
In particular, for every $X \in \Irr$, $\{Q_g^X \mid g \in G\}$ is a commuting family of orthogonal projections, whose linear span thus forms an abelian $*$-subalgebra $D_X \subset M_{d(X)}(\C)$. Fix $X \in \Irr$. Define the normal linear map $E_X : A \to M_{d(X)}(\C) : (E_X(a))_{ij} = (a)^X_{ij}$. By construction, $D_X$ equals the linear span of $\{E_X(v_g) \mid g \in G\}$. By normality, $D_X = E_X(A)$. Since $E_X$ is surjective, we conclude that $D_X = M_{d(X)}(\C)$ for all $X \in \Irr$. Because $D_X$ is abelian, it follows that $d(X) = 1$ for all $X \in \Irr$. This means that $(A,\Delta)$ is the dual of a discrete group.

More precisely, $\Irr$ is a subgroup of $\cU(A)$ that generates $A$, such that $\Delta(X) = X \ot X$ for all $X \in \Irr$ and $\vphi(X) = 0$ if $X$ is not the trivial corepresentation.

Fix $g \in G$. We claim that there is a unique $X \in \Irr$ such that $P_g^X \neq 0$. Since $v_g \neq 0$, there is at least one $X \in \Irr$ such that $P_g^X \neq 0$. Assume that $X \neq Y$ and that both $P_g^X \neq 0$ and $P_g^Y \neq 0$. Since $X$ and $Y$ are $1$-dimensional, this means that $Q_g^X = 1$ and $Q_g^Y = 1$. It follows from \eqref{eq.orthog-Q} that $Q_h^X = 0$ and $Q_h^Y = 0$ for all $h \in G \setminus \{g\}$. Define the normal linear map $E_1 : A \to \C X + \C Y : E_1(a) = (a)^X X + (a)^Y Y$. Since $E_1(v_h)=0$ for all $h \in G \setminus \{g\}$, the image of $E_1$ is at most $1$-dimensional. But $E_1$ is surjective. So the claim is proven.

By the claim, for every $g \in G$, there is a unique $\delta(g) \in \Irr$ and $\om(g) \in \T$ such that $\pi(u_g) = \om(g) \delta(g)$. Since $\pi$ is a $*$-isomorphism, it follows that $\om : G \to \T$ is a character. Then $\pi = \pi_0 \circ \pi_\om$, where the $*$-isomorphism $\pi_0 : L(G) \to A$ satisfies $\pi_0(u_g) = \delta(g)$ for all $g \in G$. It follows that $\pi_0$ is a quantum group isomorphism.
\end{proof}

\subsection{Examples of relatively rigid noncommutative groups}\label{sec.examples-relative-rigidity-noncommutative}

In this section, we give several examples of noncommutative groups $K$ that are rigid relative to their automorphism group.

In Proposition \ref{prop.no-go-connected-Lie} in the next section, we prove that \emph{no} noncommutative compact connected group is rigid relative to \emph{any} group of automorphisms. So it is natural to consider finite groups and their direct products with tori $\T^n$ as candidates for relative rigidity. Note that here we systematically view finite groups $K$ as being compact groups with associated compact quantum group $(L^\infty(K),\Delta_K)$.

In Theorem \ref{thm.relative-rigid-An-tilde}, we prove relative rigidity for the symmetric groups $S_n$, the alternating groups $A_n$ and their double cover $\Atil_n$. Before doing that, we prove in Theorem \ref{thm.rigid-SL2} that the groups $\SL_2(\F_p)$ are relatively rigid, where $\F_p = \Z/p\Z$.

In Theorem \ref{thm.rigid-SLn-Fp}, we prove relative rigidity for the groups $\SL_n(\F_q)$ and $\PSL_n(\F_q)$ when $n \geq 3$, as well as for the groups $\F_q^n \rtimes \SL_n(\F_q)$ when $n \geq 2$, where $\F_q$ is the finite field of order $q =p^k$. In Proposition \ref{prop.direct-product-relative-rigid}, we analyze when a direct product of $\T^n$ with a finite group is relatively rigid.

To give examples of quantum W$^*$-superrigidity, we not only need relative rigidity of a finite group $K$, but we also need that $K$ is a perfect group and that $H^2(K,\T)$ is trivial. For that reason, we also include the double cover $\Atil_n$ of $A_n$~; cf.\ Section \ref{sec.about-trivial-2-cohomology}.

\begin{lemma}\label{lem.finite-group-rigid}
Let $K$ be a finite group.
\begin{enumlist}
\item\label{lem.finite-group-rigid.one} $K$ is rigid relative to $\Aut K$ if and only if the following holds: whenever $\Lambda$ is a group and $\pi : K \to \Lambda$ is a bijection satisfying $\pi \circ \al \circ \pi^{-1} \in \Aut \Lambda$ for all $\al \in \Aut K$, we have $\Lambda \cong K$.

\item\label{lem.finite-group-rigid.two} If $K$ is rigid relative to $\Aut K$, if every automorphism of $\Aut K$ is inner and if $\Lambda$ is a group and $\pi : K \to \Lambda$ a bijection such that $\pi \circ \al \circ \pi^{-1} \in \Aut \Lambda$ for all $\al \in \Aut K$, there exists a group isomorphism $\phi : K \to \Lambda$ such that $\pi \circ \al \circ \pi^{-1} = \phi \circ \al \circ \phi^{-1}$ for all $\al \in \Aut K$.
\end{enumlist}
\end{lemma}
\begin{proof}
(i) One implication being obvious, assume that the second statement holds. Let $\Lambda$ be a finite group, $\pi : K \to \Lambda$ a bijection and $\Aut K \actson^\be \Lambda$ an action such that $\be_\al \circ \pi = \pi \circ \al$ for all $\al \in \Aut K$. By assumption, there exists a group isomorphism $\theta : \Lambda \to K$. Write $\pitil = \theta \circ \pi$. Then $\pitil$ is a permutation of the set $K$ and it suffices to prove that $(\Ad \pitil)(\Aut K) = \Aut K$ as subgroups of the permutation group of the set $K$. By construction, $(\Ad \pitil)(\Aut K) \subset \Aut K$. Since $K$ is finite, $\Aut K$ is finite and this inclusion must be an equality.

(ii) Since $K$ is rigid relative to $\Aut K$, we first find a group isomorphism $\theta : \Lambda \to K$. As explained in the proof of (i), $\Ad (\theta \circ \pi)$ defines an automorphism of the group $\Aut K$. By assumption, we find an automorphism $\theta_1 \in \Aut K$ such that $\Ad (\theta \circ \pi) = \Ad \theta_1$ as automorphisms of $\Aut K$. Then $\phi = \theta^{-1} \circ \theta_1$ has the required property.
\end{proof}

\begin{theorem}\label{thm.rigid-SL2}
For every prime $p$, the groups $\SL_2(\F_p)$, $\PSL_2(\F_p)$ and $\PGL_2(\F_p)$ are rigid relative to their automorphism group.
\end{theorem}
\begin{proof}
We start with the exceptional case $p=2$, in which $\SL_2(\F_2) \cong \PSL_2(\F_2) \cong \PGL_2(\F_2) \cong S_3$. The group $S_3$ has order $6$ and the only other group of order $6$ is $\Z/6\Z$. The inner automorphisms give $S_3 < \Aut S_3$ (which is of course an equality), while the automorphism group of $\Z/6\Z$ has order $2$. So there is no bijection $\pi : S_3 \to \Z/6\Z$ such that $\pi \circ \al \circ \pi^{-1} \in \Aut \Z/6\Z$ for all $\al \in \Aut S_3$. Therefore for $p=2$, relative rigidity follows.

For $i \neq j$, we denote by $E_{ij}$ the matrix with $1$ in position $(i,j)$ and $0$ elsewhere and denote by $I_2$ the $2 \times 2$ identity matrix. We denote by $e_1,e_2$ the standard basis vectors $\F_p^2$.

Recall that for each of the groups $\SL_2(\F_p)$, $\PSL_2(\F_p)$ and $\PGL_2(\F_p)$, the automorphism group is naturally identified with $\PGL_2(\F_p)$ acting by the automorphisms $\Ad A$. For $\PSL_2(\F_p)$, this is proven in \cite{SVdW28}. Then the result can be easily deduced for $\SL_2(\F_p)$ because every of its automorphisms must act as the identity on the center $\{\pm I_2\}$ of order $2$, while for $\PGL_2(\F_p)$, every of its automorphisms must globally preserve its unique index $2$ subgroup $\PSL_2(\F_p)$.

For the rest of the proof, assume that $p \geq 3$. Let $K$ be one of the groups $\SL_2(\F_p)$, $\PSL_2(\F_p)$ or $\PGL_2(\F_p)$. Let $\Lambda$ be a group and $\pi : K \to \Lambda$ a bijection such that $\pi \circ \al \circ \pi^{-1} \in \Aut \Lambda$ for every $\al \in \Aut K$. By Lemma \ref{lem.finite-group-rigid}, it suffices to prove that $\Lambda \cong K$.

Denote by $\Gamma < \Lambda$ the subgroup generated by all $p$-Sylow subgroups of $\Lambda$. By construction, $\Gamma$ is a characteristic subgroup of $\Lambda$. We denote by $\Phi : \Aut \Lambda \to \Aut \Gamma : \gamma \mapsto \gamma|_{\Gamma}$ the restriction homomorphism. We then consider the group homomorphism $\Theta : \Aut K \to \Aut \Gamma : \Theta(\al) = \Phi(\pi \circ \al \circ \pi^{-1})$. The main step in the proof is to show the following two statements.
\begin{enumlist}
\item $\Phi(\Aut \Lambda) = \Theta(\Aut K)$.
\item $\Theta$ is faithful and the group $S := \{s \in \Lambda \mid \Phi(\Ad s) = \id\}$ satisfies $|S| \mid p-1$.
\end{enumlist}
As we will explain below, once (i) and (ii) are proven, the homomorphism $\Lambda \to \PGL_2(\F_p) : s \mapsto \Theta^{-1}(\Phi(\Ad s))$ writes $\Lambda$ as an extension of degree at most $2$ of $\PSL_2(\F_p)$ or $\PGL_2(\F_p)$, and then the conclusion will follow easily.

We first prove (i) and (ii) in the case where $K = \SL_2(\F_p)$. Note that $\SL_2(\F_p)$ has order $p(p^2-1)$.

For every $A \in \GL_2(\F_p)$, define $\be_A = \pi \circ (\Ad A) \circ \pi^{-1} \in \Aut \Lambda$. Since $\{\pm I_2 + x E_{12} \mid x \in \F_p\}$ is the fixed point subgroup of $\Ad(I_2 + E_{12})$ in $K$, its image $\{\pi(\pm I_2 + x E_{12}) \mid x \in \F_p\}$ in $\Lambda$ is the fixed point subgroup of $\be_{I_2 + E_{12}}$, and thus a subgroup of $\Lambda$. Denote by $P^1$ the projective line, i.e.\ the set of $1$-dimensional subspaces of $\F_p^2$. Since every $V \in P^1$ is of the form $B(\F_p e_1)$ for some $B \in \GL_2(\F_p)$, conjugating with $\Ad B$ implies that for every $V \in P^1$,
$$K_V := \{A \in K \mid \exists \eps \in \{\pm 1\}, \forall v \in V : A \cdot v = \eps v \}$$
is the fixed point subgroup of $\Ad(B(I_2 + E_{12})B^{-1})$, so that its image $\pi(K_V)$ is a subgroup of $\Lambda$. Note that $|K_V| = 2p$, so that $\pi(K_V)$ is a group of order $2 p$. Since $p$ is odd, $\pi(K_V)$ has a unique $p$-Sylow subgroup, which we denote as $\Lambda_V$. Note that $|\Lambda_V| = p$ and that every element of order $p$ in $\pi(K_V)$ is contained in $\Lambda_V$.

Also note that for every $B \in \GL_2(\F_p)$ and $V \in P^1$, we have that $(\Ad B)(K_V) = K_{B(V)}$, so that $\be_B(\pi(K_V)) = \pi(K_{B(V)})$ and thus $\be_B(\Lambda_V) = \Lambda_{B(V)}$.

When $V,W \in P^1$ are distinct, $K_V \cap K_W = \{\pm I_2\}$. So, $\Lambda_V \cap \Lambda_W \subset \pi(\{\pm I_2\})$. Since every nontrivial element in $\Lambda_V$ has order $p$, we conclude that $\Lambda_V \cap \Lambda_W = \{e\}$. Since $|P^1| = p+1$, we have found $p+1$ $p$-Sylow subgroups $(\Lambda_V)_{V \in P^1}$ in $\Lambda$. We claim that these are all the $p$-Sylow subgroups of $\Lambda$. Assume the contrary. Denote by $n_p$ the number of $p$-Sylow subgroups of $\Lambda$. Since $n_p \equiv 1$ modulo $p$ and, by assumption, $n_p > 1+p$, we can write $n_p = 1 + (k+1)p$ with $k \geq 1$. Since $|\Lambda| = |K| = p (p^2-1)$, we have that $n_p \mid p^2-1$ and we write $p^2-1 = a n_p$ with $a \geq 1$. Thus, $(p+1)(p-1) = p^2-1 = a n_p = a (p+1) + a k p$, so that $p+1 \mid a k p$. Since $p$ is prime and $p \nmid p+1$, it follows that $p+1 \mid ak$, so that $ak > p$. Then,
$$p^2-1 = a n_p = a(p+1) + a kp > a(p+1) + p^2 > p^2 \; ,$$
which is absurd. We therefore conclude that $n_p = p+1$ and that $(\Lambda_V)_{V \in P^1}$ is the complete list of $p$-Sylow subgroups of $\Lambda$.

We deduce statement (i) above from the following two results that we prove first. Let $\gamma \in \Aut \Lambda$ be any automorphism. Recall that we defined $\Gamma < \Lambda$ as the subgroup generated by all $p$-Sylow subgroups of $\Lambda$.
\begin{enumlist}[label=(\alph*),ref=(\alph*)]
\item There exists $A \in \GL_2(\F_p)$ such that the automorphism $\gamma_0 = \be_A \circ \gamma$ satisfies $\gamma_0(\Lambda_0) = \Lambda_0$ for every $p$-Sylow subgroup $\Lambda_0 < \Lambda$.
\item If $\gamma \in \Aut \Lambda$ satisfies $\gamma(\Lambda_0) = \Lambda_0$ for every $p$-Sylow subgroup $\Lambda_0 < \Lambda$, then $\gamma(s) = s$ for all $s \in \Gamma$.
\end{enumlist}
{\bf Proof of (a).} Consider the distinct $p$-Sylow subgroups $\Lambda_i := \Lambda_{\F_p e_i}$ with $i=1,2$. Define $N_2 = \cN_\Lambda(\Lambda_2)$ as the normalizer of $\Lambda_2 < \Lambda$. Since $n_p = p+1$ and $|\Lambda| = p (p+1)(p-1)$, we get that $|N_2| = p (p-1)$. Since $N_2/\Lambda_2$ has order $p-1$, every element of order $p$ in $N_2$ belongs to $\Lambda_2$. In particular, $\Lambda_1 \cap N_2 = \{e\}$. It follows that the subgroups $(s \Lambda_2 s^{-1})_{s \in \Lambda_1}$ are all distinct. Note that for $s \in \Lambda_1$, we also have that $s \Lambda_2 s^{-1} \neq \Lambda_1$, because otherwise $\Lambda_2 = s^{-1} \Lambda_1 s = \Lambda_1$. We conclude that the $p+1$ $p$-Sylow subgroups of $\Lambda$ are precisely given by $\Lambda_1$ and $(s \Lambda_2 s^{-1})_{s \in \Lambda_1}$.

To prove (a), it thus suffices to find $A \in \GL_2(\F_p)$ such that the automorphism $\gamma_0 = \be_A \circ \gamma$ satisfies $\gamma_0(s) = s$ for all $s \in \Lambda_1$ and $\gamma_0(\Lambda_2) = \Lambda_2$.

Since $\gamma(\Lambda_1)$ and $\gamma(\Lambda_2)$ are distinct $p$-Sylow subgroups of $\Lambda$, we can take distinct $V_1,V_2 \in P^1$ such that $\gamma(\Lambda_i) = \Lambda_{V_i}$ for $i = 1,2$. Take nonzero vectors $v_i \in V_i$. Since $V_1 \neq V_2$, the vectors $v_1,v_2$ are linearly independent and we can define $B \in \GL_2(\F_p)$ by $B(v_i) = e_i$ for $i=1,2$. Then, $B(V_i) = \F_p e_i$, so that $\be_B(\Lambda_{V_i}) = \Lambda_{\F_p e_i} = \Lambda_i$. Write $\gamma_1 := \be_B \circ \gamma$.

Since $V_1 = \F_p e_1$, we have that $K_{V_1} = \{\pm I_2 + x E_{12} \mid x \in \F_p \}$. Take $\eps_1 \in \{\pm 1\}$ and $x_1 \in \F_p$ such that $s_1 := \pi(\eps_1 I_2 + x_1 E_{12}) \in \Lambda_1 \setminus \{e\}$. Since $\Lambda_1$ intersects trivially the subgroup $\pi(\{\pm I_2\})$ of order $2$, we must have that $x_1 \neq 0$. For every $\al \in \F_p^\times$, define the matrix $A(\al) := \al E_{11} + E_{22}$. Denote $\be_\al := \be_{\Ad A(\al)}$. Since $(\Ad A(\al))(\pm I_2 + x E_{12}) = \pm I_2 + \al x E_{12}$, we have that $(\Ad A(\al))(K_{V_1}) = K_{V_1}$ and thus $\be_\al(\Lambda_1) = \Lambda_1$. Since $\be_\al(\pi(\eps_1 I_2 + x_1 E_{12})) = \pi(\eps_1 I_2 + \al x_1 E_{12})$, we conclude that $\pi(\eps_1 I_2 + x E_{12})$, $x \in \F_p^\times$, are precisely the elements of $\Lambda_1 \setminus \{e\}$. Since $\gamma_1(s_1) \in \Lambda_1$, we can thus choose $\al \in \F_p^\times$ such that $\be_\al(\gamma_1(s_1)) = s_1$. Define $\gamma_0 = \be_\al \circ \gamma_1$. Since $\Lambda_1$ is generated by $s_1$, we get that $\gamma_0(s) = s$ for all $s \in \Lambda_1$. In particular, $\gamma_0(\Lambda_1) = \Lambda_1$.

Since $\gamma_0(\Lambda_2) = \be_\al(\Lambda_2) = \Lambda_{A(\al)(\F_p e_2)} = \Lambda_{\F_p e_2} = \Lambda_2$, we then also find that $\gamma_0(s \Lambda_2 s^{-1}) = \gamma_0(s) \Lambda_2 \gamma_0(s)^{-1} = s \Lambda_2 s^{-1}$ for all $s \in \Lambda_1$. Since the $p$-Sylow subgroups of $\Lambda$ are $\Lambda_1$ and $(s \Lambda_2 s^{-1})_{s \in \Lambda_1}$, statement (a) is proven.

{\bf Proof of (b).} Choose arbitrary distinct $p$-Sylow subgroups $\Lambda_1,\Lambda_2 < \Lambda$. It suffices to prove that $\gamma(s) = s$ for all $s \in \Lambda_1$. As in the proof of (a), we get that $\Lambda_1$ and $(s \Lambda_2 s^{-1})_{s \in \Lambda_1}$ precisely are the $p$-Sylow subgroups of $\Lambda$. Fix $s \in \Lambda_1$. Since $\gamma$ globally preserves every $p$-Sylow subgroup of $\Lambda$, we have that $\gamma(s \Lambda_2 s^{-1}) = s \Lambda_2 s^{-1}$. On the other hand, $\gamma(s \Lambda_2 s^{-1}) = \gamma(s) \gamma(\Lambda_2) \gamma(s)^{-1} = \gamma(s) \Lambda_2 \gamma(s)^{-1}$. We thus find that $s \Lambda_2 s^{-1} = \gamma(s) \Lambda_2 \gamma(s)^{-1}$. Since $\gamma(s) \in \Lambda_1$, it follows that $\gamma(s) = s$.

{\bf Proof of (i).} Take $\gamma \in \Aut \Lambda$. By (a), we can take $A \in \GL_2(\F_p)$ such that $\gamma_0 := \be_A \circ \gamma$ satisfies $\gamma_0(\Lambda_0) = \Lambda_0$ for every $p$-Sylow subgroup $\Lambda_0 < \Lambda$. By (b), we get that $\gamma_0(s) = s$ for all $s \in \Gamma$. This means that $\Phi(\gamma) = \Theta(\Ad A^{-1}) \in \Theta(\Aut K)$.

{\bf Proof of (ii).} Take $A \in \GL_2(\F_p)$ such that $\Theta(\Ad A) = \id$. We have to prove that $\Ad A = \id$. As in the proof of (a) above, take $\eps_1 \in \{\pm 1\}$ and $x_1 \in \F_p^\times$ such that $\pi(\eps_1 I_2 + x_1 E_{12}) \in \Lambda_1 < \Gamma$. Since $\Theta(\Ad A) = \id$, we get that
$$\pi(\eps_1 I_2 + x_1 E_{12}) = \be_A(\pi(\eps_1 I_2 + x_1 E_{12})) = \pi(A(\eps_1 I_2 + x_1 E_{12})A^{-1}) \; .$$
It follows that $A$ commutes with $\eps_1 I_2 + x_1 E_{12}$, which forces $A \in K_{\F_p e_1}$. We similarly find that $A \in K_{\F_p e_2}$, so that $A \in \{\pm I_2\}$ and $\Ad A = \id$.

To conclude the proof of (ii), we prove that the order of $S$ divides $p-1$. Note that, by definition, $S$ equals the centralizer of $\Gamma$ in $\Lambda$. Fix two distinct $p$-Sylow subgroups $\Lambda_1,\Lambda_2 < \Lambda$ and denote by $N_i = \cN_{\Lambda}(\Lambda_i)$ the normalizer of $\Lambda_i$. Since $S$ centralizes $\Lambda_1$, we have $S < N_1$. Denote by $\theta : N_1 \to N_1 / \Lambda_1$ the quotient homomorphism. Since $|N_1 / \Lambda_1|$ has order $p-1$, it suffices to prove the restriction of $\theta$ to $S$ is faithful. Take $s \in S$ with $\theta(s) = e$. Then, $s \in \Lambda_1$. Since $s$ commutes with $\Gamma$, $s$ commutes with $\Lambda_2$. So, $s \in N_2$. We have seen before that $\Lambda_1 \cap N_2 = \{e\}$, so that $s=e$.

{\bf End of the proof.} We proved statements (i) and (ii) when $K = \SL_2(\F_p)$. First note that statements (i) and (ii) remain valid in the cases where $K = \PSL_2(\F_p)$ or $\PGL_2(\F_p)$: the proof is identical and even becomes a bit easier because now the image of $\{I_2 + x E_{12} \mid x \in \F_p\}$ in $\PGL_2(\F_p)$ is the fixed point subgroup of $\Ad(I_2 + E_{12})$.

Write $c = p-1$ when $K = \SL_2(\F_p)$ or $K = \PGL_2(\F_p)$, and write $c = (p-1)/2$ when $K = \PSL_2(\F_p)$. Then, $\Lambda$ is a group of order $p(p+1)c$. By statements (i) and (ii), $\Psi := \Theta^{-1} \circ \Phi : \Aut \Lambda \to \PGL_2(\F_p)$ is a surjective group homomorphism. Since the order of the subgroup $S < \Lambda$ in statement (ii) divides $p-1$, we get that $|S| \mid c$, so that $\Psi(\Ad \Lambda)$ is a normal subgroup of $\PGL_2(\F_p)$ whose order is a multiple of $p(p+1)$.

When $p \geq 5$, the group $\PSL_2(\F_p)$ is simple and the only normal subgroups of $\PGL_2(\F_p)$ are the trivial subgroup, $\PSL_2(\F_p)$ and the entire group. When $p=3$, by Lemma \ref{lem.exceptional-iso}, $\PGL_2(\F_3) \cong S_4$. The normal subgroups of $S_4$ are the trivial subgroup, $A_4$, $S_4$ and the extra subgroup of order $4$ given by $\{e,(12)(34),(13)(24),(14)(23)\}$. Since $\Psi(\Ad \Lambda)$ has order at least $p(p+1)$, it follows in all cases that $\Psi(\Ad \Lambda)$ is equal to $\PSL_2(\F_p)$ or $\PGL_2(\F_p)$.

In the case where $K = \PSL_2(\F_p)$, looking at the order of $\Lambda$, it already follows that $s \mapsto \Psi(\Ad s)$ must be an isomorphism $\Lambda \cong \PSL_2(\F_p)$. So, the theorem is proven in this case.

In the other cases, if $\Psi(\Ad \Lambda) = \PGL_2(\F_p)$, the same order argument implies that $\Lambda \cong \PGL_2(\F_p)$. If $\Psi(\Ad \Lambda) = \PSL_2(\F_p)$, it follows that $s \mapsto \Psi(\Ad s)$ has a kernel of order $2$, so that $\Lambda$ is a central extension $0 \to \Z/2\Z \to \Lambda \to \PSL_2(\F_p) \to e$. By e.g.\ \cite[Theorem 3.2 in Chapter 16]{Kar93}, it follows that either $\Lambda \cong \SL_2(\F_p)$, or $\Lambda \cong \SL_2(\F_p) \times \Z/2\Z$. To conclude the proof in the remaining cases, it thus suffices to show that between two distinct groups $K_1$ and $K_2$ among $\SL_2(\F_p)$, $\PGL_2(\F_p)$ and $\PSL_2(\F_p) \times \Z/2\Z$, there is no bijection $\pi : K_1 \to K_2$ satisfying $\pi \circ \al \circ \pi^{-1} \in \Aut K_2$ for all $\al \in \Aut K_1$.

From the discussion in the beginning of the proof, we know that all these groups have the same automorphism group $\PGL_2(\F_p)$. As in Lemma \ref{lem.finite-group-rigid.two}, it thus suffices to prove that there is no bijection $\pi : K_1 \to K_2$ satisfying $\pi \circ \Ad A = \Ad A \circ \pi$ for all $A \in \PGL_2(\F_p)$. We will show this by counting the number of fixed points of $\Ad A$, for specific matrices $A$.

When $A = \bigl(\begin{smallmatrix} 1 & 0 \\ 0 & -1\end{smallmatrix}\bigr)$, the number of fixed points of $\Ad A$ in resp.\ $\SL_2(\F_p)$, $\PGL_2(\F_p)$ and $\PSL_2(\F_p) \times \Z/2\Z$ is $p-1$, $2(p-1)$ and $2(p-1)$. On the other hand, the number of points fixed by all $\Ad B$ is resp.\ equal to $2$, $1$ and $2$. So we have distinguished the three groups.
\end{proof}

Given an integer $n \geq 2$, the symmetric group $S_n$ can be seen as the universal group with generators $(ab)$, for all distinct $a,b \in \{1,\ldots,n\}$, and relations
\begin{equation}\label{eq.rel-Sn}
(ab) = (ba) \;\; , \quad (ab)^2 = e \;\; , \quad (ab) (bc) (ab) = (ac) \;\;\text{if $a,b,c$ are distinct.}
\end{equation}
Note that if $n \geq 4$ and $a,b,c,d$ are distinct, it indeed follows from the relations in \eqref{eq.rel-Sn} that $(ab)$ commutes with $(cd)$, so that there is no need to add this extra relation: since $(cd) = (ac)(ad)(ac)$ and since we can conjugate each of the three factors with $(ab)$, we get that $(ab)(cd) (ab) = (bc)(bd)(bc) = (cd)$.

One can define a $2$-fold covering $\Stil_n \to S_n$ as the universal group with generators $z$ and $[ab]$, for all distinct $a,b \in \{1,\ldots,n\}$, and relations
\begin{equation}\label{eq.rel-Sntil}
\begin{split}
& z \;\;\text{is central,}\quad z^2 = e \;\; , \\
& [ab] = z [ba] \;\; , \quad [ab]^2 = e \;\; , \quad [ab] [bc] [ab] = z [ac] \;\;\text{if $a,b,c$ distinct.}
\end{split}
\end{equation}
By the same argument as above, if $n \geq 4$ and $a,b,c,d$ are distinct, we get that $[ab] [cd] = z [cd] [ab]$. We have the canonical surjective group homomorphism $\theta : \Stil_n \to S_n$ defined by $\theta(z) = e$ and $\theta([ab]) = (ab)$ for all distinct $a,b$. By universality, it is easy to check that the kernel of $\theta$ equals $\{e,z\}$. It is nontrivial that $z \neq e$, see e.g.\ \cite[Theorem 2.2 in Chapter 12]{Kar93}. It then follows that $0 \to \Z/2\Z \to \Stil_n \to S_n \to e$ is a central extension and one checks easily that for $n \geq 4$, it is not split.

We denote by $\eps : S_n \to \{\pm 1\}$ the sign of a permutation: $\eps(ab) = -1$ for all distinct $a,b$. Then the alternating group $A_n$ is the kernel of $\eps$, and we define $\Atil_n$ as the kernel of $\eps \circ \theta$. By construction, we get the central extension $0 \to \Z/2\Z \to \Atil_n \to A_n \to e$, which is again not split if $n \geq 4$ (cf.\ Lemma \ref{lem.better-generators}).

\begin{theorem}\label{thm.relative-rigid-An-tilde}
The following groups are rigid relative to their automorphism group.
\begin{enumlist}
\item\label{thm.relative-rigid-An-tilde.one} The symmetric group $S_n$ and the alternating group $A_n$ for every $n \geq 2$.
\item\label{thm.relative-rigid-An-tilde.two} The $2$-fold cover $\Atil_n$ of $A_n$ if $n \geq 4$ and $n \neq 6$.
\end{enumlist}
\end{theorem}

\begin{proof}
For later use throughout the proof, we start with the following observations when $n \geq 4$. Note that the group $A_n$ is generated by the $3$-cycles $(abc) = (ab)(bc)$. For the same reason, $\Atil_n$ is generated by the elements $[ab][bc]$, which we denote as $s(abc)$. In Lemma \ref{lem.better-generators}, we give a presentation of $A_n$ and $\Atil_n$ in terms of these generators. Then note that for $n \geq 4$, the center of $A_n$ is trivial, so that the center of $\Atil_n$ equals $\{e,z\}$.

{\bf\boldmath Automorphism groups of $A_n$, $S_n$ and $\Atil_n$.} Every $\si \in S_n$ defines automorphisms of $A_n$ and $S_n$ by conjugation, and an automorphism of $\Stil_n$ given by $z \mapsto z$ and $[ab] \mapsto [\si(a)\si(b)]$. Assume that $n \geq 4$. Since every automorphism of $\Atil_n$ preserves its center $\{e,z\}$, it must be the identity on $\{e,z\}$, so that we identify $\Aut \Atil_n = \Aut A_n$. We have $S_n < \Aut A_n$ and $S_n < \Aut S_n$ by conjugation and when $n \neq 6$, these are equalities.

{\bf\boldmath The groups $A_n$ and $S_n$ if $n \in \{2,3,4,5\}$. The groups $\Atil_n$ if $n=4,5$.} Since $S_2 \cong \Z/2\Z$, $A_2 \cong \{e\}$ and $A_3 \cong \Z/3\Z$ and since there is only one group of order $2$, resp.\ $1$, $3$, the relative rigidity follows. The other cases follow from Theorem \ref{thm.rigid-SL2} and Lemma \ref{lem.exceptional-iso} below.

{\bf\boldmath The group $A_6$.} Assume that $\Lambda$ is a group and $\pi : A_6 \to \Lambda$ a bijection such that $\be_\al := \pi \circ (\Ad \al) \circ \pi^{-1} \in \Aut \Lambda$ for all $\al \in S_6 < \Aut A_6$. Consider the transposition $\al_1 = (56)$. Define the subgroup $K_1 < A_6$ as the fixed point subgroup of $\Ad \al_1$. Note that $K_1$ consists of the permutations of the form $\mu \, (56)^{\eps(\mu)}$ where $\mu \in S_4$. In particular, $K_1 \cong S_4$.

Define the subgroup $\Lambda_1 < \Lambda$ as the fixed point subgroup of $\be_{\al_1}$. Then, $\pi(K_1) = \Lambda_1$. Denote by $\pi_1$ the restriction of $\pi$ to $K_1$. We view $S_4 < S_6$ as the subgroup of permutations of $\{1,2,3,4\}$ that fix $5$ and $6$. Note that $(\Ad \al)(K_1) = K_1$ for all $\al \in S_4$. Since every $\al \in S_4$ commutes with $\al_1 = (56)$, also $\be_\al$ restricts to an automorphism of $\Lambda_1$ for all $\al \in S_4$. By construction, $\pi_1 \circ (\Ad \al) \circ \pi_1^{-1}$ equals the restriction of $\be_\al$ to $\Lambda_1$, for all $\al \in S_4$.

Since $S_4$ is rigid relative to its automorphism group and since every automorphism of $S_4$ is inner, by Lemma \ref{lem.finite-group-rigid.two}, we find a group isomorphism $\phi : S_4 \to \Lambda_1$ such that $\be_\al \circ \phi = \phi \circ (\Ad \al)$ for all $\al \in S_4$. Whenever $a,b,c \in \{1,2,3,4\}$ are distinct, we denote $\si(abc) = \phi(abc) \in \Lambda$. Since $\phi$ is a group homomorphism, the relations in Lemma \ref{lem.better-generators.2} hold whenever $a,b,c,d \in \{1,2,3,4\}$ are distinct.

We claim that $\be_\al(\si(abc)) = \si(abc)$ whenever $\al \in S_6$ fixes $a,b,c$ pointwise and $a,b,c \in \{1,2,3,4\}$ are distinct. By symmetry, it suffices to consider $\si(123)$. By construction, $\si(123)$ is fixed by $\be_\al$ when $\al = (123)$ and when $\al = (56)$. So, $\pi^{-1}(\si(123)) \in A_6$ commutes with $(123)$ and with $(56)$. This forces $\pi^{-1}(\si(123)) \in \{e,(123),(321)\}$, so that $\pi^{-1}(\si(123))$ commutes with every $\al \in S_6$ that fixes $1,2,3$ pointwise. From this, the claim follows.

When $\al \in S_6$ globally preserves $\{1,2,3,4\}$, either $\al \in S_4$, or $\al$ is the composition of an element of $S_4$ with $(56)$. So by construction, $\be_\al(\si(abc)) = \si(\al(a)\al(b)\al(c))$. Together with the claim from the previous paragraph, we can thus unambiguously define $\si(abc) \in \Lambda$ for all distinct $a,b,c \in \{1,2,3,4,5,6\}$ such that $\be_\al(\si(abc)) = \si(\al(a)\al(b)\al(c))$ for all $\al \in S_6$ and all distinct $a,b,c$. Whenever $a,b,c,d \in \{1,2,3,4,5,6\}$ are distinct, we can choose $\si \in S_6$ such that $\si(1) = a$, $\si(2) = b$, $\si(3) = c$ and $\si(4) = d$. Since the relations in Lemma \ref{lem.better-generators.2} hold for distinct elements of $\{1,2,3,4\}$ and since $\be_\si$ is an automorphism of $\Lambda$, they thus hold for $a,b,c,d$.

By Lemma \ref{lem.better-generators} and the simplicity of $A_6$, we find a faithful group homomorphism $A_6 \to \Lambda$. Since $|\Lambda| = |A_6|$, it follows that $\Lambda \cong A_6$.

{\bf\boldmath The group $A_n$ for $n \geq 7$.} We note that $A_4 < A_n$ is precisely the centralizer of the permutation group of $\{5,6,\ldots,n\}$. We then reason in the same way as for $A_6$.

{\bf\boldmath The group $S_n$ for $n \geq 6$.} We note that $S_3 < S_n$ is precisely the centralizer of the permutation group of $\{4,5,\ldots,n\}$. We reason in the same way as for $A_6$, using the rigidity of $S_3$ and using the relations \eqref{eq.rel-Sn} for $S_n$, which only involve three elements $a,b,c$ of $\{1,\ldots,n\}$. By construction, the resulting homomorphism $\phi : S_n \to \Lambda$ is faithful on $S_3 < S_n$ and thus faithful on $S_n$.

{\bf\boldmath The group $\Atil_n$ for $n \geq 7$.} The group $S_n$ acts by automorphisms on $\Atil_n$. Then $\Atil_4 < \Atil_n$ precisely is the fixed point subgroup of the permutation group of $\{5,6,\ldots,n\}$. We again reason in the same way as for $A_6$, now using the rigidity of $\Atil_4$ and the relations in Lemma \ref{lem.better-generators.1} for $\Atil_n$. By construction, the resulting homomorphism $\phi : \Atil_n \to \Lambda$ is faithful on $\Atil_4 < \Atil_n$ and thus faithful on $\Atil_n$.
\end{proof}

\begin{lemma}\label{lem.better-generators}
Let $n \geq 4$ be an integer.
\begin{enumlist}
\item\label{lem.better-generators.1} Via the map $s(abc) \mapsto [ab][bc]$ and $z \mapsto z$, the group $\Atil_n$ is isomorphic to the universal group with generators $z$ and $s(abc)$, for all distinct $a,b,c \in \{1,\ldots,n\}$, and relations
\begin{align*}
& z \;\;\text{is central,}\quad z^2 = e \;\; , \quad s(abc) = s(bca) \;\; , \quad s(abc)^3 = e \;\; ,\\
& s(abc)^{-1} = s(cba) \;\; , \quad s(cbd) s(bad) s(abc) = z \;\; \text{for all distinct $a,b,c,d$.}
\end{align*}
The canonical central extension $0 \to \Z/2\Z \to \Atil_n \to A_n \to e$ is not split. If $n \geq 5$, the group $\Atil_n$ is perfect.
\item\label{lem.better-generators.2} Via the map $\si(abc) \mapsto (ab)(bc)$, the group $A_n$ is isomorphic to the universal group with generators $\si(abc)$, for all distinct $a,b,c \in \{1,\ldots,n\}$, and relations
$$\si(abc) = \si(bca) \;\; , \quad \si(abc)^3 = e \;\; , \quad \si(abc)^{-1} = \si(cba) \;\; , \quad \si(cbd) \si(bad) \si(abc) = e \;\; ,$$
for all distinct $a,b,c,d$.
\end{enumlist}
\end{lemma}
\begin{proof}
We first deduce a few extra relations. Denote $I = \{1,\ldots,n\}$. Let $a,b,c \in I$ be distinct. Consider the permutation $\phi$ of $I$ given by the cycle $(abc)$. We prove that
\begin{equation}\label{eq.conjugacy}
s(abc) s(xyz) s(abc)^{-1} = s(\phi(x)\phi(y)\phi(z)) \quad\text{for all distinct $x,y,z \in I$.}
\end{equation}
When the sets $\{a,b,c\}$ and $\{x,y,z\}$ coincide, \eqref{eq.conjugacy} follows from the cyclic invariance relation $s(xyz) = s(yzx) = s(zxy)$ for all distinct $x,y,z \in I$. When the intersection of $\{a,b,c\}$ and $\{x,y,z\}$ consists of two points, after cyclically permuting $a,b,c$ and $x,y,z$, it suffices to consider the cases $(xyz) = (abd)$ and $(xyz) = (bad)$, with $a,b,c,d$ distinct. Since $s(abd)=s(bad)^{-1}$, it actually suffices to consider $(xyz) = (bad)$. Then the required relation $s(abc) s(bad) s(abc)^{-1} = s(cbd)$ follows because by our given relation,
$$s(abc) s(bad) = z s(adc) = z s(cad) \quad\text{and}\quad s(cbd) s(abc) = s(cbd) s(bca) = z s(cad) \; .$$
To prove \eqref{eq.conjugacy} when the sets $\{a,b,c\}$ and $\{x,y,z\}$ have one point in common, after a cyclic permutation, we may assume that $(xyz) = (c d f)$ with $a,b,c,d,f$ distinct. We need to compute the conjugation of $s(cdf)$ by $s(abc)$. Since $s(cdf) = z s(caf) s(acd)$ and since we already know how to conjugate $s(caf)$ and $s(acd)$ by $s(abc)$, because they all have the points $a,c$ in common, we get that
$$s(abc) s(cdf) s(abc)^{-1} = z s(abf) s(bad) = s(adf) \; ,$$
so that again \eqref{eq.conjugacy} follows. When finally $\{a,b,c\}$ and $\{x,y,z\}$ are disjoint, we write $s(xyz) = z s(xaz)s(axy)$. We know how to conjugate $s(xaz)$ and $s(axy)$ by $s(abc)$ and get that
$$s(abc) s(xyz) s(abc)^{-1} = z s(xbz)s(bxy) = s(xyz) \; .$$
So, \eqref{eq.conjugacy} holds in full generality.

Denote by $\Gtil_n$ and $G_n$ the two universal groups defined in the lemma. We have the natural homomorphism $\Gtil_n \to G_n$ given by $z \mapsto e$. By construction, the kernel is $\{e,z\}$. We also have the natural surjective homomorphisms $\Gtil_n \to \Atil_n$ and $G_n \to A_n$ as stated in the lemma. All these arrows commute and it thus suffices to prove that $G_n \to A_n$ is faithful.

In \cite[Theorem 1.2 in Chapter 12]{Kar93}, a presentation of $A_n$ is given. It thus suffices to check that the relations of this presentation hold in $G_n$. So, we define $t_1 = \si(123)$ and for all $2 \leq i \leq n-2$, $t_i = \si(1,i+1,2)\si(i+1,i+2,1)$. Since $t_1^3 = e$, it remains to check that
\begin{align*}
& t_i^2 = e \quad\text{and}\quad (t_{i-1} t_i)^3 = e \;\;\text{if $2 \leq i \leq n-2$,}\\
& (t_j t_k)^2=e \;\;\text{if $1 \leq j \leq k-2$ and $k \leq n-2$.}
\end{align*}
Each of these follow easily from our defining relations and \eqref{eq.conjugacy}.

If $\theta : A_n \to \Atil_n$ would be a splitting homomorphism, we must have $\theta(\si(abc)) = s(abc)$, because both elements have order $3$, while $z$ has order $2$. But this is incompatible with the remaining relation.

Assume that $n \geq 5$. To see that $\Atil_n$ is perfect, assume that $\om : K \to \T$ is a group homomorphism. Take distinct $a,b,c \in \{1,\ldots,n\}$ and write $\om_0 := \om(s(abc))$. Since $s(abc)^3= e$, $\om_0^3 = 1$. Since $n \geq 5$, we can pick distinct $d,x$ that also differ from $a,b,c$. By \eqref{eq.conjugacy}, $s(c d x) s(abc) s(cdx)^{-1} = s(abd)$, so that $\om(s(abd)) = \om_0$. Taking the inverse, $\om(s(bad)) = \om_0^{-1}$. Similarly, $s(adx) s(abc) s(adx)^{-1} = s(dbc)$, so that $s(cbd) = \om_0^{-1}$. Since $s(cbd) s(bad) s(abc) = z$, we conclude that $\om(z) = \om_0^{-1}$. Taking the square, $\om_0^{-2} = 1$. Since $\om_0^3 = 1$, we conclude that $\om_0=1$. Since $a,b,c$ were arbitrary, $\om = 1$.
\end{proof}

The following lemma gathers a few well known exceptional isomorphisms.

\begin{lemma}\label{lem.exceptional-iso}
We have $S_3 \cong \SL_2(\F_2)$. When $n=4$, resp.\ $n=5$, the sequence $\Atil_n \to A_n \hookrightarrow S_n$ is isomorphic with $\SL_2(\F_p) \to \PSL_2(\F_p) \hookrightarrow \PGL_2(\F_p)$ for $p=3$, resp.\ $p=5$.
\end{lemma}
\begin{proof}
For the isomorphisms $S_3 \cong \SL_2(\F_2)$, $A_4 \cong \PSL_2(\F_3)$ and $A_5 \cong \PSL_2(\F_5)$, we refer to \cite[Theorem 2.5 in Chapter 16]{Kar93}. Taking the automorphism groups, we get that for $n=4,5$, $A_n < S_n$ is isomorphic with $\PSL_2(\F_p) < \PGL_2(\F_p)$ for $p=3,5$.

The isomorphism $A_4 \cong \PSL_2(\F_3)$ is very explicit, because the projective line on which $\PSL_2(\F_3)$ acts has $4$ points. This immediately lifts to an isomorphism between $\Atil_4 \to A_4$ and $\SL_2(\F_3) \to \PSL_2(\F_3)$, determined by $s(123) \mapsto \bigl(\begin{smallmatrix} 1 & 1 \\ 0 & 1\end{smallmatrix}\bigr)$, $s(124) \mapsto \bigl(\begin{smallmatrix} 1 & 0 \\ 1 & 1\end{smallmatrix}\bigr)$ and $z \mapsto -I_2$.

By \cite[Theorem 3.2 in Chapter 12]{Kar93}, the Schur multiplier $M(A_5)$ has order $2$ (see discussion before Lemma \ref{lem.trivial-2-cohom-examples}) and $A_5$ is perfect. By e.g.\ \cite[Corollary 7.8 in Chapter 11]{Kar93}, it follows that there is a universal central extension $0 \to \Z/2\Z \to G \to A_5 \to e$. So all non split central extensions of $A_5$ by $\Z/2\Z$ are isomorphic, which implies that $\Atil_5 \to A_5$ is isomorphic with $\SL_2(\F_5) \to \PSL_2(\F_5)$.
\end{proof}

We now prove relative rigidity for some of the finite linear groups.

\begin{theorem}\label{thm.rigid-SLn-Fp}
Let $q=p^k$ be a prime power and $\F_q$ the unique field of order $q$.
\begin{enumlist}
\item\label{thm.rigid-SLn-Fp.one} The additive group $K = \F_q$ is rigid relative to $\F_q^\times < \Aut K$.
\item\label{thm.rigid-SLn-Fp.two} If $n \geq 2$, the additive group $K = \F_q^n$ is strictly rigid relative to $\GL_n(\F_q) < \Aut K$.
\item\label{thm.rigid-SLn-Fp.three} If $n \geq 3$, the groups $K = \SL_n(\F_q)$ and $K = \PSL_n(\F_q)$ are rigid relative to $\Aut K$.
\item\label{thm.rigid-SLn-Fp.four} If $n \geq 2$, the group $K = \F_q^n \rtimes \SL_n(\F_q)$ is rigid relative to $\Aut K$.
\end{enumlist}
\end{theorem}

Before proving Theorem \ref{thm.rigid-SLn-Fp}, we gather a few background results and notations that will be used in our approach to the four statements of the theorem.

{\bf Fixed point subgroups.} Given a finite group $K$, we say that a subgroup $K_0 < K$ is the \emph{fixed point subgroup} of $\cG_0 < \Aut K$ if $K_0 = \{k \in K \mid \forall \al \in \cG_0 : \al(k) = k\}$.
Throughout the proof, we use the following straightforward observation. If $\Lambda$ is another finite group and $\pi : K \to \Lambda$ a bijection such that $\be_\al := \pi \circ \al \circ \pi^{-1} \in \Aut \Lambda$ for all $\al \in \cG_0$, then the image $\pi(K_0)$ of the fixed point subgroup $K_0$ of $\cG_0$ is a subgroup of $\Lambda$. Indeed, $\pi(K_0) = \{s \in \Lambda \mid \forall \al \in \cG_0 : \be_\al(s) = s\}$, which is a subgroup of $\Lambda$.

{\bf\boldmath Automorphism groups of $\SL_n(\F_q)$, $\SL_n(\F_q) / D$ and $\PSL_n(\F_q)$.} By \cite[Satz 1]{SVdW28}, for every integer $n \geq 2$, the automorphism group of $\PSL_n(\F_q)$ is generated by the automorphisms of the form $\Ad A_0$ with $A_0 \in \PGL_n(\F_q)$, the automorphisms given by applying a field automorphism $\zeta \in \Autfield \F_q$ to every component of a matrix, and the automorphism $\delta : A \mapsto (A^T)^{-1}$, which is only needed when $n \geq 3$ because it is inner when $n=2$.

Write $C = \{a \in \F_q^\times \mid a^n = 1\}$. Then $\cZ := \{a I_n \mid a \in C\}$ is the center of $\SL_n(\F_q)$. Every automorphism $\al$ of $\SL_n(\F_q)$ globally preserves the center $\cZ$ and thus induces an automorphism $\eta(\al)$ of $\PSL_n(\F_q)$. If $\eta(\al) = \id$, we find that $\al(A) = \om(A) A$ for all $A \in \SL_n(\F_q)$, where $\om : \SL_n(\F_q) \to C$ is a group homomorphism. By e.g.\ \cite[Theorem 2.3 in Chapter 16]{Kar93}, the group $\SL_n(\F_q)$ is generated by the elementary matrices, which have order $p$, while the order of $C$ divides $q-1$. So, $\om = 1$ and $\al = \id$. The homomorphism $\eta : \Aut \SL_n(\F_q) \to \Aut \PSL_n(\F_q)$ is thus faithful. From the above description of the automorphisms of $\PSL_n(\F_q)$, it follows that $\eta$ is also surjective. We conclude that the automorphism group of $\SL_n(\F_q)$ has exactly the same description as the one of $\PSL_n(\F_q)$.

When $D < \cZ$ is a subgroup, we similarly define $\eta_0 : \Aut \SL_n(\F_q)/D \to \Aut \PSL_n(\F_q)$, which is a faithful homomorphism. Making a similar analysis as in the previous paragraph, an automorphism $\al$ of $\PSL_n(\F_q)$ belongs to the image of $\eta_0$ if its unique lift to an automorphism $\altil$ of $\SL_n(\F_q)$ satisfies $\altil(D) = D$. It follows that $\Aut \SL_n(\F_q)/D$ is generated by $\Ad A_0$ with $A_0 \in \PGL_n(\F_q)$, field automorphisms $\zeta \in \Autfield \F_q$ with $\zeta(D) = D$, and the automorphism $A \mapsto (A^T)^{-1}$.

Finally observe that it follows that $\Ad \PSL_n(\F_q)$ is a normal subgroup of $\Aut \SL_n(\F_q)/D$ and that the quotient is a subgroup of $\F_q^\times/M \rtimes \Autfield(\F_q) \times \Z/2\Z$, where $M = \{a^n \mid a \in \F_q^\times\}$ is the subgroup of $n$-th powers. Since $\Autfield(\F_q) \cong \Z/k\Z$, this quotient is thus solvable.

Throughout the proof of Theorem \ref{thm.rigid-SLn-Fp}, we denote by $E_{ij}$ the matrix with $1$ in position $(i,j)$ and $0$ elsewhere. We denote by $I_n$ the $n \times n$ identity matrix. We denote by $e_i$ the standard basis elements of $\F_q^n$.

\begin{proof}[{\bf\boldmath Proof of Theorem \ref{thm.rigid-SLn-Fp.one}: the groups $\F_q$.}]
Assume that $\Lambda$ is a finite group and $\pi : \F_q \to \Lambda$ is a bijection such that $\pi(\al x) = \be_\al(\pi(x))$ for all $\al \in \F_q^\times$ and $x \in \F_q$, where $\be_\al \in \Aut \Lambda$. Since $|\Lambda| = |\F_q| = p^k$, $\Lambda$ is a finite $p$-group, so that the center $\cZ(\Lambda)$ is nontrivial. Since the action $(\be_\al)_{\al \in \F_q^\times}$ is transitive on $\Lambda \setminus \{e\}$, it follows that $\cZ(\Lambda) = \Lambda$, so that $\Lambda$ is abelian and of order $q$.

Since $\Lambda$ is abelian, we write the group operation in $\Lambda$ additively and consider the ring $\End \Lambda$ of group homomorphisms $\Lambda \to \Lambda$. Since the action $(\be_\al)_{\al \in \F_q^\times}$ is transitive on $\Lambda \setminus \{0\}$ and since $|\Lambda \setminus \{0\}| = |\F_q^\times|$, the action is also free.

We claim that the subset $\F' := \{0\} \cup \{\be_\al \mid \al \in \F_q^\times\}$ of $\End \Lambda$ is a subfield of $\End \Lambda$. The only nontrivial point is to prove that $\be_\al + \be_{\al'} \in \F'$ for all $\al,\al' \in \F_q^\times$. Fix a nonzero element $s_1 \in \Lambda \setminus \{0\}$. If $(\be_\al + \be_{\al'})(s_1)=0$, we apply $\be_\gamma$ for $\gamma \in \F_q^\times$, use that $\F_q^\times$ is commutative and that the action $(\be_\al)_{\al \in \F_q^\times}$ is transitive on $\Lambda \setminus \{0\}$ to conclude that $\be_\al + \be_{\al'} = 0 \in \F'$. If $(\be_\al + \be_{\al'})(s_1) \neq 0$, by transitivity, we can choose $\al\dpr \in \F_q^\times$ such that $(\be_\al + \be_{\al'})(s_1) = \be_{\al\dpr}(s_1)$. By the same reasoning as above, $\be_\al + \be_{\al'} = \be_{\al\dpr} \in \F'$.

By freeness of the action, $|\F'| = q$. Since there is only one field of order $q$, we can choose a field isomorphism $\phi : \F_q \to \F'$. Then, $\phi(\al) = \be_{\zeta(\al)}$ for all $\al \in \F_q^\times$, where $\zeta \in \Aut(\F_q^\times,\cdot)$. Define the bijection $\pi_0 : \F_q \to \Lambda$ by $\pi_0(0) = 0$ and $\pi_0(\al) = \phi(\al)(s_1) = \be_{\zeta(\al)}(s_1)$. By construction, $\pi_0$ is a group isomorphism between $(\F_q,+)$ and $\Lambda$ and $\pi_0(\al x) = \be_{\zeta(\al)}(\pi_0(x))$ for all $\al \in \F_q^\times$ and $x \in \F_q$. We have thus proven that $(\F_q,+)$ is rigid relative to $\F_q^\times < \Aut(\F_q,+)$.
\end{proof}

\begin{proof}[{\bf\boldmath Proof of Theorem \ref{thm.rigid-SLn-Fp.two}: the groups $\F_q^n$ with $n \geq 2$.}]
We put $K = (\F_q^n,+)$, let $\Lambda$ be a group and $\pi : K \to \Lambda$ a bijection such that $\be_A := \pi \circ A \circ \pi^{-1}$ is a group automorphism of $\Lambda$ for every $A \in \GL_n(\F_q)$. We prove that $\pi$ is a group isomorphism.

Since $|\Lambda| = |K| = p^{kn}$ and since the action of $\SL_n(\F_q)$ is transitive on $K \setminus \{e\}$, by the same argument as in the beginning of the proof of (i), we get that $\Lambda$ is abelian. For every subset $I \subset \{1,\ldots,n\}$, we have the natural subgroup $\F_q^I \subset \F_q^n$. For every $I$, the subgroup $\F_q^I$ is a fixed point subgroup, under the automorphisms $I_n + E_{1i} \in \GL_n(\F_q)$ for all $i \not\in I$. We thus find the subgroups $\Lambda_I < \Lambda$ such that $\pi(\F_q^I) = \Lambda_I$. In particular, we write $\Lambda_i := \Lambda_{\{i\}}$ and get that $\pi(\F_q e_i) = \Lambda_i$.

For every $\al \in \F_q^\times$ and $i \in \{1,\ldots,n\}$, define $A_i(\al) \in \GL_n(\F_q)$ by $A_i(\al) = \al E_{ii} + (I_n - E_{ii})$. Define $\be_{i,\al} \in \Aut \Lambda$ by $\be_{i,\al} = \pi \circ A_i(\al) \circ \pi^{-1}$. Since $A_1(\al)(\F_q e_1) = \F_q e_1$, we get that $\be_{1,\al}$ restricts to an automorphism of $\Lambda_1$, for all $\al \in \F_q^\times$. We can then apply (i) to the restriction of $\pi$ to $\F_q e_1$. There thus exists a group isomorphism $\phi_1 : \F_q \to \Lambda_1$ and $\zeta \in \Aut(\F_q^\times,\cdot)$ such that $\phi_1(\zeta(\al) x) = \be_{1,\al}(\phi_1(x))$ for all $x \in \F_q$ and $\al \in \F_q^\times$.

We view any permutation $\si$ of $\{1,\ldots,n\}$ as a permutation matrix in $\GL_n(\F_q)$ and consider $\be_\si = \pi \circ \si \circ \pi^{-1} \in \Aut \Lambda$. Since $\si(\F_q^I) = \F_q^{\si(I)}$, we get that $\be_\si(\Lambda_I) = \Lambda_{\si(I)}$. In particular, we denote for $i \geq 2$ by $\si_i$ the flip of $1$ and $i$, so that $\Lambda_i = \be_{\si_i}(\Lambda_1)$. Define $\phi_i = \be_{\si_i} \circ \phi_1$. Since $\Lambda$ is abelian and $\Lambda_I \cap \Lambda_J = \Lambda_{I \cap J}$, it follows that $\phi : \F_q^n \to \Lambda : \phi(x) = \phi_1(x_1) \phi_2(x_2) \cdots \phi_n(x_n)$ is a group isomorphism and $\phi(\F_q^I) = \Lambda_I$ for all $I \subset \{1,\ldots,n\}$.

Define the permutation $\phi_0$ of $\F_q$ such that $\pi(x e_1) = \phi_1(\phi_0(x))$ for all $x \in \F_q$. Since $0$ is the only element of $\F_q^n$ that is fixed under all the automorphisms of $\GL_n(\F_q)$, we have that $\phi_0(0) = 0$. By construction, $\phi_0(\al a) = \zeta(\al) \phi_0(a)$ for all $\al \in \F_q^\times$ and $a \in \F_q$.

For the rest of the proof of (i), we replace a few times $\pi$ with its composition with a group isomorphism, up to the point where $\pi$ becomes the identity homomorphism. At that moment, it is proven that the initial $\pi$ was a group isomorphism. As a first step, we replace $\pi$ by $\phi^{-1} \circ \pi$. So from now on, $\Lambda = K$ and the bijection $\pi$ satisfies $\pi(a e_i) = \phi_0(a) e_i$ for all $i \in \{1,\ldots,n\}$ and $a \in \F_q$. We also get that $\pi(\F_q^I) = \F_q^I$ for every $I \subset \{1,\ldots,n\}$.

So after this replacement, $\be_{i,\al}$ is an automorphism of $\F_q^n$ that satisfies $\be_{i,\al}(b e_j) = b e_j$ for all $b \in \F_q$ if $j \neq i$, and $\be_{i,\al}(b e_i) = \zeta(\al) b e_i$ for all $b \in \F_q$. This means that $\be_{i,\al} = A_i(\zeta(\al))$. With a similar reasoning, we find that $\be_{\si_i} = \si_i$.

We now determine the automorphisms $\be_x = \pi \circ (I_n + x E_{12}) \circ \pi^{-1}$ whenever $x \in \F_q$. Since $(I_n + x E_{12})(b e_i) = b e_i$ when $i \neq 2$ and $\pi(b e_i) = \phi_0(b) e_i$, also $\be_x(d e_i) = d e_i$ whenever $i \neq 2$ and $d \in \F_q$. Write $W = \F_q e_1 + \F_q e_2$. Since $(I_n + x E_{12})(W) = W$ and $\pi(W) = W$, also $\be_x(W) = W$. Since $\be_x$ is an automorphism of $(\F_q^n,+)$, we thus find $A_x \in \Aut(\F_q,+)$ and $B_x \in \End(\F_q,+)$ such that $\be_x(b e_2) = B_x(b) e_1 + A_x(b) e_2$ for all $b \in \F_q$.

For $\al \in \F_q^\times$, write $T_\al = \al(E_{11} + E_{22}) + (I_n - E_{11} - E_{22})$ and denote $\gamma_\al = \pi \circ T_\al \circ \pi^{-1}$. In the same way as with $\be_{i,\al}$, we get that $\gamma_\al = T_{\zeta(\al)}$ for all $\al \in \F_q^\times$. Since $T_\al$ commutes with $I_n + x E_{12}$, it follows that $\gamma_\al$ commutes $\be_x$. So, $\be_x$ commutes with $T_{\zeta(\al)}$ for all $\al \in \F_q^\times$. We conclude that $B_x(\al b) = \al B_x(b)$ and $A_x(\al b) = \al A_x(b)$ for all $\al \in \F_q^\times$ and $x,b \in \F_q$. This means that $A_x \in \F_q^\times$ and $B_x \in \F_q$, viewed as multiplication homomorphisms on $\F_q$.

Since $I_n + x E_{12}$ has order $p$, also $\be_x^p = \id$. It follows that $A_x^p = 1$ and thus $A_x^q = 1$. But $A_x^q = A_x$, so that $A_x = 1$. Since $(I_n + x E_{12})(I_n + y E_{12}) = I_n + (x+y)E_{12}$, also $\be_x \circ \be_y = \be_{x+y}$. It follows that $B_{x+y} = B_x + B_y$ for all $x,y \in \F_q$.

Finally note that $(I_n + x E_{12}) A_2(\al) = A_2(\al) (I_n + \al x E_{12})$, so that $\be_x \circ \be_{2,\al} = \be_{2,\al} \circ \be_{\al x}$ for all $\al \in \F_q^\times$ and $x \in \F_q$. Since $\be_{2,\al} = A_2(\zeta(\al))$, we conclude that $B_{\al x} = \zeta(\al) B_x$. When $x = 1$, we have that $\be_1 \neq \id$, so that $B_1 \neq 0$. We find that $\zeta(\al) = B_1^{-1} B_\al$ for all $\al \in \F_q^\times$. Since $B_{x+y} = B_x + B_y$ for all $x,y \in \F_q$, it thus follows that $\zeta$ is the restriction to $\F_q^\times$ of a field automorphism of $\F_q$ that we still denote by $\zeta$.

Replacing $\pi$ by $(\zeta \times \cdots \times \zeta)^{-1} \circ \pi$, we may assume that $\zeta = \id$. Since $\phi_0(\al a) = \zeta(\al) \phi_0(a)$, we conclude that $\phi_0(a) = a c_0$ for all $a \in \F_q$, and some $c_0 \in \F_q^\times$. We also get that $B_x = b_0 x$ for all $x \in \F_q$, where $b_0 = B_1 \in \F_q^\times$. So, $\pi \circ (I_n + x E_{12}) \circ \pi^{-1} = \be_x = I_n + b_0 x E_{12}$ for all $x \in \F_q$. Conjugating with $\be_{\si_2} = \si_2$, we also get that $\pi \circ (I_n + x E_{21}) \circ \pi^{-1} = I_n + b_0 x E_{21}$.

Define $\eta \in \GL_n(\F_q)$ by $\eta = E_{12} - E_{21} + (I_n - E_{11} - E_{22})$. Since $\pi(a e_i) = a c_0 e_i$, it follows that $\be_\eta = \pi \circ \eta \circ \pi^{-1}$ equals $\eta$. A direct computation shows that $(I_n + E_{12}) (I_n - E_{21}) (I_n + E_{12}) = \eta$. Conjugating with $\pi$, it follows that $(I_n + b_0 E_{12}) (I_n - b_0 E_{21}) (I_n + b_0 E_{12}) = \eta$. This is only possible if $b_0 = 1$. We have thus proven that
\begin{equation}\label{eq.good-be-x}
\pi \circ (I_n + x E_{12}) \circ \pi^{-1} = I_n + x E_{12} \quad\text{for all $x \in \F_q$.}
\end{equation}
Since $\be_{\si_i} = \si_i$ for every $i$, we get that $\pi \circ \si \circ \pi^{-1} = \si$ for every permutation matrix $\si \in \GL_n(\F_q)$. Whenever $i \neq j$, we can choose a permutation $\si$ such that $\si(1) = i$ and $\si(2) = j$. Conjugating \eqref{eq.good-be-x} with $\be_\si = \si$, it follows that $\pi \circ (I_n + x E_{ij}) \circ \pi^{-1} = I_n + x E_{ij}$ whenever $i \neq j$. Since the elementary matrices generate $\SL_n(\F_q)$, it follows that $\pi \circ A = A \circ \pi$ for all $A \in \SL_n(\F_q)$.

We have seen above that $\pi(a e_i) = a c_0 e_i$. Replacing $\pi(\cdot)$ by $c_0^{-1} \pi(\cdot)$, we get moreover that $\pi(e_1) = e_1$. So, $\pi(A \cdot e_1) = A \cdot \pi(e_1) = A \cdot e_1$ for all $A \in \SL_n(\F_q)$. Since $\SL_n(\F_q)$ acts transitively on $\F_q^n \setminus \{0\}$, we have proven that $\pi = \id$. So (ii) is proven.
\end{proof}

\begin{proof}[{\bf\boldmath Proof of Theorem \ref{thm.rigid-SLn-Fp.three}: the groups $\SL_n(\F_q)$ and $\PSL_n(\F_q)$.}]
Fix $n \geq 3$. In the first steps of the proof, we consider the case $k=1$, i.e.\ $K = \SL_n(\F_p)$ where $p$ is prime. Let $\Lambda$ be a finite group and $\pi : K \to \Lambda$ a bijection such that $\be_\al := \pi \circ \al \circ \pi^{-1} \in \Aut \Lambda$ for all $\al \in \Aut K$.

Define the subgroup $C < \F_p^\times$ by $C = \{a \in \F_p^\times \mid a^n = 1\}$. Note that the center of $K$ is given by $\cZ(K) = \{a I_n \mid a \in C\}$. Being the fixed points of $\Ad A$ for all $A \in \GL_n(\F_p)$, we see that $\cZ(K)$ is a fixed point subgroup, so that $\Lambda_\cZ := \pi(\cZ(K))$ is a subgroup of $\Lambda$. Write $c := |C|$ and note that $c \mid p-1$.

Note that $K_{12} = \{a I_n + x E_{12} \mid a \in C, x \in \F_p\}$ is the fixed point subgroup of $\Ad (I_n + E_{ij})$ with $i \neq 2$, $j \neq 1$ and $i \neq j$. So $\pi(K_{12})$ is a subgroup of order $cp$ of $\Lambda$. Since the number $n_p$ of $p$-Sylow subgroups of $\pi(K_{12})$ is congruent to $1$ modulo $p$ and divides $c$, which divides $p-1$, we get that $n_p = 1$. So, $\pi(K_{12})$ has a unique $p$-Sylow subgroup $\Lambda_{12}$, which is normal in $\pi(K_{12})$ and has order $p$. So, $\Lambda_{12} \cong \F_p$. Also, every element of order $p$ in $\pi(K_{12})$ belongs to $\Lambda_{12}$. Choose a generator $s_{12} \in \Lambda_{12}$. Note that $s_{12}$ is an element of order $p$.

Whenever $i \neq j$, we similarly define the fixed point subgroup $K_{ij} = \{a I_n + x E_{ij} \mid a \in C, x \in \F_p\}$. Whenever $\si$ is a permutation of $\{1,\ldots,n\}$, we view $\si$ as a permutation matrix in $\GL_n(\F_p)$ and consider the automorphism $\be_\si = \pi \circ (\Ad \si) \circ \pi^{-1}$ of $\Lambda$. Since $\si K_{ij} \si^{-1} = K_{\si(i)\si(j)}$, we find that $\be_\si(\pi(K_{ij})) = \pi(K_{\si(i)\si(j)})$.

If $\si(1) = 1$ and $\si(2) = 2$, $\si$ commutes with $K_{12}$, so that $\be_\si$ acts as the identity on $\pi(K_{12})$. We can thus unambiguously define generators $s_{ij}$ for the unique $p$-Sylow subgroup $\Lambda_{ij}$ of $\pi(K_{ij})$ such that $\be_\si(s_{ij}) = s_{\si(i)\si(j)}$ for all $i \neq j$. The main result in the first part of the proof of (iii) is to show that for some $f \in \F_p^\times$, the elements $s_{ij}^f$ satisfy the defining relations of $\SL_n(\F_p)$, i.e.\ the same relations as the elementary matrices $I_n + E_{ij}$.

{\bf\boldmath Commutation of $s_{ik}$ and $s_{jk}$ and commutation of $s_{ij}$ and $s_{ik}$.} Define $K_1 = \{a I_n + x E_{13} + y E_{23} \mid a \in C , x,y \in \F_p\}$. Note that $K_1$ is the fixed point subgroup of $\Ad (I_n + E_{ij})$ with $i \neq 3$ and $j \not\in \{1,2\}$. So, $\pi(K_1)$ is a subgroup of $\Lambda$ of order $c p^2$. By the same reasoning as above, $\pi(K_1)$ has a unique $p$-Sylow subgroup $\Lambda_1$, which is normal in $\pi(K_1)$ and contains all elements of order $p$ of $\pi(K_1)$. Since $\Lambda_1$ is a group of order $p^2$, $\Lambda_1$ is abelian. Since $K_{13}$ and $K_{23}$ are subgroups of $K_1$ and $s_{13}, s_{23}$ have order $p$, it follows that $s_{13},s_{23} \in \Lambda_1$. In particular, $s_{13}$ commutes with $s_{23}$. Applying an arbitrary $\be_\si$, we conclude that $s_{ik}$ commutes with $s_{jk}$ whenever $i \neq k$ and $j \neq k$. For later use, we also note that, because $s_{23} \not\in \pi(K_{13})$, the elements $s_{13}$ and $s_{23}$ generate $\Lambda_1$.

By symmetry, we also find that $s_{ij}$ commutes with $s_{ik}$ whenever $i \neq j$ and $i \neq k$.

{\bf\boldmath Commutation of $s_{ij}$ and $s_{rt}$ if $i,j,r,t$ are distinct.} When $n \geq 4$, we prove that $s_{13}$ commutes with $s_{24}$. This is slightly more delicate, especially when $p=2$. For every $a \in C$ and $X \in \F_p^{2 \times 2}$, define the matrix
$$\psi_a(X) = \begin{pmatrix}\begin{array}{c|c|c} a I_2 & X & 0 \\ \hline 0 & a I_2 & 0 \\ \hline 0 & 0 & a I_{n-4}\end{array}\end{pmatrix} \in \SL_n(\F_p) \; .$$
Note that $S := \{\psi_a(X) \mid a \in C, X \in \F_p^{2 \times 2}\}$ is the fixed point subgroup of $\Ad (I_n + E_{ij})$ with $i \not\in \{3,4\}$ and $j \not\in \{1,2\}$. Then $\pi(S)$ is a subgroup of $\Lambda$ of order $c p^4$. By the same reasoning as above, it has a unique $p$-Sylow subgroup $T < \pi(S)$, which is normal in $\pi(S)$, contains all elements of order $p$ in $\pi(S)$, and has order $p^4$.

For all $A,B \in \GL_2(\F_p)$, define $\be_{A,B} \in \Aut \Lambda$ by $\be_{A,B} = \pi \circ \Ad(A \oplus B \oplus I_{n-4}) \circ \pi^{-1}$. Note that $\Ad(A \oplus B \oplus I_{n-4})(\psi_a(X)) = \psi_a(A X B^{-1})$. So, $\be_{A,B}(\pi(S)) = \pi(S)$ and thus, $\be_{A,B}(T) = T$ for all $A,B \in \GL_2(\F_p)$.

The orbits of the action of $\GL_2(\F_p) \times \GL_2(\F_p)$ on $\F_p^{2 \times 2}$ by $(A,B) \cdot X = A X B^{-1}$ are $\{0\}$, the matrices of rank $1$, and $\GL_2(\F_p)$, which have respectively $1$, $(p-1)(p+1)^2$ and $p (p-1)^2 (p+1)$ elements. Since $T$ is a finite $p$-group, the center $\cZ(T)$ is nontrivial. So, the order of $\cZ(T)$ is $p^r$ with $r \in \{1,2,3,4\}$. Since $\be_{A,B}(T) = T$, also $\be_{A,B}(\cZ(T)) = \cZ(T)$ for all $A,B \in \GL_2(\F_p)$.

Since $T$ is a $p$-group and $|\pi(\cZ(K))| = c \mid p-1$, we have that $T \cap \pi(\cZ(K)) = \{e\}$. So all orbits of the action by $\be_{A,B}$ on $\cZ(T) \setminus \{e\}$ have size $(p-1)(p+1)^2$ or $p(p-1)^2 (p+1)$. The number of elements of $\cZ(T) \setminus \{e\}$ is $p^r-1$, which is not divisible by $p$ and thus, not divisible by $p(p-1)^2 (p+1)$. We conclude that there must be at least one orbit of size $(p-1)(p+1)^2$. We thus find $a_0 \in C$ such that for every rank $1$ matrix $X \in \F_p^{2 \times 2}$, we have $\pi(\psi_{a_0}(X)) \in \cZ(T) \setminus \{e\}$.

Taking $X_{11} = 1$ and $X_{ij} = 0$ when $(i,j) \neq (1,1)$, we find that $\pi(a_0 I_n + E_{13}) \in \cZ(T) \setminus \{e\}$. In particular, $\pi(a_0 I_n + E_{13})$ has order $p$ and also belongs to $\pi(K_{13})$, so that it generates $\Lambda_{13}$. We conclude that $\Lambda_{13} \subset \cZ(T)$. We similarly find that $\Lambda_{24} \subset \cZ(T)$. Since $\cZ(T)$ is abelian, we get that $s_{13}$ commutes with $s_{24}$. Applying an arbitrary $\be_\si$, it follows that $s_{ij}$ commutes with $s_{rt}$ whenever $i,j,r,t$ are all distinct.

{\bf\boldmath For some $f \in \F_p^\times$, we have $[s_{12}^f,s_{23}^f] = s_{13}^f$.} Note that
\begin{equation}\label{eq.subgroup-H}
H = \left\{ \begin{pmatrix}\begin{array}{ccc|c} a & x & z & \\ 0 & a^{-n+1} & y & 0 \\ 0 & 0 & a & \\ \hline & 0 & & a I_{n-3} \end{array}\end{pmatrix} \middle| a \in \F_p^\times, x,y,z \in \F_p \right\}
\end{equation}
is the fixed point subgroup of $\Ad (I_n + E_{ij})$ with $i \not\in\{2,3\}$ and $j\not\in\{1,2\}$. Then $\pi(H)$ is a subgroup of $\Lambda$ of order $(p-1)p^3$. By the same reasoning as before, $\pi(H)$ has a unique $p$-Sylow subgroup $\Lambda_H$, which is a normal subgroup of $\pi(H)$ of order $p^3$ that contains all elements of order $p$ of $\pi(H)$.

Since $K_{12}$, $K_{23}$ and $K_{13}$ are all subgroups of $H$, we thus find that $\Lambda_{12}$, $\Lambda_{23}$ and $\Lambda_{13}$ are all subgroups of $\Lambda_H$. Above we have seen that $\Lambda_{13}$ and $\Lambda_{23}$ commute and that they generate the unique $p$-Sylow subgroup $\Lambda_1$ of $\pi(K_1)$, where $K_1 = \{a I_n + z E_{13} + y E_{23} \mid a \in C , z,y \in \F_p\}$.

Since $\Lambda_{12} \cap \Lambda_1 \subset \pi(K_{12} \cap K_1) = \pi(\cZ(K))$ and $\pi(\cZ(K))$ has order $c \mid p-1$, we get that $\Lambda_{12} \cap \Lambda_1 = \{e\}$. Since $\Lambda_1$ has order $p^2$ and $\Lambda_H$ has order $p^3$, we conclude that $\Lambda_H$ is generated by $\Lambda_{12}$ and $\Lambda_1$.

We now prove that $\Lambda_1$ is a normal subgroup of $\Lambda_H$. Note that $\Lambda_{13}$ commutes with $\Lambda_{12}$, $\Lambda_{13}$ and $\Lambda_{23}$. So, $\Lambda_{13} \subset \cZ(\Lambda_H)$. It follows that $\Lambda_H / \Lambda_{13}$ is a group of order $p^2$, which is thus abelian. In particular, $\Lambda_1 / \Lambda_{13}$ is normal in $\Lambda_H / \Lambda_{13}$, so that $\Lambda_1$ is normal in $\Lambda_H$. We retain that $\Lambda_{12}$ normalizes $\Lambda_1$.

We claim that $\Lambda_{12}$ does not commute with $\Lambda_1$. Assume the contrary. Let $\si_0$ be the flip of $1$ and $2$. Note that $\si_0 K_1 \si_0^{-1} = K_1$, so that $\be_{\si_0}(\pi(K_1)) = \pi(K_1)$ and thus $\be_{\si_0}(\Lambda_1) = \Lambda_1$. Since we assumed that $\Lambda_{12}$ commutes with $\Lambda_1$, applying $\be_{\si_0}$ gives us that also $\Lambda_{21}$ commutes with $\Lambda_1$.

Define the subgroup $U = \{ A \oplus b I_{n-2} \mid A \in \GL_2(\F_p), b \in \F_p^\times, b^{n-2} \det A = 1 \}$ of $K$. When $p \neq 2$ or $n \geq 4$, note that $U$ is the fixed point subgroup of $\Ad (a I_2 \oplus B)$ with $a \in \F_p^\times$ and $B \in \GL_{n-2}(\F_p)$. When $p=2$ and $n=3$, consider the automorphism $\delta$ of $K$ given by $\delta(A) = (A^T)^{-1}$. Define the matrix $\Sigma = E_{12} + E_{21}$ in $\SL_2(\F_2)$. Then $U$ is the fixed point subgroup of $\Ad(\Sigma \oplus I_1) \circ \zeta$ in $\SL_3(\F_2)$.

Denote by $R < \Lambda$ the subgroup generated by $\Lambda_{12}$ and $\Lambda_{21}$. By definition, $R$ is a subgroup of $\pi(U)$. By our assumptions, $R$ commutes with $\Lambda_1$. Define $Q < \Lambda$ as the subgroup generated by $R$ and $\Lambda_1$. Since $R \cap \Lambda_1 \subset \pi(U \cap K_1) = \pi(\cZ(K))$ and since $\Lambda_1$ has order $p^2$, while $\pi(\cZ(K))$ has order $c \mid p-1$, we find that $R \cap \Lambda_1 = \{e\}$. So, $Q \cong R \times \Lambda_1$ and we can define a group homomorphism $\theta : Q \to \Lambda_1$ such that $\theta(s) = s$ for all $s \in \Lambda_1$ and $\Ker \theta = R$.

Since the order of $U$ divides $|\GL_2(\F_p)| \, (p-1) = p(p-1)^3 (p+1)$, we get that $p^2 \nmid |R|$. We will reach a contradiction by constructing a surjective group homomorphism from $R$ to $\F_p^2$.

Define $\al_1 = \Ad(I_n + E_{23})$ and denote $\be_1 = \pi \circ \al_1 \circ \pi^{-1} \in \Aut \Lambda$. Since $\al_1$ acts as the identity on $K_{21}$, also $\be_1$ acts as the identity on $\Lambda_{21}$. With $H$ defined by \eqref{eq.subgroup-H}, we have $\al_1(H) = H$. So, $\be_1(\pi(H))$ and thus $\be_1(\Lambda_H) = \Lambda_H$. Since $\Lambda_H$ is generated by $\Lambda_{12}$ and $\Lambda_1$, and since $\be_1(\Lambda_{21}) = \Lambda_{21}$, it follows that $\be_1$ globally preserves the subgroup generated by $\Lambda_{12}$, $\Lambda_1$ and $\Lambda_{21}$, which is precisely $Q$. So, $\be_1$ restricts to an automorphism of $Q$ that acts as the identity on $\Lambda_{21}$.

Note that $\be_1(\Lambda_{12}) \not\subset R$. Indeed, we otherwise find that
$$\be_1(\Lambda_{12}) = \be_1(\Lambda_{12}) \cap R \subset \be_1(\pi(K_{12})) \cap \pi(U) = \pi(\al_1(K_{12}) \cap U) = \pi(\cZ(K)) \; .$$
Since $\pi(\cZ(K))$ is a subgroup of $\Lambda$ with order $c \mid p-1$ and $\be_1(\Lambda_{12})$ is a subgroup of order $p$, this is absurd.

Define $\psi_1 : R \to \Lambda_1 : \psi_1(r) = \theta(\be_1(r))$. Then $\psi_1$ is a group homomorphism, $\psi_1(s_{21}) = e$ and $\psi_1(s_{12}) \neq e$ because $\be_1(s_{12}) \not\in R = \Ker \theta$. By symmetry, we also find a group homomorphism $\psi_2 : R \to \Lambda_1$ such that $\psi_2(s_{21}) \neq e$ and $\psi_2(s_{12}) = e$. Recall that $\Lambda_1 \cong \F_p^2$. So, the image of the homomorphism $\psi_1 \oplus \psi_2 : R \to \Lambda_1 \times \Lambda_1 \cong \F_p^4$ is isomorphic to $\F_p^m$ for $2 \leq m \leq 4$. It follows that $\F_p^2$ is a quotient of $R$, which is absurd. So, the claim that $\Lambda_{12}$ does not commute with $\Lambda_1$ is proven.

We proved above that $\Lambda_1$ is a normal subgroup of $\Lambda_H$. So, $\Ad s_{12}$ defines an automorphism of $\Lambda_1$. Since $\Lambda_{12}$ does not commute with $\Lambda_1$, this automorphism is not the identity. We also know that $\Lambda_1$ is generated by $\Lambda_{13}$ and $\Lambda_{23}$. We thus find a unique group isomorphism $\rho : \Lambda_1 \to \F_p^2$ satisfying $\rho(s_{13}) = e_1$ and $\rho(s_{23}) = e_2$, where $e_1 = (1,0)$ and $e_2 = (0,1)$. Then $\gamma = \rho \circ (\Ad s_{12}) \circ \rho^{-1}$ is an automorphism of $\F_p^2$ and $\gamma \neq \id$. Since $s_{12}$ commutes with $s_{13}$, we have that $\gamma(e_1) = e_1$. Since $\gamma$ is an automorphism of $\F_p^2$, we get that $\gamma(e_2) = b e_1 + a e_2$ for some $b \in \F_p$ and $a \in \F_p^\times$. Since $s_{12}$ has order $p$, also $\gamma^p = \id$. It follows that $a^p = 1$ in $\F_p$, which means that $a=1$. Since $\gamma \neq \id$, we get that $b \neq 0$. We have thus found $b \in \F_p^\times$ such that $(\Ad s_{12})(s_{23}) = s_{13}^b s_{23}$.

When $f \in \F_p^\times$ and using that $s_{12}$ commutes with $s_{13}$, we get that $(\Ad s_{12}^f)(s_{23}) = s_{13}^{bf} s_{23}$. Taking the power $f$ of this expression, gives us that $(\Ad s_{12}^f)(s_{23}^f) = s_{13}^{bf^2} s_{23}^f$. Applying this with the specific choice $f = b^{-1}$, we find that $(\Ad s_{12}^f)(s_{23}^f) = s_{13}^f s_{23}^f$.

{\bf\boldmath We replace all $s_{ij}$ by $s_{ij}^f$.} We still have that $s_{ij}$ generates $\Lambda_{ij}$ and that $\be_\si(s_{ij}) = s_{\si(i)\si(j)}$ for all $i \neq j$ and permutations $\si$. We now also have that the commutator $[s_{12},s_{23}]$ equals $s_{13}$. Applying $\be_\si$ for an arbitrary permutation $\si$, it follows that
\begin{equation}\label{eq.rel-1}
[s_{ij},s_{jk}] = s_{ik} \quad\text{for all distinct $i,j,k$.}
\end{equation}
In those cases where we proved that $s_{ij}$ commutes with $s_{rt}$, this remains valid. We thus find that
\begin{equation}\label{eq.rel-2}
[s_{ij},s_{kr}] = e \quad\text{whenever $i \neq j$, $k \neq r$, $j \neq k$ and $i \neq r$, and}\qquad s_{ij}^p = e \quad\text{if $i \neq j$.}
\end{equation}
By e.g. \cite[Theorem 1.16 in Chapter 16]{Kar93}, because $n \geq 3$, the group $K = \SL_n(\F_p)$ is the universal group with generators $s_{ij}$ for $i \neq j$ and relations \eqref{eq.rel-1} and \eqref{eq.rel-2}. We could use this to finish the proof of (iii) for $\SL_n(\F_p)$, but for efficiency, we now turn to $\SL_n(\F_q)$ when $q$ is a prime power.

{\bf\boldmath The general case $K = \SL_n(\F_q)$: finding copies of $\F_q$ in $\Lambda$.} From now, we write $K = \SL_n(\F_q)$ and assume that $\Lambda$ is a group, $\pi : K \to \Lambda$ is a bijection and $\be_\al := \pi \circ \al \circ \pi^{-1} \in \Aut \Lambda$ for all $\al \in \Aut K$.

Define the subgroup $C < \F_q^\times$ by $C = \{a \in \F_q^\times \mid a^n = 1\}$. Note that the center of $K$ is given by $\cZ(K) = \{a I_n \mid a \in C\}$. Being the fixed points of $\Ad A$ with $A \in \GL_n(\F_q)$, we see that $\cZ(K)$ is a fixed point subgroup, so that $\Lambda_\cZ := \pi(\cZ(K))$ is a subgroup of $\Lambda$. Write $c := |C|$ and note that $c \mid q-1$.

Similarly as above, define the subgroup $K_1 = \{a I_n + x E_{13} + y E_{23} \mid a \in C, x,y \in \F_q\}$. Note that $K_1$ consists of the fixed points of $\Ad (I_n + E_{ij})$ for all $i \neq j$ with $j \not\in \{1,2\}$ and $i \neq 3$. So, $L_1 := \pi(K_1)$ is a subgroup of $\Lambda$. An important part of the proof consists in realizing $\F_q^2$ as a subgroup of $L_1$. Contrary to the case $q=p$, it is not immediately clear that $L_1$ has a unique $p$-Sylow subgroup, neither that this would be abelian.

We view every field automorphism $\zeta \in \Autfield(\F_q)$ as an automorphism of $K$ by applying $\zeta$ to every component of a matrix. Since $\F_p \subset \F_q$ is a Galois extension, $\SL_n(\F_p) < \SL_n(\F_q)$ is the subgroup of elements fixed by all $\zeta \in \Autfield(\F_q)$, so that $\SL_n(\F_p)$ is a fixed point subgroup of $K$. Then also $K_1 \cap \SL_n(\F_p)$ is a fixed point subgroup, so that $\pi(K_1 \cap \SL_n(\F_p))$ is a subgroup of $L_1$ of order $c_0 p^2$, where $c_0 = |\{a \in \F_p^\times \mid a^n = 1\}|$. First choosing a $p$-Sylow subgroup of $\pi(K_1 \cap \SL_n(\F_p))$, we can next fix a $p$-Sylow subgroup $\Lambda_1 < L_1$ such that the intersection $\Lambda_1 \cap \pi(K_1 \cap \SL_n(\F_p))$ is nontrivial (actually, has order $p^2$).

Since $\Lambda_\cZ < L_1$ is a subgroup of order $c \mid q-1$, while $\Lambda_1$ is a subgroup of order $q^2$, and $\gcd(c,q^2) = 1$, we first conclude that $\Lambda_\cZ \cap \Lambda_1 = \{e\}$, and then conclude that $L_1 = \Lambda_\cZ \Lambda_1$. Define $\Lambda_0 := \{s \in \Lambda_\cZ \mid s \Lambda_1 s^{-1} = \Lambda_1\}$. Since all $p$-Sylow subgroups of $L_1$ are conjugate and $L_1 = \Lambda_\cZ \Lambda_1$, it follows that $s \mapsto s \Lambda_1 s^{-1}$ defines a bijection between $\Lambda_\cZ/\Lambda_0$ and the set of $p$-Sylow subgroups of $L_1$.

For $A \in \GL_2(\F_q)$, we define $\be_A = \pi \circ \Ad(A \oplus I_{n-2}) \circ \pi^{-1} \in \Aut \Lambda$. Since $\Ad (A \oplus I_{n-2})(K_1) = K_1$, also $\be_A(L_1) = L_1$. Since automorphisms of $L_1$ permute the $p$-Sylow subgroups of $L_1$, we find an action $\eta$ of $\GL_2(\F_q)$ by permutations of $\Lambda_\cZ/\Lambda_0$ such that
$$\be_A(s \Lambda_1 s^{-1}) = (\eta_A(s\Lambda_0)) \Lambda_1 (\eta_A(s\Lambda_0))^{-1} \quad\text{for all $s \in \Lambda_\cZ$.}$$
On the other hand, $\Ad (A \oplus I_{n-2})$ acts as the identity on $\cZ(K)$, so that $\be_A(s) = s$ for all $s \in \Lambda_\cZ$. Choosing $s_A \in \Lambda_\cZ$ such that $\be_A(\Lambda_1) = s_A \Lambda_1 s_A^{-1}$, we find that $\be_A(s \Lambda_1 s^{-1}) = s \be_A(\Lambda_1) s^{-1} = s s_A \Lambda_1 s_A^{-1} s^{-1}$. We conclude that $\eta_A(s\Lambda_0) = s s_A \Lambda_0$ for all $s \in \Lambda_\cZ$. In particular, $s_A$ belongs to the normalizer $\cN_{\Lambda_\cZ}(\Lambda_0)$, so that $\GL_2(\F_q) \to \cN_{\Lambda_\cZ}(\Lambda_0) / \Lambda_0 : A \mapsto s_A \Lambda_0$ is a well-defined group homomorphism. The order of the group at the right hand side divides $|\Lambda_\cZ| = c$ and thus divides $q-1$. The subgroup $\SL_2(\F_q) < \GL_2(\F_q)$ is generated by the elementary matrices, which have order $p$. It follows that $s_A \Lambda_0 = \Lambda_0$, meaning that $\be_A(\Lambda_1) = \Lambda_1$ for all $A \in \SL_2(\F_q)$.

Since the intersection of $\Lambda_1$ with $\pi(K_1 \cap \SL_n(\F_p))$ is nontrivial, we can choose $a_1 \in \F_p^\times$ with $a_1^n = 1$ and $x_1,y_1 \in \F_p$ such that $\pi(a_1 I_n + x_1 E_{13} + y_1 E_{23})$ is a nontrivial element of $\Lambda_1$. Since $\Lambda_1 \cap \Lambda_\cZ = \{e\}$, it follows that $(x_1,y_1) \neq (0,0)$. For all $(x,y) \in \F_q^2 \setminus \{(0,0)\}$, denote $\lambda(x,y) = a_1 I_n + x E_{13} + y E_{23} \in K_1$ and $\phi(x,y) = \pi(\lambda(x,y))$. Since $\Ad(A \oplus I_{n-2})(\lambda(x,y)) = \lambda(A \cdot (x,y))$ for all $A \in \GL_2(\F_q)$ and $(x,y) \in \F_q^2 \setminus \{(0,0)\}$, also $\be_A(\phi(x,y)) = \phi(A \cdot (x,y))$. Since $\phi(x_1,y_1) \in \Lambda_1$, $\be_A(\Lambda_1) = \Lambda_1$ for all $A \in \SL_2(\F_q)$ and $\SL_2(\F_q)$ acts transitively on $\F_q^2 \setminus \{(0,0)\}$, we conclude that $\phi(x,y) \in \Lambda_1 \setminus \{e\}$ for all $(x,y) \in \F_q^2 \setminus \{(0,0)\}$. Since $\F_q^2$ has $q^2$ elements, complementing with $\phi(0,0) = e$, we have found a bijection $\phi : \F_q^2 \to \Lambda_1$ satisfying $\be_A \circ \phi = \phi \circ A$ for all $A \in \GL_2(\F_q)$.

In (ii), we proved that $\F_q^2$ is strictly rigid relative to $\GL_2(\F_q) < \Aut(\F_q^2,+)$. Using Remark \ref{rem.strictly-rigid-no-invariant-characters}, it follows that $\phi : \F_q^2 \to \Lambda_1$ is an isomorphism of groups.

{\bf\boldmath Construction of the generators $t_{ij}(x)$ of $\Lambda$.} As explained above, $\SL_n(\F_p)$ is a fixed point subgroup of $K$, fixed by all the field automorphisms. By the description of all automorphisms of $\SL_n(\F_p)$, we get that each such automorphism $\al$ can be extended to an automorphism $\altil$ of $K$. It follows that $\pi \circ \al \circ \pi^{-1}$ is an automorphism of $\pi(\SL_n(\F_p))$ for every $\al \in \Aut \SL_n(\F_p)$. So, we can apply the first part of the proof to the restriction of $\pi$ to $\SL_n(\F_p)$.

For $i \neq j$, denote as above $K_{ij} = \{a I_n + x E_{ij} \mid a \in C , x \in \F_q\}$. Again view every permutation $\si$ as a permutation matrix and write $\be_\si = \pi \circ (\Ad \si) \circ \pi^{-1}$. By the first part of the proof, we find elements $s_{ij} \in \pi(K_{ij} \cap \SL_n(\F_p))$ that generate the unique $p$-Sylow subgroup $\Lambda_{ij}$ of $\pi(K_{ij} \cap \SL_n(\F_p))$. Moreover, $\Lambda_{ij}$ contains all elements of order $p$ of $\pi(K_{ij} \cap \SL_n(\F_p))$ and the relations \eqref{eq.rel-1} and \eqref{eq.rel-2} hold. Also, $\be_\si(s_{ij}) = s_{\si(i)\si(j)}$ for all $i \neq j$ and all permutations $\si$.

Since $\phi : \F_q^2 \to \Lambda_1$ is an isomorphism of groups, $\phi(1,0)$ is an element of order $p$ in $\Lambda_1$. By definition, $\phi(1,0) = \pi(a_1 I_n + E_{13})$. So, $\phi(1,0)$ is an element of order $p$ in $\pi(K_{12} \cap \SL_n(\F_p))$. By the discussion above, $\phi(1,0) \in \Lambda_{13}$ so that $\phi(\F_p,0) = \Lambda_{13}$. Take $x_1 \in \F_p^\times$ such that $s_{13} = \phi(x_1,0) = \pi(a_1 I_n + x_1 E_{13})$. We then define the group homomorphism $t_{13} : \F_q \to \Lambda : t_{13}(x) = \phi(x_1 x,0)$. By definition,
\begin{equation}\label{eq.good-t}
t_{13}(x) = \pi(a_1 I_n + x_1 x E_{13}) \quad\text{for all $x \in \F_q^\times$, and}\quad t_{13}(1) = s_{13} \; .
\end{equation}
For every $i$ and every $a \in \F_q^\times$, we define $\be_{i,a} \in \Aut \Lambda$ by $\be_{i,a} = \pi \circ \Ad(a E_{ii} + (I_n - E_{ii})) \circ \pi^{-1}$. Note that
$$\Ad(a E_{ii} + (I_n - E_{ii})) (a_1 I_n + x_1 x E_{13}) = \begin{cases} a_1 I_n + a x_1 x E_{13} &\quad\text{if $i=1$,}\\
a_1 I_n + a^{-1} x_1 x E_{13} &\quad\text{if $i=3$,}\\
a_1 I_n + x_1 x E_{13} &\quad\text{if $i\not\in \{1,3\}$.}\end{cases}$$
It then follows from \eqref{eq.good-t} that
\begin{equation}\label{eq.good-t-2}
\begin{split}
& \be_{1,a}(t_{13}(x)) = t_{13}(ax) \;\; , \quad \be_{3,a}(t_{13}(x)) = t_{13}(a^{-1} x) \;\;\text{and}\\ & \be_{i,a}(t_{13}(x))=t_{13}(x) \;\;\text{if $i \not\in \{1,3\}$.}
\end{split}
\end{equation}
If a permutation $\si$ satisfies $\si(1)=1$ and $\si(3) = 3$, then $(\Ad \si)(a_1 I_n + x_1 x E_{13}) = a_1 I_n + x_1 x E_{13}$. So, by \eqref{eq.good-t}, $\be_\si(t_{13}(x)) = t_{13}(x)$. It follows that we can unambiguously define, for all $i \neq j$, group homomorphisms $t_{ij} : \F_q \to \Lambda$ such that $\be_\si \circ t_{ij} = t_{\si(i) \si(j)}$ for every permutation $\si$ and all $i \neq j$.

Whenever $i \neq j$, we can choose a permutation $\si$ such that $\si(1) = i$ and $\si(3) = j$. We conclude that
$$t_{ij}(1) = \be_\si(t_{13}(1)) = \be_\si(s_{13}) = s_{ij} \; .$$
Applying an arbitrary $\be_\si$ to the equalities in \eqref{eq.good-t-2}, we find that for all $i \neq j$, all $k$, $a \in \F_q^\times$ and $x \in \F_q$,
\begin{equation}\label{eq.good-t-3}
\begin{split}
& \be_{i,a}(t_{ij}(x)) = t_{ij}(ax) \;\; , \quad \be_{j,a}(t_{ij}(x)) = t_{ij}(a^{-1} x) \;\;\text{and}\\ & \be_{k,a}(t_{ij}(x))=t_{ij}(x) \;\;\text{if $k \not\in \{i,j\}$.}
\end{split}
\end{equation}
Acting with $\be_{i,a} \circ \be_{k,b^{-1}}$ on \eqref{eq.rel-1}, it follows that
\begin{equation}\label{eq.rel-1-bis}
[t_{ij}(a),t_{jk}(b)] = t_{ik}(ab) \quad\text{for all distinct $i,j,k$ and all $a,b \in \F_q$.}
\end{equation}
When $i,j,k$ are distinct, acting with $\be_{i,a} \circ \be_{j,b}$ on $[s_{ik},s_{jk}] = e$ implies that $[t_{ik}(a),t_{jk}(b)] = e$ for all $a,b \in \F_q$. We similarly get that $[t_{ij}(a),t_{ik}(b)] = e$ for all distinct $i,j,k$ and $a,b \in \F_q$. When $i,j,k,r$ are all distinct, we act with $\be_{i,a} \circ \be_{k,b}$ on $[s_{ij},s_{kr}]=e$ and get that $[t_{ij}(a),t_{kr}(b)] = e$ for all $a,b \in \F_q$. Summarizing, we have proven that
\begin{equation}\label{eq.rel-2-bis}
[t_{ij}(a),t_{kr}(b)] = e \quad\text{whenever $i \neq j$, $k \neq r$, $j \neq k$ and $i \neq r$, and for all $a, b \in \F_q$.}
\end{equation}
Since $t_{ij}$ is a group homomorphism, we have
\begin{equation}\label{eq.rel-3-bis}
t_{ij}(x+y) = t_{ij}(x) t_{ij}(y) \quad\text{whenever $i \neq j$, and for all $x,y \in \F_q$.}
\end{equation}
By e.g. \cite[Theorem 1.16 in Chapter 16]{Kar93}, because $n \geq 3$, the group $K = \SL_n(\F_q)$ is the universal group with generators $t_{ij}(x)$ for $i \neq j$, $x \in \F_q$, and relations \eqref{eq.rel-1-bis}, \eqref{eq.rel-2-bis} and \eqref{eq.rel-3-bis}.

{\bf\boldmath Construction of a subgroup $P < \Lambda$ with $P \cong K/D$.} Define $P < \Lambda$ as the subgroup generated by all $t_{ij}(x)$, $i \neq j$, $x \in \F_q$. Since the homomorphisms $t_{ij}$ are faithful, the group $P$ is a nontrivial quotient of $K$. By e.g. \cite[Theorem 1.12 in Chapter 16]{Kar93}, the group $\PSL_n(\F_q)$ is simple. One deduces that every normal subgroup of $K$ that is different from $K$ is contained in the center $\cZ(K)$. We thus find a subgroup $D < \cZ(K)$ and a group isomorphism $\psi : P \to K/D$ satisfying $\psi(t_{ij}(x)) = (I_n + x E_{ij})D$ for all $i \neq j$ and $x \in \F_q$.

{\bf\boldmath The subgroup $\Lambda_\cZ = \pi(\cZ(K))$ normalizes $t_{12}(\F_q)$.} As above, we consider the subgroup $\Lambda_\cZ = \pi(\cZ(K)) = \pi(\{a I_n \mid a \in C\})$. We claim that $\Lambda_\cZ$ normalizes the subgroup $t_{12}(\F_q)$. If $q=2$, we have that $\Lambda_\cZ = \{e\}$ and there is nothing to prove. So we assume that $q \geq 3$. By construction, $\Lambda_\cZ$ and $t_{12}(\F_q)$ are subgroups of $\pi(K_{12})$ of order $c$, resp.\ $q$, while $\pi(K_{12})$ has order $c q$. Since $\gcd(c,q) = 1$, we get that $\Lambda_\cZ \cap t_{12}(\F_q) = \{e\}$. We then conclude that $\pi(K_{12}) = \Lambda_\cZ t_{12}(\F_q) = t_{12}(\F_q) \Lambda_\cZ$.

Fix $g \in \Lambda_\cZ$ and $x \in \F_q$. We can then uniquely write $g t_{12}(x) = t_{12}(y) h$ with $y \in \F_q$ and $h \in \Lambda_\cZ$. We apply $\be_{1,a}$ to this equality, with $a \in \F_q^\times$. Since $a E_{11} + (I_n-E_{11})$ commutes with $\cZ(K)$, we have that $\be_{1,a}$ acts as the identity on $\Lambda_\cZ$. Using \eqref{eq.good-t-3}, we find that $g t_{12}(ax) = t_{12}(ay) h$. So,
$$g t_{12}((a-1)x) g^{-1} = g t_{12}(ax) (g t_{12}(x))^{-1} = t_{12}(ay) h h^{-1} t_{12}(y)^{-1} =t_{12}((a-1)y) \; .$$
It follows that $g t_{12}((a-1)x) g^{-1} \in t_{12}(\F_q)$ for all $g \in \Lambda_\cZ$, $a \in \F_q^\times$ and $x \in \F_q$. Since $q \geq 3$, we can take $a \in \F_q^\times$ with $a-1 \neq 0$ and the claim that $\Lambda_\cZ$ normalizes $t_{12}(\F_q)$ is proven.

{\bf\boldmath Construction of field automorphisms $\zeta_g$ for every $g \in \Lambda_\cZ$.} For every $g \in \Lambda_\cZ$, define $\zeta_g \in \Aut(\F_q,+)$ such that $g t_{12}(x) g^{-1} = t_{12}(\zeta_g(x))$ for all $x \in \F_q$. Since every permutation matrix $\si$ commutes with $\cZ(K)$, we get that $\be_\si(g) = g$ for all $g \in \Lambda_\cZ$. Applying $\be_\si$, we find that $g t_{ij}(x) g^{-1} = t_{ij}(\zeta_g(x))$ for all $i \neq j$ and $x \in \F_q$. Conjugating the relation \eqref{eq.rel-1-bis} by $g$, it follows that $\zeta_g(a) \zeta_g(b) = \zeta_g(ab)$ for all $a,b \in \F_q$. So, $\zeta_g \in \Autfield(\F_q)$ for all $g \in \Lambda_\cZ$.

{\bf\boldmath Proof that $\zeta_g(D) = D$ for all $g \in \Lambda_\cZ$ and $\zeta_g = \id$ for all $g \in P \cap \Lambda_\cZ$.} Since $(\Ad g)(t_{ij}(x)) = t_{ij}(\zeta_g(x))$, we get that $(\Ad g)(P) = P$ for all $g \in \Lambda_\cZ$, so that $\mu_g := \psi \circ (\Ad g)|_P \circ \psi^{-1}$ are automorphisms of $K/D$ satisfying $\mu_g((I_n + x E_{ij})D) = (I_n + \zeta_g(x) E_{ij})D$ whenever $i \neq j$ and $x \in \F_q$. It follows that $\zeta_g(D) = D$ and that $\mu_g$ equals the field automorphism of $K/D$ given by $\zeta_g$. When $g \in P \cap \Lambda_\cZ$, we get that $\mu_g = \Ad \psi(g)$ is inner. This forces $\zeta_g = \id$, so that $\mu_g = \id$ and thus $\psi(g) \in \cZ(K/D)$. We thus find a subgroup $D_1 < \cZ(K)$ such that $D < D_1$ and $\psi(P \cap \Lambda_\cZ) = D_1/D$.

{\bf\boldmath Proof that $\Lambda = \Lambda_\cZ P$ and $\psi(P \cap \Lambda_\cZ) = \cZ(K)/D$.} To prove these statements, note that
$$|K| = |\Lambda| \geq |\Lambda_\cZ P| = \frac{|\Lambda_\cZ| \, |P|}{|P \cap \Lambda_\cZ|} = \frac{|\cZ(K)| \, (|K|/|D|)}{|D_1| / |D|} = \frac{|\cZ(K)|}{|D_1|} \, |K| \geq |K| \; .$$
It follows that all inequalities must be equalities, so that $\Lambda = \Lambda_\cZ P$ and $D_1 = \cZ(K)$.

{\bf\boldmath For every $A \in \PGL_n(\F_q)$, $\be_A = \pi \circ (\Ad A) \circ \pi^{-1}$ satisfies $\be_A(P) = P$.} Since $\Ad A$ acts as the identity on $\cZ(K)$, we get that $\be_A(g) = g$ for all $g \in \Lambda_\cZ$. It thus suffices to prove that an arbitrary automorphism $\be \in \Aut \Lambda$ with $\be(g) = g$ for all $g \in \Lambda_\cZ$, satisfies $\be(P) = P$.

Define the normal subgroup $\Lambda_\zeta < \Lambda_\cZ$ as the kernel of the homomorphism $\Lambda_\cZ \to \Autfield(\F_q) : g \mapsto \zeta_g$ constructed above. We proved above that $P \cap \Lambda_\cZ < \Lambda_\zeta$. Since $\Autfield(\F_q) \cong \Z/k\Z$, the quotient $\Lambda_\cZ / \Lambda_\zeta$ is abelian. Since $\Lambda_\cZ$ normalizes $P$ and since $\Lambda = \Lambda_\cZ P$, we get that $P$ is a normal subgroup of $\Lambda$ and we denote by $\nu : \Lambda/P \to \Lambda_\cZ / (P \cap \Lambda_\cZ)$ the canonical identification. Then, $P \to \Lambda_\cZ / \Lambda_\zeta : a \mapsto \nu(\be(a)P)\Lambda_\zeta$ is a well-defined group homomorphism. Since $P \cong K/D$ is a perfect group, this homomorphism must be trivial. That means that $\nu(\be(a)P) \in \Lambda_\zeta / (P \cap \Lambda_\cZ)$ for all $a \in P$.

When $g \in \Lambda_\zeta$, we have that $g$ commutes with $P$ and $\be(g) = g$. So, $\nu(\be(a)P)$ belongs to the center of $\Lambda_\zeta / (P \cap \Lambda_\cZ)$ for every $a \in P$. Again using that $P$ is a perfect group, we conclude that $\nu(\be(a)P) = P \cap \Lambda_\cZ$ and thus, $\be(a) \in P$, for all $a \in P$. So, $\be(P) \subset P$, implying that $\be(P) = P$.

{\bf\boldmath Replacing $\psi$ by $\psitil$ such that $\psitil \circ (\be_A)|_P = (\Ad A) \circ \psitil$.} Define for every $A \in \PGL_n(\F_q)$, the automorphism $\gamma_A \in \Aut (K/D)$ by $\gamma_A = \psi \circ (\be_A)|_P \circ \psi^{-1}$. By construction, $A \mapsto \gamma_A$ is a faithful group homomorphism. It follows from our description of the automorphisms of $K/D$ that $\Ad \PSL_n(\F_q)$ is a normal subgroup of $\Aut (K/D)$ and that the quotient is solvable. Since the group $\PSL_n(\F_q)$ is perfect, it follows that $\gamma_A \in \Ad \PSL_n(\F_q)$ for all $A \in \PSL_n(\F_q)$. We may thus view the restriction $\gamma_0$ of $\gamma$ to $\PSL_n(\F_q)$ as an automorphism of $\PSL_n(\F_q)$. From our description of the automorphisms of $\PSL_n(\F_q)$, it follows that $\gamma_0$ lifts uniquely to an automorphism $\rho$ of $K$. Defining $\Dtil = \rho^{-1}(D)$ and $\psitil = \rho^{-1} \circ \psi$, we get that $\psitil : P \to K/\Dtil$ is a group isomorphism. We get that $\psitil \circ (\be_A)|_P = \gammatil_A \circ \psitil$ for all $A \in \PGL_n(\F_q)$, where $A \mapsto \gammatil_A$ is a faithful group homomorphism from $\PGL_n(\F_q)$ to $\Aut(K/\Dtil)$ satisfying $\gammatil_A = \Ad A$ for all $A \in \PSL_n(\F_q)$. When $A \in \PGL_n(\F_q)$, it follows that $(\Ad A^{-1}) \circ \gammatil_A$ is an automorphism of $K/D$ that commutes with all automorphisms $\Ad B$, $B \in \PSL_n(\F_q)$. So, $(\Ad A^{-1}) \circ \gammatil_A = \id$ and we have proven that $\psitil \circ (\be_A)|_P = (\Ad A) \circ \psitil$ for all $A \in \PGL_n(\F_q)$.

{\bf End of the proof.} It suffices to show that $\Dtil = \{I_n\}$. Indeed, it then follows that $\psitil(P \cap \Lambda_\cZ) = \cZ(K)$. Since $|\cZ(K)| = |\Lambda_\cZ|$, this implies that $\Lambda_\cZ \subset P$. Then, $\Lambda = P$ and $P \cong K$, so that $\Lambda \cong K$.

For every $A \in \PGL_n(\F_q)$, we denote by $\Fix_K(\Ad A)$ the fixed point subgroup of $\Ad A$ in $K$. We similarly denote by $\Fix_\Lambda(\be_A)$ the fixed point subgroup of the corresponding automorphism $\be_A \in \Aut \Lambda$. Since $\be_A = \pi \circ (\Ad A) \circ \pi^{-1}$, we get that $|\Fix_K(\Ad A)| = |\Fix_\Lambda(\be_A)|$ for all $A \in \PGL_n(\F_q)$.

On the other hand, $\be_A$ acts as the identity on $\Lambda_\cZ$ and satisfies $\psitil \circ (\be_A)|_P = (\Ad A) \circ \psitil$. Also, because $\psitil(P \cap \Lambda_\cZ) = \cZ(K)/\Dtil$,
$$|\Lambda_\cZ / (P \cap \Lambda_\cZ)| = \frac{|\Lambda_\cZ|}{|P \cap \Lambda_\cZ|} = \frac{|\cZ(K)|}{|\cZ(K)|/|\Dtil|} = |\Dtil| \; .$$
We conclude that
$$|\Fix_\Lambda(\be_A)| = |\Fix_{K/\Dtil}(\Ad A)| \, |\Lambda_\cZ / (P \cap \Lambda_\cZ)| = |\Dtil| \, |\Fix_{K/\Dtil}(\Ad A)| \; .$$
So, $|\Fix_K(\Ad A)| = |\Dtil| \, |\Fix_{K/\Dtil}(\Ad A)|$ for all $A \in \PGL_n(\F_q)$. Since $\Dtil < \cZ(K)$, we have $\Dtil < \Fix_K(\Ad A)$ and thus $\Fix_K(\Ad A) / \Dtil \subset \Fix_{K/\Dtil}(\Ad A)$. We just proved that both sets have the same number of elements. They must thus be equal:
\begin{equation}\label{eq.equal-fixed-points}
\Fix_K(\Ad A) / \Dtil = \Fix_{K/\Dtil}(\Ad A) \quad\text{for all $A \in \PGL_n(\F_q)$.}
\end{equation}
As explained above, it only remains to show that $\Dtil = \{I_n\}$. Assume the contrary. We can then take $d_0 \in \F_q^\times$ such that $d_0 \neq 1$, $d_0^n = 1$ and $d_0 I_n \in \Dtil$. Define $A_0 \in \GL_n(\F_p)$ as the diagonal matrix $A_0 = \sum_{i=1}^n d_0^{i-1} E_{ii}$. Define the element $X \in K$ by
$$X = (-1)^{n+1} E_{n,1} + \sum_{k=1}^{n-1} E_{k,k+1} \; .$$
We have that $AXA^{-1} = d_0^{-1} X$. Then $X \Dtil$ is an element of $\Fix_{K/\Dtil}(\Ad A)$ that does not belong to $\Fix_K(\Ad A)/\Dtil$, contradicting \eqref{eq.equal-fixed-points}. This concludes the proof of the rigidity of $K = \SL_n(\F_p)$ relative to $\Aut K$.

The proof of the rigidity of $\PSL_n(\F_p)$ relative to its automorphism group is identical to the argument above, with the simplification that all considerations about the center disappear.
\end{proof}

\begin{proof}[{\bf\boldmath Proof of Theorem \ref{thm.rigid-SLn-Fp.four}: the groups $\F_q^n \rtimes \SL_n(\F_q)$.}]
Fix $n \geq 2$ and $q = p^k$. Put $K_1 = \F_q^n \rtimes \SL_n(\F_q)$. Let $\Lambda_1$ be a group and $\pi : K_1 \to \Lambda_1$ a bijection such that $\pi \circ \al \circ \pi^{-1} \in \Aut \Lambda_1$ for all $\al \in \Aut K_1$.

We denote the elements of $K_1$ as $(a,A)$ with $a \in \F_q^n$ and $A \in \SL_n(\F_q)$. The product is given by $(a,A) \cdot (b,B) = (a + A \cdot b, AB)$, where we view $a$ as a column matrix. We view $\F_q^n$ and $\SL_n(\F_q)$ as subgroups of $K_1$, identifying $a$ with $(a,I_n)$ and $A$ with $(0,A)$. We also view $K_1$ as a subgroup of $\F_q^n \rtimes \GL_n(\F_q)$, so that conjugation $\Ad A$ with $A \in \GL_n(\F_q)$ is a well-defined automorphism of $K_1$. For the intuition behind some of the computations in the proof, it is useful to see $K_1$ as a subgroup of $\SL_{n+1}(\F_q)$ by identifying $(a,A)$ with the matrix $\bigl(\begin{smallmatrix} A & a \\ 0 & 1\end{smallmatrix}\bigr)$.

We denote by $e_i \in \F_q^n$ the standard basis elements. As in the proof of (iii), we denote by $E_{ij}$ the matrix that has $1$ in position $(i,j)$ and $0$ elsewhere. For every $A \in \GL_n(\F_q)$, we denote $\be_A = \pi \circ (\Ad A) \circ \pi^{-1} \in \Aut \Lambda_1$. We in particular use the notation $\be_\si$ when $\si \in \GL_n(\F_q)$ is a permutation matrix.

Note that $\F_q^n$ is the fixed point subgroup of the automorphisms $\Ad e_i$ for all $i$. We can thus define the subgroup $N_1 = \pi(\F_q^n)$ of $\Lambda_1$. For every $A \in \GL_n(\F_q)$, we have that $(\Ad A)(\F_q^n) = \F_q^n$, and the equality $\pi|_{\F_q^n} \circ \Ad A = \be_A \circ \pi|_{\F_q^n}$ holds on $\F_q^n$. By (ii), the restriction $\pi|_{\F_q^n}$ is a group isomorphism $\pi : \F_q^n \to N_1$. Throughout the proof, we use without mentioning that $\pi|_{\F_q^n}$ is a group homomorphism.

{\bf\boldmath The exceptional case $n=2$, $q=2$.} We claim that $\F_2^2 \rtimes \SL_2(\F_2) \cong S_4$. To prove this claim, it suffices to observe that $(12)(34) \mapsto e_1$, $(14)(23) \mapsto e_2$, together with $(12) \mapsto \bigl(\begin{smallmatrix} 1 & 1 \\ 0 & 1\end{smallmatrix}\bigr)$ and $(123) \mapsto \bigl(\begin{smallmatrix} 0 & 1 \\ 1 & 1\end{smallmatrix}\bigr)$ concretely defines an isomorphism $S_4 \to \F_2^2 \rtimes \SL_2(\F_2)$. By Theorem \ref{thm.relative-rigid-An-tilde.one}, the isomorphism $\F_2^2 \rtimes \SL_2(\F_2) \cong \Lambda_1$ follows.

But for later use, we remark that we get more: using Lemma \ref{lem.finite-group-rigid.two}, we find a group isomorphism $\rho : K_1 \to \Lambda_1$ such that $\rho \circ \al \circ \rho^{-1} = \pi \circ \al \circ \pi^{-1}$ for all $\al \in \Aut K_1$. Looking at the fixed point subgroup of $\Ad e_1$, $\Ad e_2$, it follows that $N_1 = \rho(\F_2^2)$. Since $\rho$ is a group isomorphism, the centralizer of $N_1$ in $\Lambda_1$ equals $N_1$~; a fact that we will use below.

For the rest of the proof, we may thus assume that $(n,q) \neq (2,2)$.

{\bf\boldmath Reduction to the subgroup $K = \F_p^n \rtimes \SL_n(\F_p)$.} If $\zeta \in \Autfield(\F_q)$, applying $\zeta$ to every component induces an automorphism $\al_\zeta$ of $K_1$. The fixed point subgroup of these field automorphisms is $K := \F_p^n \rtimes \SL_n(\F_p)$, and we define the subgroup $\Lambda = \pi(K)$ of $\Lambda_1$. We write $N = \pi(\F_p^n)$ and note that $N$ is a subgroup of $\Lambda$. The main part of the proof consists in showing that the elements $\pi(I_n + E_{ij})$ normalize $N$ and define a nontrivial automorphism of $N$ by conjugation.

{\bf\boldmath The elements $\pi(I_n + E_{12})$ and $\pi(e_1)$ commute.} We note that $\{(a e_1, I_n + b E_{12}) \mid a,b \in \F_p\}$ is the fixed point subgroup of $K$ of the automorphisms $\Ad e_i$ and $\Ad (I_n + E_{1j})$ for all $i \neq 2$ and $j \geq 2$. So its image is a subgroup of $\Lambda$ of order $p^2$, which is therefore abelian. Since this subgroup contains $\pi(e_1)$ and $\pi(I_n+E_{12})$, it follows that these two elements commute.

{\bf\boldmath If $n \geq 3$, the elements $\pi(I_n + E_{13})$ and $\pi(e_2)$ commute.} For every matrix $X = \bigl(\begin{smallmatrix} x & a \\ y & b\end{smallmatrix}\bigr)$ in $\F_p^{2 \times 2}$, define $\psi(X) \in K$ by $\psi(X) = (a e_1+b e_2,I_n + x E_{13} + y E_{23})$. Note that $\psi(\F_p^{2 \times 2})$ is the fixed point subgroup of $K$ of $\Ad e_i$ and $\Ad (I_n + E_{1j})$ for all $i \neq 3$ and $j \geq 3$. We define the subgroup $L = \pi(\psi(\F_p^{2 \times 2}))$ of $\Lambda$.

For every $A \in \GL_2(\F_p)$ and $c \in \F_p$, we have that
$$\Ad(c e_3, A \oplus I_{n-2})(\psi(X)) = \psi\bigl(A X \bigl(\begin{smallmatrix} 1 & -c \\ 0 & 1\end{smallmatrix}\bigr)\bigr) \; .$$
So, $\be_{A,c} := \pi \circ \Ad(c e_3, A \oplus I_{n-2}) \circ \pi^{-1}$ restricts to an action of $\GL_2(\F_p) \times \F_p$ by automorphisms of $L$. Since $L$ is a group of order $p^4$, the center $\cZ(L)$ is nontrivial and thus $|\cZ(L)|$ is divisible by $p$. Then, $|\cZ(L) \setminus \{e\}|$ is not divisible by $p$. The action of $\GL_2(\F_p) \times \F_p$ on $\F_p^{2 \times 2} \setminus \{0\}$ by $(A,c) \cdot X = A X \bigl(\begin{smallmatrix} 1 & -c \\ 0 & 1\end{smallmatrix}\bigr)$ has the following three orbits:
$$\bigl\{ \bigl(\begin{smallmatrix} 0 & a \\ 0 & b\end{smallmatrix}\bigr) \bigm| (a,b) \in \F_p^2 \setminus \{(0,0)\}\bigr\} \; , \quad
\bigl\{ \bigl(\begin{smallmatrix} x & dx \\ y & dy\end{smallmatrix}\bigr) \bigm| (x,y) \in \F_p^2 \setminus \{(0,0)\}, d \in \F_p \bigr\} \; , \quad \GL_2(\F_p) \; .$$
The number of elements in these orbits are $p^2-1$, $p(p^2-1)$ and $p(p-1)^2(p+1)$. Since $\cZ(L) \setminus \{e\}$ is a union of orbits and since $|\cZ(L) \setminus \{e\}|$ is not divisible by $p$, we conclude that the first orbit must occur, so that $\pi\bigl(\psi\bigl(\begin{smallmatrix} 0 & a \\ 0 & b\end{smallmatrix}\bigr)\bigr) \in \cZ(L)$ for all $a,b \in \F_p$. We get that $\pi(e_2)$ belongs to the center of $L$. Since $\pi(I_n + E_{13})$ belongs to $L$, it follows that $\pi(I_n + E_{13})$ commutes with $\pi(e_2)$.

{\bf\boldmath The element $\pi(I_n + E_{12})$ normalizes the subgroup $\pi(\F_p e_1 + \F_p e_2)$.} Note that
\begin{equation}\label{eq.fixed-point-T}
T = \{(z e_1 + y e_2, I_n + x E_{12}) \mid x,y,z \in \F_p\}
\end{equation}
is the fixed point subgroup of $K$ of the automorphisms $\Ad e_i$ and $\Ad (I_n + E_{1j})$ for all $i \neq 2$ and $j \geq 3$. So, $\pi(T)$ is a subgroup of $\Lambda$ of order $p^3$. Also, $N_0 := \pi(\F_p e_1 + \F_p e_2)$ is a subgroup of $\pi(T)$ of order $p^2$ and $\pi(I_n + E_{12}) \not\in N_0$. So, $\pi(T)$ is generated by $\pi(I_n + E_{12})$ and $N_0$. We have proven above that $\pi(I_n + E_{12})$ commutes with $\pi(\F_p e_1)$. Since $\pi(\F_p e_1) < N_0$ and $N_0$ is abelian, it follows that $\pi(\F_p e_1)$ is a subgroup of order $p$ of the center of $\pi(T)$. It follows that $\pi(T)/\pi(\F_p e_1)$ is a group of order $p^2$, which is therefore abelian. In particular, $N_0 / \pi(\F_p e_1)$ is a normal subgroup, so that $N_0$ is normal in $\pi(T)$. We get that $\pi(I_n + E_{12})$ normalizes $N_0$.

{\bf\boldmath The element $\pi(I_n + E_{12})$ does not commute with $\pi(e_2)$.} This is the hardest statement to prove, especially when $q = 2$. Assume by contradiction that $\pi(I_n + E_{12})$ commutes with $\pi(e_2)$. Since we have seen above that $\pi(I_n + E_{12})$ commutes with $\pi(e_1)$ and also with $\pi(e_i)$ for all $i \geq 3$, we get that $\pi(I_n + E_{12})$ commutes with $N$. Applying $\be_\si$ for an arbitrary permutation $\si$, it follows that $s_{ij} := \pi(I_n + E_{ij})$ commutes with $N$ for all $i \neq j$. We derive a contradiction on a case by case basis: first for $q \geq 3$, and then for $q=2$ with resp.\ $n \geq 4$ or $n =3$.

{\bf\boldmath Case 1: $q \geq 3$.} We repeat an argument that we already used in the proof of (iii). Because $q \geq 3$, the subgroup $U := \{A \oplus I_{n-2} \mid A \in \SL_2(\F_p)\}$ is the fixed point subgroup of $K$ of $\Ad (a I_2 \oplus b I_{n-2})$ and $\Ad e_i$ for all $a,b \in \F_q^\times$ and $i \geq 3$. Then $\pi(U)$ is a subgroup of $\Lambda$ of order $|\SL_2(\F_p)| = p (p^2-1)$. We define the subgroup $R < \pi(U)$ generated by $s_{12}$ and $s_{21}$, so that $|R|$ divides $p (p^2-1)$. By assumption, $R$ commutes with $N$. Since $R \cap N \subset \pi(U) \cap N = \{e\}$, the subgroup $Q$ generated by $R$ and $N$ is isomorphic with $R \times N$ and we find a group homomorphism $\theta : Q \to N$ such that $\Ker \theta = R$ and $\theta(g) = g$ for all $g \in N$.

Consider the automorphism $\be_2 = \pi \circ \Ad e_2 \circ \pi^{-1}$ of $\Lambda$. Define the fixed point subgroup $T$ as in \eqref{eq.fixed-point-T}. Since $e_2 \in T$, we have that $\be_2(\pi(T)) = \pi(T)$. Since $(\Ad e_2)(0,I_n + E_{12}) = (-e_1,I_n + E_{12})$, we get that $\be_2(s_{12}) \not\in \pi(U)$. Since $e_2$ commutes with $I_n + E_{21}$, we get that $\be_2(s_{21}) = s_{21}$. We observed above that $\pi(T)$ is generated by $\pi(\F_p e_1 + \F_p e_2)$ and $s_{12}$. So altogether $\be_2(Q) = Q$, $\be_2(s_{12}) \not\in R$ and $\be_2(s_{21}) = s_{21}$. Then, $\psi_2 : R \to N : \psi_2(g) = \theta(\be_2(g))$ is a group homomorphism satisfying $\psi_2(s_{12}) \neq e$ and $\psi_2(s_{21}) = e$. By symmetry, we also find a group homomorphism $\psi_1 : R \to N$ satisfying $\psi_1(s_{12}) = e$ and $\psi_1(s_{21}) \neq e$. Since $N \cong \F_p^n$, the image of the group homomorphism $\psi_1 \oplus \psi_2$ is a group of order $p^m$ with $m \geq 2$. This contradicts the fact that $p^2$ does not divide the order of $R$.

{\bf\boldmath Case 2: $q =2$ and $n \geq 4$.} Note that $K_0 := \{(a e_1 + b e_2 , A \oplus I_{n-2}) \mid a,b \in \F_2, A \in \SL_2(\F_2)\}$ is the fixed point subgroup of $\Ad(I_2 \oplus B)$ for all $B \in \SL_{n-2}(\F_2)$. Since $K_0 \cong \F_2^2 \rtimes \SL_2(\F_2)$, it follows from the exceptional case $(n,q) = (2,2)$ studied in the beginning of the proof that the centralizer of $\pi(\F_2 e_1 + \F_2 e_2)$ in $\pi(K_0)$ equals $\pi(\F_2 e_1 + \F_2 e_2)$. So, $s_{12} \in N$, which is absurd.

{\bf\boldmath Case 3: $q=2$ and $n=3$.} Recall that for every $A \in \SL_3(\F_2)$, we define the automorphism $\be_A := \pi \circ \Ad A \circ \pi^{-1}$. Since $(\Ad A)(\F_2^3) = \F_2^3$, we get that $\be_A(N) = N$. We define $G$ as the subgroup of $\Lambda$ generated by $N$ and the elements $\be_A(s_{12})$ for all $A \in \SL_3(\F_2)$. By construction, $\be_A(G) = G$ for all $A \in \SL_3(\F_2)$. Since $s_{12}$ is assumed to commute with $N$, also $\be_A(s_{12})$ commutes with $N$, so that $N$ is a central subgroup of $G$. We define $\Gamma = G/N$. To reach a contradiction, we prove that the order of $\Gamma$ is not a power of $2$ and we also construct a family of group homomorphisms $\psi_a : \Gamma \to N$ whose kernels have a trivial intersection.

We first prove the following statement: if $B \in \SL_3(\F_2)$ and $\pi(B) \in s_{12} N$, then $B = I_3 + E_{12}$. We start by proving that $s_{32} \not\in s_{12} N$. Note that this is not a triviality since we do not know if $s_{12} \pi(a) = \pi(a,I_3 + E_{12})$ for all $a \in \F_2^3$. We already mentioned above that
\begin{equation}\label{eq.fixed-point-T1}
T_1 := \{(a e_1 , I_3 + x E_{12}) \mid a,x \in \F_2\}
\end{equation}
is the fixed point subgroup of $\Ad e_i$ and $\Ad(I_3 + E_{1j})$ for $i=1,3$ and $j=2,3$. Also $T_2 := \{(b e_2 + c e_3, I_3 + y E_{32}) \mid b,c,y \in \F_2\}$ is the fixed point subgroup of $\Ad e_1$, $\Ad e_3$ and $\Ad (I_3 + E_{31})$. By definition, $T_1 \cap T_2 = \{e\}$. Assume that $s_{32} = s_{12} \pi(a e_1 + b e_2 + c e_3)$ with $a,b,c \in \F_2$. Since $\pi|_{\F_2^3}$ is a group homomorphism, we get that $s_{32} \pi(b e_2 + c e_3) = s_{12} \pi(a e_1)$. Since $s_{32}$ and $\pi(be_2 + c e_3)$ belong to $\pi(T_2)$ and $\pi(T_2)$ is a subgroup, the left hand side belongs to $\pi(T_2)$. Similarly, the right hand side belongs to $\pi(T_1)$. Since $\pi(T_1) \cap \pi(T_2) = \{e\}$, we conclude that $s_{12} \pi(a e_1) = e$. So, $\pi(I_3 + E_{12}) = s_{12} = \pi(a e_1)$, which is absurd.

Since $\Ad(I_3 + E_{13})(I_3 + E_{32}) = I_3 + E_{12} + E_{32}$, while $\Ad(I_3 + E_{13})(I_3 + E_{12}) =I_3 + E_{12}$, applying the automorphism $\pi \circ \Ad(I_3 + E_{13}) \circ \pi^{-1}$ to $s_{32} \not\in s_{12} N$ implies that $\pi(I_3 + E_{12} + E_{13}) \not\in s_{12} N$. Since $s_{12} = \pi(I_3 + E_{12}) \not\in N$, we have thus proven that $\pi(I_3 + x E_{12} + y E_{32}) \not\in s_{12}N$ if $(x,y) \in \F_2^2 \setminus \{(1,0)\}$.

Finally assume that $\pi(B) \in s_{12} N$. Note that $C := \{(a, I_3 + x E_{12} + y E_{32}) \mid a \in \F_2^3, x,y \in \F_2\}$ is the fixed point subgroup of $\Ad e_1$ and $\Ad e_3$. Since $s_{12} \in \pi(C)$, $N \subset \pi(C)$ and $\pi(C)$ is a subgroup, also $s_{12} N \subset \pi(C)$. It follows that $(0,B) \in C$. So, $B = I_3 + x E_{12} + y E_{32}$ for certain $x,y \in \F_2$. Since $\pi(B) \in s_{12} N$, it follows from the previous paragraph that $x=1$ and $y=0$, so that $B = I_3 + E_{12}$. Thus, the statement above is proven.

Define the subset $R_0 \subset \SL_3(\F_2)$ as the orbit $R_0 := \{A (I_3 + E_{12})A^{-1} \mid A \in \SL_3(\F_2)\}$. We claim that if $B,B' \in R_0$ and $B \neq B'$, then $\pi(B) N \neq \pi(B') N$. Since we can use the automorphisms $\be_A = \pi \circ \Ad A \circ \pi^{-1}$, it suffices to note that $\pi(B) N \neq \pi(I_3 + E_{12})N$ if $B \in R_0$ and $B \neq I_3+E_{12}$. But this follows from the statement above.

Denote by $\theta : G \to \Gamma = G/N$ the quotient homomorphism. For every $B = A (I_3 + E_{12})A^{-1} \in R_0$, we have that $\pi(B) = \be_A(s_{12})$ belongs to $G$. By construction, $\pi(B) \not\in N$. So by the previous paragraph, the elements $\theta(\pi(B))$, $B \in R_0$, are distinct elements of $\Gamma \setminus \{e\}$. It follows that $|\Gamma| \geq 1 + |R_0|$. Since $|\SL_3(\F_2)| = 2^3 \cdot 3 \cdot 7$ and the centralizer of $I_3 + E_{12}$ has order $8$, we get that $|R_0| = 21$. So, $|\Gamma| \geq 22$. On the other hand, $N < G < K$, so that $|\Gamma|$ divides $|\SL_3(\F_2)| = 2^3 \cdot 3 \cdot 7$. Both together imply that at least one of the primes $3$ or $7$ divides $|\Gamma|$. We have thus proven that $|\Gamma|$ is not a power of $2$.

We next construct a faithful family of group homomorphisms $\psi_a : \Gamma \to N$. For every $a \in \F_2^3$, define the automorphism $\be_a = \pi \circ \Ad a \circ \pi^{-1}$ of $\Lambda$. Since $(\Ad a)(0,I_3+E_{12}) = (a_2 e_1,I_3 + E_{12})$, we get that $\be_a(\pi(I_3+E_{12})) = \pi(a_2 e_1,I_3 + E_{12})$. We observed above that with $T_1$ defined by \eqref{eq.fixed-point-T1}, the subgroup $\pi(T_1)$ is generated by $s_{12}$ and $\pi(e_1)$, so that $\pi(T_1) < G$. We conclude that $\be_a(s_{12}) \in G$ for all $a \in \F_2^3$.

Acting with $\be_A$ and using that $\be_A \circ \be_a = \be_{A \cdot a} \circ \be_A$, it follows that $\be_a(\be_A(s_{12})) \in G$ for all $A \in \SL_3(\F_2)$. Since $\be_a(g) = g$ for all $g \in N$, we conclude that $\be_a(G) = G$.

We next claim that $\theta(\be_a(g)) = \theta(g)$ for all $g \in G$. Since $\be_a(g) = g$ for all $g \in N$ and since we can act with $\be_A$, it suffices to prove that $\be_a(s_{12}) \in s_{12} N$. We already mentioned above that $\be_a(s_{12}) = \pi(a_2 e_1,I_3 + E_{12})$. We again use the fixed point subgroup $T_1$ defined by \eqref{eq.fixed-point-T1}. Then $\pi(T_1)$ is a group of order $4$. It is thus abelian and $\pi(T_1) / \pi(\F_2 e_1)$ has only one nontrivial element. Since $\pi(a_2 e_1,I_3 + E_{12}) \not\in \pi(\F_2 e_1)$, it follows that $\be_a(s_{12}) \pi(\F_2 e_1)$ is a nontrivial element of $\pi(T_1) / \pi(\F_2 e_1)$ and thus equal to $s_{12} \pi(\F_2 e_1)$. This means that $\be_a(s_{12}) \in s_{12} \pi(\F_2 e_1)$ and the claim is proven.

By construction, we have a central extension $e \to N \to G \to \Gamma \to e$, with quotient homomorphism $\theta : G \to \Gamma$. Choose a map $\vphi : \Gamma \to G$ satisfying $\theta \circ \vphi = \id$. Fix $a \in \F_2^3$. Since $\theta \circ \be_a = \theta$, we can uniquely define for every $g \in \Gamma$, the element $\psi_a(g) \in N$ such that $\be_a(\vphi(g)) = \vphi(g) \psi_a(g)$. Since $\be_a(k) = k$ for all $k \in N$, it follows that $\psi_a(gh) = \psi_a(g) \psi_a(h)$ for all $g,h \in \Gamma$. So, every $\psi_a : \Gamma \to N$ is a group homomorphism.

To reach the desired contradiction, we claim that if $g \in \Gamma$ and $\psi_a(g) = e$ for all $a \in \F_2^3$, then $g = e$. For such an element $g \in \Gamma$, we have by definition that $\be_a(\vphi(g)) = \vphi(g)$ for all $a \in \F_2^3$. But $N$ is the fixed point subgroup of the automorphisms $\be_a$, $a \in \F_2^3$. So, $\vphi(g) \in N$, meaning that $g = e$. We have thus finally proven that $\pi(I_n + E_{12})$ does not commute with $\pi(e_2)$.

{\bf End of the proof.} Define the automorphism $\gamma_{12}$ of $\F_p^n$ such that $\pi \circ \gamma_{12} = (\Ad s_{12}) \circ \pi$. We thus know that $\gamma_{12}(e_i) = e_i$ for all $i \neq 2$, that $\gamma_{12}(e_2) \in \F_p e_1 + \F_p e_2$ and that $\gamma_{12} \neq \id$. Using the fixed point subgroup $T_1$ defined by \eqref{eq.fixed-point-T1}, we also know that $s_{12}^{p^2} = e$, so that $\gamma_{12}^{p^2} = \id$. Altogether, this implies that $\gamma_{12}(e_2) = b_0 e_1 + e_2$ for some $b_0 \in \F_p^\times$. So, $s_{12} \pi(e_2) s_{12}^{-1} = \pi(b_0 e_1 + e_2)$.

We now turn back to the larger groups $K_1 = \F_q^2 \rtimes \SL_n(\F_q)$ and $\Lambda_1$. For every $a \in \F_q^\times$, we define as before $\be_{i,a} = \pi \circ \Ad (a E_{ii} + (I_n - E_{ii})) \circ \pi^{-1}$. Applying the appropriate $\be_i(a)$, as well as $\be_\si$ for permutations $\si$, it follows that the elements $t_{ij}(a) := \pi(I_n + a E_{ij})$ satisfy: $t_{ij}(a)$ commutes with $\pi(b e_k)$ whenever $j \neq k$ and $a,b \in \F_q$, while $(\Ad t_{ij}(a))(\pi(b e_j)) = \pi(ab_0 b e_i + b e_j)$ for all $a,b \in \F_q$ and $i \neq j$.

Define $\cG < \Lambda_1$ as the subgroup generated by $N_1$ and the elements $t_{ij}(a)$. Denote by $\pi_0$ the restriction of $\pi$ to $\F_q^n$, so that $\pi_0 : \F_q^n \to N_1$ is a group isomorphism. By the previous paragraph, $N_1$ is a normal subgroup of $\cG$ and $\pi_0^{-1} \circ \Ad t_{ij}(a) \circ \pi_0 = I_n + a b_0 E_{ij}$ in $\Aut \F_q^n = \GL_n(\F_q)$. Since $\SL_n(\F_q)$ is generated by the elementary matrices, we find that $\Psi : \cG \to \SL_n(\F_q) : \Psi(g) = \pi_0^{-1} \circ \Ad g \circ \pi_0$ is a surjective group homomorphism.

By definition, $N_1 < \Ker \Psi$. So,
$$|\SL_n(\F_q)| = |\Im \Psi| = |\cG| / |\Ker \Psi| \leq |\cG| / |N_1| \leq |\Lambda_1| / |N_1| = |K_1| / |\F_q^n| = |\SL_n(\F_q)| \; .$$
It follows that all inequalities are equalities. Thus, $\cG = \Lambda_1$ and $\Ker \Psi = N_1$. In particular, $N_1$ is a normal subgroup of $\Lambda_1$ and we get the extension $e \to N_1 \to \Lambda_1 \to \SL_n(\F_q) \to e$.

Choose a map $\vphi : \SL_n(\F_q) \to \Lambda_1$ such that $\Psi \circ \vphi = \id$. By definition, $\vphi(A) \pi_0(a) \vphi(A)^{-1} = \pi_0(A \cdot a)$ for all $A \in \SL_n(\F_q)$ and $a \in \F_q^n$.

When $q \geq 3$, we have that $\SL_n(\F_q)$ is the fixed point subgroup of $K_1$ of $\Ad a I_n$ for all $a \in \F_q^\times$. Then $\pi(\SL_n(\F_q))$ is a subgroup of $\Lambda_1$ whose intersection with $N_1$ is trivial. Then
the restriction of $\Psi$ to $\pi(\SL_n(\F_q))$ is an isomorphism. We can thus choose $\vphi$ as the inverse of this isomorphism and conclude that $(a,A) \mapsto \pi_0(a) \vphi(A)$ defines an isomorphism $K_1 \cong \Lambda_1$.

Next assume that $q=2$, so that $n \geq 3$. We claim that the group $\Out \Lambda_1 := \Aut \Lambda_1 / \Inn \Lambda_1$ is abelian. Define $\cB_1 = \{\be \in \Aut \Lambda_1 \mid \forall a \in N_1 : \be(a) = a\}$. It then suffices to prove that $\cB_1$ is abelian and that for every automorphism $\be \in \Aut \Lambda_1$, there exists a $g \in \Lambda_1$ such that $(\Ad g) \circ \be \in \cB_1$.

Take $\be \in \cB_1$. We prove that $\Psi(\be(h)) = \Psi(h)$ for all $h \in \Lambda_1$. To prove this statement, define $A,B \in \SL_n(\F_2)$ by $A = \Psi(h)$ and $B = \Psi(\be(h))$. Then, $h \pi_0(a) h^{-1} = \pi_0(A \cdot a)$ for all $a \in \F_2^n$. Applying $\be$ and using that $\be$ acts as the identity on $N_1$, we get that $\be(h) \pi_0(a) \be(h)^{-1} = \pi_0(A \cdot a)$ for all $a \in \F_2^n$. But the left hand side equals $\pi_0(B \cdot a)$. We thus find that $B = A$ and we have proven that $\Psi \circ \be = \Psi$. It follows that $\be(\vphi(g)) = \mu(\be,g) \vphi(g)$ for certain elements $\mu(\be,g) \in N_1$.

Take $\be,\be' \in \cB_1$. Since $N_1$ is abelian and $\be,\be'$ act as the identity on $N_1$, it follows that $(\be \circ \be')(\vphi(g)) = (\be' \circ \be)(\vphi(g))$ for all $g \in \SL_n(\F_2)$. Since $\be \circ \be'$ and $\be' \circ \be$ both act as the identity on $N_1$, we get that $\be \circ \be' = \be' \circ \be$, so that $\cB_1$ is abelian.

Next take any $\be \in \Aut \Lambda_1$. Since $\Psi(\be(N_1))$ is an abelian normal subgroup of $\SL_n(\F_2)$ and since $\SL_n(\F_2) = \PSL_n(\F_2)$ is simple because $n \geq 3$, we get that $\be(N_1) = N_1$. Since $\GL_n(\F_2) = \SL_n(\F_2)$, we can take $A \in \SL_n(\F_2)$ such that $\be(\pi_0(a)) = \pi_0(A \cdot a)$ for all $a \in \F_2^n$. It follows that $(\Ad \vphi(A)^{-1}) \circ \be \in \cB_1$. So the claim is proven.

We now consider the group homomorphism $\Phi : K_1 \to \Aut \Lambda_1 : g \mapsto \pi \circ \Ad g \circ \pi^{-1}$, which is faithful because $K_1$ has trivial center. Since $n \geq 3$, the group $K_1$ is perfect. We have proven above that $\Out \Lambda_1$ is abelian. It follows that $\Phi(K_1) < \Inn \Lambda_1$. Since $\SL_n(\F_2)$ has trivial center and since its action on $\F_2^n$ has $0$ as the only fixed point, we get that the center of $\Lambda_1$ is trivial. We thus find a unique group homomorphism $\eta : K_1 \to \Lambda_1$ such that $\Phi(g) = \Ad \eta(g)$. Since $\Phi$ is faithful, also $\eta$ is faithful. Since $|K_1| = |\Lambda_1|$, it follows that $\eta$ is surjective, so that $\Lambda_1 \cong K_1$.
\end{proof}

\begin{question}\label{question.all-finite-simple-rigid}
We have proven in Theorems \ref{thm.rigid-SL2}, \ref{thm.relative-rigid-An-tilde} and \ref{thm.rigid-SLn-Fp} that the alternating groups and (almost all) the projective special linear groups over finite fields are rigid relative to their automorphism group. It is therefore natural to pose the following question, for which we expect the answer to be positive: is every finite simple group $K$ rigid relative to its automorphism group $\Aut K$~?
\end{question}

\subsection{Relative rigidity of direct products and counterexamples}

The following permanence property, in combination with Theorems \ref{thm.rigid-SL2}, \ref{thm.relative-rigid-An-tilde} and \ref{thm.rigid-SLn-Fp}, provides more examples of relatively rigid compact groups. The main reason to also include this result however is to illustrate how relative rigidity may fail in subtle ways.

\begin{proposition}\label{prop.direct-product-relative-rigid}
Let $K_0$ be a second countable connected compact abelian group and $n \geq 3$ an integer. Consider $K_1 = K_0^n$ and take any countable group $\Gamma$ such that $\SL_n(\Z) < \Gamma < \Autgr(K_1)$, as in Theorem \ref{thm.rigid-Kn}. Let $K_2$ be any finite group that is rigid relative to $\Aut K_2$.
\begin{enumlist}
\item If $K_2$ is generated by $\{\zeta(s) s^{-1} \mid s \in K_2, \zeta \in \Aut K_2\}$, then $K_1 \times K_2$ is rigid relative to $\Gamma \times \Aut K_2 < \Aut(K_1 \times K_2)$.
\item If $k \geq 2$, the group $K_1 \times S_k$ is not rigid relative to any action $\cG \actson K_1 \times S_k$.
\end{enumlist}
\end{proposition}

Note that the generating property of $\zeta(s) s^{-1}$ holds in particular when $K_2$ is a perfect group, by only considering the inner automorphisms. So combining Proposition \ref{prop.direct-product-relative-rigid} with Theorems \ref{thm.rigid-SL2}, \ref{thm.relative-rigid-An-tilde} and \ref{thm.rigid-SLn-Fp}, we find that the groups $\T^n \times \Atil_k$ with $n \geq 3$ and $k \geq 4$ and $k \neq 6$, as well as $\T^n \times \SL_k(\F_q)$ with $n,k \geq 3$ and $q$ a prime power, or $n \geq 3$, $k = 2$ and $q \geq 5$ prime, are rigid relative to the natural groups of automorphisms.

\begin{proof}
(i) We denote by $\al_g \in \Aut K_1$, $g \in \Gamma$, the given automorphisms. Let $\cT$ be a second countable compact group, $\pi : K_1 \times K_2 \to \cT$ a pmp isomorphism and assume that we have commuting group automorphisms $\be_g \in \Autgr(\cT)$ and $\be_\zeta \in \Autgr(\cT)$, for all $g \in \Gamma$ and $\zeta \in \Aut K_2$, such that
\begin{equation}\label{eq.equiv-0}
(\be_g \circ \be_\zeta) \circ \pi = \pi \circ (\al_g \times \zeta) \quad\text{a.e., whenever $g \in \Gamma$ and $\zeta \in \Aut K_2$.}
\end{equation}
Since the action $\Gamma \actson^\al K_1$ is weakly mixing, by the same reasoning as in the beginning of the proof of Theorem \ref{thm.rigid-Kn}, we find a finite group $\Lambda_2$ and a continuous surjective group homomorphism $\theta : \cT \to \Lambda_2$ such that $L^\infty(\cT)^{\be_\Gamma} = \{F \circ \theta \mid F \in \ell^\infty(\Lambda_2)\}$. Since we have that $L^\infty(K_1 \times K_2)^{\al_\Gamma \times \id} = 1 \ot \ell^\infty(K_2)$, we find a bijection $\pi_2 : K_2 \to \Lambda_2$ such that
\begin{equation}\label{eq.equiv-1}
\theta(\pi(k,s)) = \pi_2(s) \quad\text{for a.e.\ $(k,s) \in K_1 \times K_2$.}
\end{equation}
Since $\be_g$ and $\be_\zeta$ commute, the group automorphisms $\be_\zeta$ leave $L^\infty(\cT)^{\be_\Gamma}$ globally invariant. We thus find an action $(\gamma_\zeta)_{\zeta \in \Aut K_2}$ by automorphisms of $\Lambda_2$ such that $\gamma_\zeta \circ \theta = \theta \circ \be_\zeta$. Using \eqref{eq.equiv-1} and \eqref{eq.equiv-0}, it follows that $\gamma_\zeta \circ \pi_2 = \pi_2 \circ \zeta$ for all $\zeta \in \Aut K_2$. Since $K_2$ is rigid relative to $\Aut K_2$, we may thus assume that $\Lambda_2 = K_2$ and that $\zeta \mapsto \gamma_\zeta$ is an automorphism of $\Aut K_2$.

Define the open subgroup $\cS < \cT$ as $\cS = \Ker \theta$. By definition of $\theta$, we have that $\theta \circ \be_g = \theta$ for all $g \in \Gamma$, so that $\be_g$ restricts to a group of automorphisms $\be^1_g$ of $\cS$. Define $s_2 \in K_2$ by $s_2 = \pi_2^{-1}(e)$. By \eqref{eq.equiv-1}, the map $\pi_1 : K_1 \to \cS : \pi_1(k) = \pi(k,s_2)$ is a pmp isomorphism. By construction, for every $g \in \Gamma$, we have that $\be^1_g \circ \pi_1 = \pi_1 \circ \al_g$ a.e. By Theorem \ref{thm.rigid-Kn.one} and Remark \ref{rem.strictly-rigid-no-invariant-characters}, it follows that $\pi_1$ is a.e.\ equal to a group isomorphism $K_1 \to \cS$. Modifying $\pi$ on a set of measure zero, we may thus assume that $\pi_1 : K_1 \to \cS$ is a group isomorphism satisfying $\be_g(\pi_1(k)) = \pi_1(\al_g(k))$ for all $g \in \Gamma$ and $k \in K_1$.

Fix $s \in K_2$. By \eqref{eq.equiv-1}, $\pi$ restricts to a pmp isomorphism between $K_1 \times \{s\}$ and $\theta^{-1}(\pi_2(s))$. Choose a map $\psi : K_2 \to \cT$ such that $\theta \circ \psi = \pi_2$. We then find a pmp automorphism $\pi_s$ of $K_1$ such that $\pi(k,s) = \pi_1(\pi_s(k)) \psi(s)$ for a.e.\ $k \in K_1$. Since $\theta \circ \be_g = \theta$, we can uniquely define $\nu_g(s) \in K_1$ such that $\be_g(\psi(s)) = \pi_1(\nu_g(s)) \psi(s)$. Since $\be_g \circ \pi = \pi \circ (\al_g \times \id)$ and since $\be_g$ is an automorphism of $\cT$, it follows that for all $g \in \Gamma$ and $s \in K_2$,
$$\al_g(\pi_s(k)) \nu_g(s) = \pi_s(\al_g(k)) \quad\text{for a.e.\ $k \in K_1$.}$$
Define the pmp automorphism $\pi'_s : K_1 \times K_1 \to K_1 \times K_1$ by $\pi'_s(a,b) = (\pi_s(a) \pi_s(b)^{-1},b)$. We conclude that $\pi'_s \circ (\al_g \times \al_g) = (\al_g \times \al_g) \circ \pi'_s$ a.e.

Applying Theorem \ref{thm.rigid-Kn.one} and Remark \ref{rem.strictly-rigid-no-invariant-characters} to $K_1 \times K_1 = (K_0 \times K_0)^n$, we find that $\pi'_s$ is a.e.\ equal to a group automorphism $\theta_s$ of $K_1 \times K_1$. By definition, the second component of $\theta_s(a,b)$ is equal to $b$ for a.e.\ $(a,b) \in K_1 \times K_1$. By continuity, this equality holds everywhere and we conclude that $\theta_s(a,b) = (\theta_{s,b}(a),b)$ for all $a,b \in K_1$, where every $\theta_{s,b}$ is a group automorphism of $K_1$. Using the Fubini theorem, there exists a $b \in K_1$ such that $\pi_s(a) = \theta_{s,b}(a) \pi_s(b)$ for a.e.\ $a \in K_1$. We only retain that $\pi_s$ is a.e.\ equal to a homeomorphism of $K_1$. Modifying $\pi$ on a set of measure zero, we may thus assume that $\pi : K_1 \times K_2 \to \cT$ is a homeomorphism, so that the equalities \eqref{eq.equiv-0} and \eqref{eq.equiv-1} now hold everywhere.

The fixed point subgroup of $(\al_g \times \id)_{g \in \Gamma}$ equals $\{e\} \times K_2$. So, $\pi(\{e\} \times K_2)$ equals the set of fixed points of $(\be_g)_{g \in \Gamma}$ and is thus equal to a subgroup $\cR < \cT$. It then follows from \eqref{eq.equiv-1} that $\theta|_\cR : \cR \to K_2$ is a group isomorphism. Since the automorphisms $\be_\zeta$ commute with $\be_g$, we have that $\be_\zeta(\cR) = \cR$. So, $\theta \circ \beta_\zeta|_\cR = \gamma_\zeta \circ \theta|_\cR$.

Fix $\zeta \in \Aut K_2$. We prove that $\be_\zeta$ acts as the identity on $\cS = \Ker \theta$. We defined above $s_2 = \pi_2^{-1}(e)$. Since $\gamma_\zeta \circ \pi_2 = \pi_2 \circ \zeta$, it follows that $\zeta(s_2) = s_2$. Since $\cS = \pi(K_1 \times \{s_2\})$ and $\be_\zeta \circ \pi = \pi \circ (\id \times \zeta)$, it follows that $\be_\zeta(k) = k$ for all $k \in \cS$.

We finally prove that $\cR$ commutes with $\cS$. Define the group isomorphism $\vphi : K_2 \to \cR$ by $\vphi = (\theta|_{\cR})^{-1}$. Note that $\be_\zeta \circ \vphi = \vphi \circ \gamma_\zeta$ for every $\zeta \in \Aut K_2$. Since $\cS < \cT$ is a normal subgroup, we define for every $r \in K_2$, the automorphism $\eta_r \in \Aut \cS$ by $\eta_r(s) = \vphi(r)s\vphi(r)^{-1}$ for all $s \in \cS$. Take $\zeta \in \Aut K_2$. Since $\be_\zeta$ acts as the identity on $\cS$, applying $\be_\zeta$ to the definition of $\eta_r(s)$ implies that $\eta_{\gamma_\zeta(r)} = \eta_r$, so that $\eta_{\gamma_\zeta(r)r^{-1}} = \id$ for all $\zeta \in \Aut K_2$ and $r \in \cR$. Since $\zeta \mapsto \gamma_\zeta$ is an automorphism of $\Aut K_2$ and since $\{\zeta(r) r^{-1} \mid r \in K_2, \zeta \in \Aut K_2\}$ generates $K_2$, it follows that $\eta_r = \id$ for all $r \in \cR$, so that $\cS$ and $\cR$ commute.

We conclude that $\pi' : K_1 \times K_2 \to \cT : \pi'(k,s) = \pi_1(k) \, \vphi(s)$ is a group isomorphism. By construction, $(\be_g \circ \be_\zeta) \circ \pi' = \pi' \circ (\al_g \times \gamma_\zeta)$ for all $g \in \Gamma$ and $\zeta \in \Aut K_2$. Since $\zeta \mapsto \gamma_\zeta$ is an automorphism of $\Aut K_2$, it follows that $K_1 \times K_2$ is rigid relative to $\Gamma \times \Aut K_2$.

(ii) Denote by $\eps : S_k \to \{\pm 1\}$ the sign of a permutation. Since the kernel of $\eps$ equals the commutator subgroup of $S_k$, we have that $\eps \circ \zeta = \eps$ for every $\zeta \in \Aut S_k$. For the same reason, $(\psi(s))^2 = e$ for every homomorphism $\psi : S_k \to K_1$, because $K_1$ is abelian.

Define the group $\cT = K_1 \rtimes_\eta S_k$, where $\eta_s(k) = k^{\eps(s)}$. Denote by $\pi : K_1 \times S_k \to \cT$ the canonical bijection. Since $K_1$ is connected, every automorphism of $K_1 \times S_k$ leaves $K_1$ globally invariant and is thus of the form $(k,s) \mapsto (\al(k) \psi(s),\zeta(s))$ where $\al \in \Aut K_1$, $\zeta \in \Aut S_k$ and $\psi : S_k \to K_1$ is a group homomorphism. Since for every $s \in S_k$, $(\psi(s))^2 = e$, also $\eta_{s'}(\psi(s)) = \psi(s)$ for all $s,s' \in S_k$. Furthermore, $\eta_{\zeta(s)} = \eta_s$ for all $s \in S_k$. We conclude that all the maps $(k,s) \mapsto (\al(k) \psi(s),\zeta(s))$ also define automorphisms of the group $\cT$.

Note however that $\cT \not\cong K_1 \times S_k$. Indeed, defining $L < K_1$ by $\{k \in K_1 \mid k^2=e\}$, the group $L$ is totally disconnected. Since the center of $\cT$ is contained in $L \times S_k$, it is totally disconnected, while the center of $K_1 \times S_k$ contains $K_1$ as a connected subgroup. So, $K_1 \times S_k$ is not rigid relative to any action by automorphisms.
\end{proof}

\begin{proposition}\label{prop.no-go-connected-Lie}
No connected non abelian compact second countable group $K$ is rigid relative to any action $\cG \actson K$ by automorphisms.

More precisely, there exists a nontrivial totally disconnected group $L$ and a pmp isomorphism $\pi : K \to K \times L$ such that for every $\al \in \Aut K$, we have that $\pi \circ \al \circ \pi^{-1}$ is a.e.\ equal to a group automorphism of $K \times L$.
\end{proposition}
\begin{proof}
First assume that $K$ is a connected compact simple Lie group and that $L$ is any compact second countable group $L$. Denote by $Z < K$ the center of $K$, which is a finite group. Define $\cG = Z \rtimes \Aut K$ and consider the action $\theta$ of $\cG$ on $K$ by pmp automorphisms $\theta_{(z,\al)}(k) = z \al(k)$. We construct a pmp isomorphism $\pi : K \to K \times L$ such that $\pi \circ \theta_g \circ \pi^{-1} = \theta_g \times \id$ a.e., for all $g \in \cG$.

As discussed before \cite[Proposition 6.59]{HM23}, the subgroup $\Inn K$ of inner automorphisms of $K$ has finite index in $\Aut K$. Identifying $\Inn K = K/Z$, we turn $\Inn K$ into a compact Lie group. We then turn $\cG$ into a compact Lie group such that $\{e\} \times \Inn K$ is an open subgroup.

Fix a maximal torus $T < K$ (see e.g.\ \cite[Definition 6.22]{HM23}) and note that $Z < T$ by \cite[Corollary 6.32]{HM23}. Define the closed subgroup $\cA < \Aut K$ as $\cA := \{\al \in \Aut K \mid \al(T) = T\}$. Define $\cG_1 := Z \rtimes \cA$ and note that the action $\cG \actson^\theta K$ restricts to an action $\cG_1 \actson^\psi T$. By definition, $\Inn K \cap \cA$ is isomorphic with the Weyl group (see e.g.\ \cite[Definition 6.22]{HM23}), which is a finite group. Since $\Inn K < \cG$ has finite index, $\Inn K \cap \cA$ has finite index in $\cG_1$, so that the image $\Lambda = \psi_{\cG_1}$ is a finite group of pmp automorphisms of $T$.

Take an integer $n \geq 1$ such that $T \cong \T^n$. We equip $T \cong \T^n$ with the Lebesgue measure $\lambda_T$. We denote by $\That$ the Pontryagin dual of $T$. Since $T \cong \T^n$, for every $\om \in \That \setminus \{1\}$, the kernel $\Ker \om = \{t \in T \mid \om(t) = 1\}$ has measure zero. So, $T_1 := T \setminus \bigcup_{\om \in \That \setminus \{1\}} \Ker \om$ is a conull Borel set of $T$. When $t \in T_1$, every character $\om \in \That$ that is equal to $1$ on $t$, must be equal to $1$ everywhere. So, the closed subgroup generated by $t$ equals $T$, for every $t \in T_1$. Since $Z < T$ is a finite subgroup, also $T_2 := \bigcap_{z \in Z} z T_1$ is a conull Borel subset of $T$.

By construction, $\be(T_2) = T_2$ for every $\be \in \Lambda$. Define $\Fix \be = \{t \in T_2 \mid \be(t) = t\}$. Since $T = \T^n$, if $\be \neq \id$, $\Fix \be$ has measure zero. It follows that $T_3 := T_2 \setminus \bigcup_{\be \in \Lambda \setminus \{\id\}} \Fix \be$ is a conull Borel set of $T$. By construction, for every $\be \in \Lambda$, we have that $\be(T_3) = T_3$ and the action of the finite group $\Lambda$ on $T_3$ is free. We can thus choose a Borel subset $T_0 \subset T_3$ that meets each orbit of this action in exactly one point.

We equip $K$ and $\cG$ with their respective Haar probability measures $\lambda_K$ and $\lambda_\cG$. We claim that the map $\vphi : \cG / \Ker \psi \times T_0 \to K : (g,t) \mapsto \theta_g(t)$ is an injective Borel map whose image is conull in $(K,\lambda_K)$ and that satisfies $\vphi_*(\lambda_\cG \times \lambda_T) \sim \lambda_K$.

To prove that $\vphi$ is injective, it suffices to prove the following: if $s,t \in T_0$ and $(z,\al) \in \cG$ are such that $z \al(t) = s$, then $\al \in \cA$, $\psi(z,\al) = \id$ and $t=s$. Since $t,s \in T_0$, we get that $t \in T_2$ and $z^{-1} s \in T_2$, so that the closed subgroups generated by $t$, resp.\ $z^{-1} s$, are equal to $T$. Since $\al(t) = z^{-1} s \in T$ and $\al^{-1}(z^{-1} s) = t$, we conclude that $\al(T) = T$. So, $(z,\al) \in \cG_1$ and $\be := \psi(z,\al)$ belongs to $\Lambda$. Since $\be(t) = s$ and $s,t \in T_0$, it follows that $\be = \id$ and $t=s$.

By considering the derivative, it follows that $K/T \times T \to K : (k,t) \mapsto k t k^{-1}$ is locally a diffeomorphism. It follows that also $\Phi : \cG / \Ker \psi \times T \to K : (g,t) \mapsto \theta_g(t)$ is locally a diffeomorphism. Since every element of $K$ is conjugate to an element in $T$ (see e.g.\ \cite[Theorem 6.30]{HM23}), the map $\Phi$ is surjective. It then follows that the image of $\vphi$ is conull and that $\vphi_*(\lambda_\cG \times \lambda_T) \sim \lambda_K$.

Since $\lambda_K$ is invariant under the action $\cG \actson^\theta K$, the measure $(\vphi^{-1})_*(\lambda_K)$ on $\cG / \Ker \pi \times T_0$ is invariant under translation by $\cG$ in the first variable. It is therefore equal to $\lambda_\cG \times \mu_0$ for a probability measure $\mu_0$ on $T_0$. Since by the claim above, it is also equivalent to $\lambda_\cG \times \lambda_T$, we find that $\mu_0 \sim \lambda_T$. In particular, $\mu_0$ is nonatomic.

Take any compact second countable group $L$ with Haar probability measure $\lambda_L$. Since $(T_0,\mu_0)$ is a standard nonatomic probability space, we can choose a pmp isomorphism $\pi_0 : (T_0,\mu_0) \to (T_0 \times L,\mu_0 \times \lambda_L)$. Then, $\pi := (\vphi \times \id) \circ (\id \times \pi_0) \circ \vphi^{-1}$ defines a pmp isomorphism between $K$ and $K \times L$ such that $\pi \circ \theta_g \circ \pi^{-1} = \theta_g \times \id$ for all $g \in \cG$.

Now assume that $K$ is an arbitrary non abelian connected compact second countable group. By the Levi-Mal'cev structure theorem (see \cite[Theorem 9.24]{HM23}), we can choose a connected compact second countable group $K_0$, a connected simply connected compact simple Lie group $K_1$, a nonempty finite or countably infinite set $J$ and a closed subgroup $D < \cZ(K_0) \times \cZ(K_1)^J$ such that $K \cong (K_0 \times K_1^J)/D$ and such that $K_1/\cZ(K_1)$ is not a quotient of $K_0$. Moreover, combining \cite[Theorems 9.76(iv) and 9.86]{HM23}, for every $\al \in \Aut K$, there exists a permutation $\si_\al$ of the set $J$, an automorphism $\al_0 \in \Aut K_0$ and, for every $j \in J$, an automorphism $\be_j \in \Aut K_1$ such that the associated automorphism $\al_0 \times (\si_\al \circ (\prod_j \be_j))$ of $K_0 \times K_1^J$ globally preserves $D$ and defines $\al$ on the quotient $K$.

Take any nontrivial totally disconnected compact group $L_1$, e.g.\ $L_1 = \Z/2\Z$. As in the first part of the proof, we define $\cG_1 = \cZ(K_1) \rtimes \Aut K_1$ and consider the action $\cG_1 \actson^\theta K_1 : \theta_{(z,\al)}(k) = z\al(k)$ by pmp automorphisms. By the first part of the proof, we find a pmp isomorphism $\pi_1 : K_1 \to K_1 \times L_1$ such that $\pi_1 \circ \theta_g = (\theta_g \times \id) \circ \pi_1$ a.e., for all $g \in \cG_1$.

Define the pmp isomorphism $\pi_2 : K_0 \times K_1^J \to K_0 \times (K_1 \times L_1)^J$ by $\pi_2 = \id \times \pi_1^J$. Denote by $\eta$ the natural action of $\cG := (\cZ(K_0) \rtimes \Aut K_0) \times (\cZ(K_1) \rtimes \Aut K_1)^J$ on $K_0 \times K_1^J$. Identifying $K_0 \times (K_1 \times L_1)^J = (K_0 \times K_1^J) \times L_1^J$, we get that for every $g \in \cG$, we have $\pi_2 \circ \eta_g = (\eta_g \times \id) \circ \pi_2$ a.e. This holds in particular for all $g \in D < \cZ(K_0) \times \cZ(K_1)^J < \cG$. We can thus take the quotient by $D$ and define the pmp isomorphism $\pi : K \to K \times L_1^J$.

Every permutation $\si$ of $J$ defines an automorphism of $K_1$ and an automorphism of $L_1$. By construction, $\pi_2 \circ \si = (\id \times \si \times \si) \circ \pi_2$. From the description that we gave above for arbitrary automorphisms $\al \in \Aut K$, it now follows that $\pi \circ \al = (\al \times \si_\al) \circ \pi$ a.e. So, for every $\al \in \Aut K$, we have that $\pi \circ \al \circ \pi^{-1}$ is a.e.\ equal to the group automorphism $\al \times \si_\al$ of $K \times L$.
\end{proof}

\subsection{\boldmath Compact groups with trivial $2$-cohomology}\label{sec.about-trivial-2-cohomology}

In order to prove quantum W$^*$-superrigidity of certain co-induced left-right Bernoulli crossed products, we need as input Kac type compact quantum groups $(A_0,\Delta_0)$ that are rigid relative to an action $\Gamma \actson^\be (A_0,\Delta_0)$ and that have the following vanishing of $2$-cohomology: every unitary $2$-cocycle on $(A_0,\Delta_0)$ must be a coboundary and every unitary bicharacter on $(A_0,\Delta_0)$ (see Proposition \ref{prop.cocycles-products}) must be equal to $1$.

\begin{remark}\label{rem.compact-and-finite-group-reductions}
Let $K$ be a second countable compact group. The unitary $2$-cocycles on $(L^\infty(K),\Delta_K)$ are precisely the measurable $2$-cocycles $\Om : K \times K \to \T$. So, every unitary $2$-cocycle on $(L^\infty(K),\Delta_K)$ is a coboundary if and only if Moore's measurable $2$-cohomology $H^2(K,\T)$ is trivial (see \cite{Moo61}).

We next notice that every unitary bicharacter on $(L^\infty(K),\Delta_K)$ is $1$ if and only if $K$ has no open normal subgroup $K_0$ such that $K/K_0$ is a nontrivial abelian group. Indeed, by Lemma \ref{lem.bichar}, every unitary bicharacter in $\cU(A \ovt A)$ is given by a continuous map $\om : K \times K \to \T$ such that for all $k \in K$, the maps $\om(k,\cdot)$ and $\om(\cdot,k)$ belong to the discrete abelian group $\Hom(K,\T)$ of continuous group homomorphisms $K \to \T$. So, $\om$ induces a continuous group homomorphism $K \to \Hom(K,\T)$. Since $K$ is compact and $\Hom(K,\T)$ is discrete abelian, the range of this homomorphism must be a finite abelian group. So if $\om \neq 1$, $K$ admits a quotient that is a nontrivial finite abelian group. Conversely, if $K$ admits a quotient that is a nontrivial finite abelian group, there also exists a prime $p$ and a surjective continuous group homomorphism $\theta : K \to \Z/p\Z$. Then, $\om(x,y) = \exp((2\pi i / p) \theta(x)\theta(y))$ defines a unitary bicharacter that is not equal to $1$.

When $K$ is a finite group, $H^2(K,\T) \cong H^2(K,\C^\times)$ and this abelian group is called the \emph{Schur multiplier} $M(K)$ of $K$. So in this case, every unitary $2$-cocycle on $(L^\infty(K),\Delta_K)$ is a coboundary if and only if $M(K)=1$. Also, by the previous paragraph, every unitary bicharacter is $1$ if and only if $K$ is perfect.
\end{remark}

\begin{lemma}\label{lem.trivial-2-cohom-examples}
For the following Kac type compact quantum groups $(A,\Delta)$, every unitary $2$-cocycle is a coboundary and every unitary bicharacter is equal to $1$.
\begin{enumlist}
\item\label{lem.trivial-2-cohom-examples.one} $(A,\Delta) = (L^\infty(K),\Delta_K)$, where $K$ is one of the following compact groups.
\begin{enumlist}
\item $K$ is a connected compact abelian second countable group.
\item For every integer $n \geq 5$ with $n \not\in\{6,7\}$, the double cover $K=\Atil_n$ of the alternating group $A_n$.
\item For every integer $n \geq 2$ and prime power $q$ with $(n,q) \not\in \{(2,2), (2,3), (2,4), (2,9),\allowbreak (3,2), (3,4), (4,2)\}$, the group $K = \SL_n(\F_q)$.
\item For every integer $n \geq 3$ and prime power $q$ with $(n,q) \not\in \{(3,2), (3,4), (4,2)\}$, the group $K = \F_q^n \rtimes \SL_n(\F_q)$.
\item A direct product of the groups above.
\end{enumlist}
\item\label{lem.trivial-2-cohom-examples.two} $(A,\Delta) = (L(G),\Delta_G)$, where $G$ is a torsion free group.
\end{enumlist}
\end{lemma}
\begin{proof}
(i.a) By Remark \ref{rem.compact-and-finite-group-reductions} and by \cite[page 43]{Moo61}, it suffices to prove that every central extension $1 \to \T \overset{\vphi}{\to} \cK \overset{\pi}{\to} K \to 1$, where $\cK$ is a compact second countable group and the morphisms are continuous group homomorphisms, is split.

We choose a Borel map $\theta : K \to \cK$ such that $\pi \circ \theta = \id$. Since $K$ is abelian, the formula $\theta(a) \theta(b) = \vphi(\om(a,b)) \theta(b) \theta(a)$ defines a Borel bicharacter $\om : K \times K \to \T$. Because $K$ is connected and by Remark \ref{rem.compact-and-finite-group-reductions}, $\om(a,b) = 1$ for a.e.\ $(a,b) \in K \times K$. It follows that $\cK$ is abelian. We thus obtain the exact sequence $0 \to \widehat{K} \overset{\widehat{\pi}}{\to} \widehat{\cK} \overset{\widehat{\vphi}}{\to} \Z \to 0$. Choosing any lift of $1 \in \Z$, it follows that this sequence is split. We can then also define a group homomorphism $\widehat{\theta} : \widehat{\cK} \to \widehat{K}$ such that $\widehat{\theta} \circ \widehat{\pi} = \id$. Its dual $\theta : K \to \cK$ is a continuous group homomorphism satisfying $\pi \circ \theta = \id$.

(i.b) For the proofs of (b), (c) and (d), we make use of Remark \ref{rem.compact-and-finite-group-reductions}. Take $n \geq 5$ with $n \not\in\{6,7\}$ and denote $K = \Atil_n$. By Lemma \ref{lem.better-generators}, $K$ is perfect. Combining \cite[Proposition 1.12 in Chapter 10]{Kar93} with \cite[Theorem 3.2 in Chapter 12]{Kar93}, the Schur multiplier of $K$ is trivial.

(i.c) By \cite[Theorems 2.3 and 3.2 in Chapter 16]{Kar93}, these groups $K$ are perfect and have trivial Schur multiplier.

(i.d) Write $K = \F_q^n \rtimes \SL_n(\F_q)$. To prove that $K$ is perfect, fix a homomorphism $\om : K \to \T$. As mentioned in (i.c), $\SL_n(\F_q)$ is perfect. So, $\om$ is equal to $1$ on $\SL_n(\F_q)$. Fix a group homomorphism $\psi_0 : (\F_q,+) \to (\T,\cdot)$ satisfying $\psi_0(1) = \exp(2\pi i/p)$. Then every homomorphism $\mu : \F_q \to \T$ is of the form $\mu(a) = \psi_0(a_1 a)$ for a unique $a_1 \in \F_q$. We thus find a unique column matrix $a_0 \in \F_q^n$ such that $\om(a) = \psi_0(a_0^T a)$ for all $a \in \F_q^n$. Since the restriction of $\om$ to $\F_q^n$ is $\SL_n(\F_q)$-invariant, it follows that $A^T \cdot a_0 = a_0$ for all $A \in \SL_n(\F_q)$. So, $a_0 = 0$ and $\om = 1$, implying that $K$ is perfect.

To prove that $H^2(K,\T) = 1$, it suffices to prove that every central extension $1 \to \T \to \cK \overset{\pi}{\to} K \to 1$ is split. Choose a lift $\zeta_0 : \F_q^n \to \cK$. Because $\F_q^n$ is abelian, we can define the map $\Om : \F_q^n \times \F_q^n \to \T$ by $\zeta_0(a) \zeta_0(b) = \Om(a,b) \zeta_0(b) \zeta_0(a)$ for all $a,b \in \F_q^n$. Then $\Om$ is a bihomomorphism: for all $a,b \in \F_q^n$, the maps $\Om(\cdot,b)$ and $\Om(a,\cdot)$ are group homomorphisms. Fix $b \in \F_q^n$. Since $\Om(\cdot,b)$ is a group homomorphism, we find a unique element $D(b) \in \F_q^n$ such that $\Om(a,b) = \psi_0(a^T D(b))$. Since $\Om(a,\cdot)$ is a group homomorphism, we get that $D : \F_q^n \to \F_q^n$ is an additive group homomorphism.

Conjugating the defining relation of $\Om$ by a lift of $A \in \SL_n(\F_q)$, it follows that $\Om(A \cdot a, A \cdot b) = 1$ for all $a,b \in \F_q^n$. This means that $D(A \cdot b) = \eta(A) \cdot D(b)$ for all $b \in \F_q^n$, where $\eta(A) = (A^T)^{-1}$. Consider the basis vector $e_1 \in \F_q^n$ and define the subgroup $\Gamma_1 < \SL_n(\F_q)$ of matrices $A$ satisfying $A \cdot e_1 = e_1$. Since $n \geq 3$, $0$ is the only element of $\F_q^n$ that is fixed by $\eta(A)$ for all $A \in \Gamma_1$. So, $D(e_1) = 0$. It follows that $D(A \cdot e_1) = 0$ for all $A \in \SL_n(\F_q)$. This means that $D(b) = 0$ for all $b \in \F_q^n \setminus \{0\}$. Since also $D(0) = 0$, we get that $D = 0$ and $\Om = 1$. So, $\zeta(a)$ commutes with $\zeta(b)$ for all $a,b \in \F_q^n$. Since every element of $\T$ has a $p$'th root, we may thus assume that the lift $\zeta_0 : \F_q^n \to \cK$ is a homomorphism.

As mentioned in (i.c), the group $\SL_n(\F_q)$ has a trivial Schur multiplier. We may thus also choose a homomorphic lift $\zeta_1 : \SL_n(\F_q) \to \cK$. We then define $\gamma : \SL_n(\F_q) \times \F_q^n \to \T$ by $\zeta_1(A) \zeta_0(a) = \gamma(A,a) \zeta_0(A \cdot a) \zeta_1(A)$. It follows that for every $A \in \SL_n(\F_q)$, the map $\gamma(A,\cdot) : \F_q^n \to \T$ is a group homomorphism. We can thus uniquely define a map $\mu : \SL_n(\F_q) \to \F_q^n$ such that $\gamma(A,a) = \psi_0(\mu(A)^T a)$.

From the definition of $\gamma$, we also get that $\gamma(AB,a) = \gamma(A,B\cdot a) \gamma(B,a)$, so that $\mu(AB) = B^T \cdot \mu(A) + \mu(B)$ for all $A,B \in \SL_n(\F_q)$. Then, $c(A) := A \cdot \mu((A^T)^{-1})$ defines a $1$-cocycle $c : \SL_n(\F_q) \to \F_q^n$ as in Lemma \ref{lem.1-cohom-Fpn}. By that lemma, we find $a_0 \in \F_q^n$ such that $c(A) = a_0 - A \cdot a_0$ for all $A \in \SL_n(\F_q)$. Replacing $\zeta_0$ by the homomorphism $a \mapsto \psi_0(a_0^T a) \zeta_0(a)$, the homomorphisms $\zeta_0$ and $\zeta_1$ now satisfy $\zeta_1(A) \zeta_0(a) = \zeta_0(A \cdot a) \zeta_1(A)$ for all $A \in \SL_n(\F_q)$ and $a \in \F_q^n$. They thus combine into a homomorphism $\F_q^n \rtimes \SL_n(\F_q) \to \cK$, so that the central extension is split.

(i.e) Since none of the groups in a)--d) admit a nontrivial finite abelian quotient, also all bicharacters on their products are trivial. The conclusion then follows from Proposition \ref{prop.cocycles-products}.

(ii) Since both $G$ and $G \times G$ are torsion free, it follows from Proposition \ref{prop.cocycle-reduce-to-core.one} that every unitary $2$-cocycle on $(L(G),\Delta_G)$ and $(L(G \times G),\Delta_{G \times G})$ is a coboundary. By Proposition \ref{prop.cocycles-products}, it follows that all unitary bicharacters on $(L(G),\Delta_G)$ are trivial.
\end{proof}

\begin{lemma}\label{lem.1-cohom-Fpn}
Let $n \geq 3$ be an integer and $q$ a prime power with $(n,q) \neq (3,2)$. If $c : \SL_n(\F_q) \to \F_q^n$ is a map satisfying $c(AB) = c(A) + A \cdot c(B)$ for all $A,B \in \SL_n(\F_q)$, there exists an $a_0 \in \F_q^n$ such that $c(A) = A \cdot a_0 - a_0$ for all $A \in \SL_n(\F_q)$.
\end{lemma}

Note that in Remark \ref{rem.nontrivial-n-q-3-2}, we show that the conclusion of Lemma \ref{lem.1-cohom-Fpn} does not hold with $(n,q) = (3,2)$.

\begin{proof}
For all $i \neq j$ and $a \in \F_q$, write $c_{ij}(a) = c(I_n + a E_{ij})$. Denote by $e_i \in \F_q^n$ the standard basis elements. We claim that it is sufficient to prove that there exist elements $a_j \in \F_q$ such that $c_{ij}(a) = a a_j e_i$ for all $i \neq j$ and $a \in \F_q$. Indeed, next defining $a_0 = \sum_{j=1}^n a_j e_j$, it follows that $c(I_n + a E_{ij}) = (I_n + a E_{ij}) \cdot a_0 - a_0$ for all $i \neq j$ and $a \in \F_q$. Since the elementary matrices generate $\SL_n(\F_q)$, we then conclude that $c(A) = A \cdot a_0 - a_0$ for all $A \in \SL_n(\F_q)$.

Note that the $1$-cocycle relation implies that $c(I_n) = 0$ and $c(A^{-1}) = - A^{-1} \cdot c(A)$. If $i,j,k$ are distinct, applying the $1$-cocycle relation to $[I_n + a E_{ij}, I_n + b E_{jk}] = I_n + ab E_{ik}$ gives
\begin{equation}\label{eq.nice}
c_{ik}(ab) = - b(a E_{ik} + E_{jk}) \cdot c_{ij}(a) + a (E_{ij} - b E_{ik}) \cdot c_{jk}(b)
\end{equation}
for all distinct $i,j,k$ and $a,b \in \F_q$.  It follows from \eqref{eq.nice} that $c_{ik}(a) \in \F_q e_i + \F_q e_j$ whenever $i,j,k$ are distinct.

First assume that $n \geq 4$. Let $i,k$ be distinct. Since $n \geq 4$, we can choose $j,r$ such that $i,k,j,r$ are all distinct. Then $c_{ik}(a)$ belongs to both $\F_q e_i + \F_q e_j$ and $\F_q e_i + \F_q e_r$. Thus, $c_{ik}(a) \in \F_q e_i$. We define the maps $A_{ik} : \F_q \to \F_q$ such that $c_{ik}(a) = A_{ik}(a) e_i$. Then \eqref{eq.nice} says that $A_{ik}(ab) = a A_{jk}(b)$ whenever $i,j,k$ are distinct and $a,b \in \F_q$. Defining $a_{jk} \in \F_q$ by $a_{jk} = A_{jk}(1)$, it follows that $A_{ik}(a) = a a_{ik}$ and that $a_{ik} = a_{jk}$ when $i,j,k$ are distinct. The latter implies the existence of elements $a_k \in \F_q$ such that $a_{ik} = a_k$ for all $i \neq k$, so that $c_{ik}(a) = a a_k e_i$, and the lemma is proven in this case.

Next assume that $n = 3$. Since $c_{jk}(a) \in \F_q e_j + \F_q e_i$ whenever $i,j,k$ are distinct, one of the terms in \eqref{eq.nice} is zero and we get that
\begin{equation}\label{eq.nice-2}
c_{ik}(ab) = - b(a E_{ik} + E_{jk}) \cdot c_{ij}(a) + a E_{ij} \cdot c_{jk}(b)
\end{equation}
for all distinct $i,j,k$ and $a,b \in \F_q$. Write $q = p^m$. We first consider the case where $p$ is odd. Since $c_{ik}(a) \in \F_q e_i + \F_q e_j$, we get that $E_{ik} \cdot c_{ik}(a) = 0$ for all $i \neq k$ and $a \in \F_q$. So, applying the $1$-cocycle relation to the equality $(I_n + a E_{ik})(I_n + b E_{ik}) = I_n + (a+b)E_{ik}$, it follows that $c_{ik}(a+b) = c_{ik}(a) + c_{ik}(b)$ for all $a,b \in \F_q$ and $i \neq k$. In particular, $c_{ik}(d a) = d c_{ik}(a)$ for all $d \in \F_p$, $a \in \F_q$ and $i \neq k$. We use this in \eqref{eq.nice-2} with $b=1$ and $da$ instead of $a$ with $d \in \F_p^\times$ and $a \in \F_q$. Dividing by $d$, we find that
$$c_{ik}(a) = -(d a E_{ik} + E_{jk}) \cdot c_{ij}(a) + a E_{ij} \cdot c_{jk}(1)$$
for all distinct $i,j,k$, $a \in \F_q$ and $d \in \F_p^\times$. In particular, the term $d a E_{ik} \cdot c_{ij}(a)$ does not depend on $d \in \F_p^\times$. Since $p \geq 3$, it follows that $a E_{ik} \cdot c_{ij}(a)=0$ for all distinct $i,j,k$ and $a \in \F_q$. Since $c_{ij}(0) = 0$, we have proven that $c_{ij}(a) \in \F_q e_i$ for all $i \neq j$ and $a \in \F_q$. Then \eqref{eq.nice-2} with $b=1$ says that $c_{ik}(a) = a E_{ij} \cdot c_{jk}(1)$ for all distinct $i,j,k$ and $a \in \F_q$. With $a_{ik} := c_{ik}(1)$, it follows that $c_{ik}(a) = a_{ik} e_i$ and $a_{ik} = a_{jk}$ for all distinct $i,j,k$. We thus find unique elements $a_k \in \F_q$ such that $a_{ik} = a_k$ and thus $c_{ij}(a) = a a_j e_i$ for all $i \neq j$ and $a \in \F_q$. Again the lemma is proven in this case.

Finally assume that $n=3$ and $p=2$. So, $\F_q$ has characteristic $2$. Since we excluded $(n,q) = (3,2)$, we get that $q \geq 4$. Since $c_{ij}(a) \in \F_q e_i + \F_q e_k$, we define the functions $A_{ij} : \F_q \to \F_q$ and $B_{ij} : \F_q \to \F_q$ such that $c_{ij}(a) = A_{ij}(a) e_i + B_{ij}(a) e_k$, where $k \in \{1,2,3\}$ is the unique element such that $\{i,j,k\} = \{1,2,3\}$. Multiplying \eqref{eq.nice-2} with $E_{ii}$ and $E_{jj}$, it follows that
\begin{equation}\label{eq.nice-3}
A_{ik}(ab) = a b B_{ij}(a) + a A_{jk}(b) \quad\text{and}\quad B_{ik}(ab) = b B_{ij}(a)
\end{equation}
for all distinct $i,j,k$ and $a,b \in \F_q$. By the second equation, we find elements $b_i \in \F_q$ such that $B_{ij}(a) = a b_i$ for all $i \neq j$ and $a \in \F_q$. We define $a_{ij} \in \F_q$ by $a_{ij} := A_{ij}(1)$. Applying \eqref{eq.nice-3} with $b=1$, we conclude that $A_{ik}(a) = a^2 b_i + a a_{jk}$ for all distinct $i,j,k$ and $a \in \F_q$. With this expression, \eqref{eq.nice-3} becomes
$$a^2 b^2 b_i + a b a_{jk} = a^2 b b_i + a b^2 b_i + a b a_{ik} \quad\text{for all distinct $i,j,k$ and $a,b \in \F_q$.}$$
When $a,b \in \F_q^\times$, we can divide by $ab$ and find that $a b b_i + a_{jk} = a b_i + b b_i + a_{ik}$. Taking $a = b$, we conclude that $a^2 b_i = a_{ik} + a_{jk}$ for all distinct $i,j,k$ and $a \in \F_q^\times$. So, $a^2 b_i$ does not depend on $a \in \F_q^\times$. Since $q \geq 4$, it follows that $b_i = 0$ for all $i$. So also $B_{ij}(a)=0$ for all $i \neq j$ and $a \in \F_q$. We get as well that $a_{ik} = a_{jk}$ for all distinct $i,j,k$. We thus find elements $a_k \in \F_q$ such that $a_{ik} = a_k$ whenever $i \neq k$. Going back to \eqref{eq.nice-3} with $b=1$, we have proven that $A_{ik}(a) = a a_k$, so that $c_{ij}(a) = a a_j e_i$ for all $i \neq j$ and $a \in \F_q$. This concludes the proof of the lemma.
\end{proof}

\begin{remark}\label{rem.nontrivial-n-q-3-2}
In case $(n,q) = (3,2)$, the conclusion of Lemma \ref{lem.1-cohom-Fpn} fails. Indeed, define the elements $c_{ij} \in \F_2^3$ by
$$c_{12} = e_3 \; , \quad c_{23} = e_1+e_2 \; , \quad c_{31} = e_2 \;  ,\quad c_{13} = e_2 \; , \quad c_{32} = e_1 + e_3 \; , \quad c_{21} = e_2 + e_3 \; .$$
By e.g.\ \cite[Theorem 1.16 in Chapter 16]{Kar93}, the group $\SL_3(\F_2)$ is the universal group with generators $s_{ij}=I_3 + E_{ij}$ for all distinct $i,j$ in $\{1,2,3\}$ and relations
$$s_{ij}^2 = e \; , \quad [s_{ik},s_{jk}] = e = [s_{ij},s_{ik}] \; , \quad [s_{ij},s_{jk}] = s_{ik} \quad\text{for all distinct $i,j,k$.}$$
It is straightforward to check that the assignment $I_n + E_{ij} \mapsto c_{ij}$ respects these relations, so that there is a unique $1$-cocycle $c : \SL_3(\F_2) \to \F_2^3$ satisfying $c(I_3 + E_{ij}) = c_{ij}$ for all $i \neq j$. If $a_0 = (a_1,a_2,a_3) \in \F_2^3$ would satisfy $c(A) = A \cdot a_0 - a_0$ for all $A \in \SL_3(\F_2)$, we get that $c_{ij} = a_j e_i$ for all $i \neq j$, which is not the case.

So, the $1$-cohomology of the action $\SL_3(\F_2) \actson \F_2^3$ is nontrivial. One can easily show that the above $1$-cocycle is the unique nontrivial $1$-cocycle, up to a coboundary.
\end{remark}

In statement d) of Lemma \ref{lem.trivial-2-cohom-examples.one}, it is necessary to take $n \geq 3$, as the following lemma shows.

\begin{lemma}\label{lem.Schur-nontrivial-SL_2-F_2}
For every prime power $q$, the Schur multiplier of $\F_q^2 \rtimes \SL_2(\F_q)$ is nontrivial .
\end{lemma}
\begin{proof}
When $q=2$, we observed at the start of the proof of Theorem \ref{thm.rigid-SLn-Fp.four} that $\F_2^2 \rtimes \SL_2(\F_2) \cong S_4$. By e.g.\ \cite[Theorem 2.2 in Chapter 12]{Kar93}, the Schur multiplier of $S_4$ is nontrivial. We may thus assume that $q \neq 2$.

As in the proof of statement d) of Lemma \ref{lem.trivial-2-cohom-examples.one}, choose a group homomorphism $\psi_0 : (\F_q,+) \to (\T,\cdot)$ such that $\psi_0(1) = \exp(2\pi i / p)$. Define the bihomomorphism $\om : \F_q^2 \times \F_q^2 \to \T$ by $\om(a,b) = \psi_0(a_1 b_2 - a_2 b_1)$. First, $\om(A \cdot a,A \cdot b) =\om(a,b)$ for all $A \in \SL_2(\F_q)$ and $a,b \in \F_q^2$, so that $\Om((a,A),(b,B)) = \om(a,A \cdot b)$ defines a $\T$-valued $2$-cocycle on $\F_q^2 \rtimes \SL_2(\F_q)$.

It suffices to prove that $\Om$ is not a coboundary. Write $q = p^m$ for a prime $p$. When $p$ is odd, we find that $\Om((a,I_2),(b,I_2)) \, \overline{\Om((b,I_2),(a,I_2))}$ is not identically $1$, so that $\Om$ is not a coboundary.

For the rest of the proof, assume that $p=2$ and thus $q \geq 4$, since we are assuming that $q\neq 2$. Denote by $1 \to \T \to \cG \to \F_q^2 \rtimes \SL_2(\F_q) \to e$ the central extension given by $\Om$, with a lift $\vphi : \F_q^2 \rtimes \SL_2(\F_q) \to \cG$ satisfying $\vphi(a,A) \vphi(b,B) = \Om((a,A),(b,B)) \vphi(a + A \cdot b , AB)$. It follows that $\gamma : \F_q^2 \to \cG : \gamma(a) = \psi_0(a_1 a_2) \vphi(a,I_2)$ is a group homomorphism.

Assume that $\Om$ is a coboundary. There then exists a homomorphic lift $\zeta : \F_q^2 \to \cG$ such that $\vphi(0,A) \zeta(a) \vphi(0,A)^{-1} = \zeta(A \cdot a)$ for all $a \in \F_q^2$ and $A \in \SL_2(\F_q)$. Since $\gamma$ is a homomorphic lift, we find a group homomorphism $\eta : \F_q^2 \to \T$ such that $\zeta(a) = \eta(a) \gamma(a)$ for all $a \in \F_q^2$. Take $b_1,b_2 \in \F_q$ such that $\eta(a) = \psi_0(b_1 a_1 + b_2 a_2)$ for all $a \in \F_q^2$. Define the map $\Psi : \F_q^2 \to \T : \Psi(a) = \psi_0(a_1 a_2 + b_1 a_1 + b_2 a_2)$. We have thus shown that $\zeta(a) = \Psi(a) \vphi(a,I_2)$.

Since $\vphi(0,A) \vphi(a,I_2) \vphi(0,A)^{-1} = \vphi(A \cdot a,I_2)$, we get that $\Psi(A \cdot a) = \Psi(a)$ for all $A \in \SL_2(\F_q)$ and $a \in \F_q^2$. Since the action of $\SL_2(\F_q)$ on $\F_q^2 \setminus \{(0,0)\}$ is transitive, we find $\eps \in \T$ such that $\Psi(a) = \eps$ for all $a \in \F_q^2 \setminus \{(0,0)\}$. Fix $a_1 \in \F_q^\times$. Since $\Psi(a_1,0) = \eps$, we get that $\psi_0(b_1 a_1) = \eps$. It follows that $\psi_0((a_1+b_2)a_2) = 1$ for all $a_2 \in \F_q$. This implies that $a_1 + b_2 = 0$. We have thus proven that every nonzero element of $\F_q$ is equal to $b_2$, which is absurd since $q \geq 4$.
\end{proof}

\section{\boldmath Quantum W$^*$-superrigidity}\label{sec.quantum-Wstar-superrigidity}

As explained in the introduction, a discrete group $G$ is called W$^*$-superrigid if for every discrete group $\Lambda$ with $L(G) \cong L(\Lambda)$, we have that $G \cong \Lambda$. In the most rigid situations, e.g.\ already in \cite[Theorem 1.1]{IPV10}, we moreover have that any isomorphism $L(G) \to L(\Lambda)$ is automatically the composition of an inner automorphism and a group like isomorphism given by combining a group isomorphism $G \to \Lambda$ and a character $G \to \T$ on $G$ (see Section \ref{sec.quantum-group-like-aut}). In Definition \ref{def.quantum-group-like-aut}, we introduced the analogous concept of a \emph{quantum group like} isomorphism, combining a quantum group isomorphism with a left translation automorphism.

We thus introduce the following form of quantum W$^*$-superrigidity, which is a stricter version of Definition \ref{def.quantum-Wstar-superrigid}.

\begin{definition}\label{definition:strict-quantum-wstar-superrigidity}
We say that a compact quantum group $(A,\Delta_A)$ is \emph{strictly quantum W$^*$-superrigid} if the following holds: if $(B,\Delta_B)$ is any compact quantum group and $\pi: A \to B$ is any von Neumann algebra isomorphism, there exists a unitary $v \in B$ and a quantum group like isomorphism $\pi_0 : (A,\Delta_A) \to (B,\Delta_B)$, in the sense of Definition \ref{def.quantum-group-like-aut}, such that $\pi = (\Ad v) \circ \pi_0$.
\end{definition}

To formulate our main quantum W$^*$-superrigidity theorem, we fix a Kac type compact quantum group $(A_0,\Delta_0)$ with Haar state $\tau_0$ and we assume that $A_0$ is an amenable von Neumann algebra. Note that by \cite[Theorem 4.5]{Rua95}, the amenability of the tracial von Neumann algebra $(A_0,\tau_0)$ is equivalent with every variant of amenability of the discrete dual quantum group of $(A_0,\Delta_0)$, as well as with every variant of co-amenability of $(A_0,\Delta_0)$.

Take a group $\Gamma$ in class $\cC$. Let $\Gamma \actson^\be (A_0,\Delta_0)$ be an action by quantum group automorphisms. This means that $\Delta_0 \circ \be_g = (\be_g \ot \be_g) \circ \Delta_0$ for all $g \in \Gamma$. Assume that the kernel of $\be$ is infinite.

As in Section \ref{sec.coarse-embeddings-co-induced}, we consider the co-induced left-right Bernoulli action $\Gamma \times \Gamma \actson^\al (A,\tau) = (A_0,\tau_0)^\Gamma$. As before, we denote by $\pi_k : A_0 \to A$ the embedding as the $k$'th tensor factor and we recall that $\al_{(g,h)}(\pi_k(a)) = \pi_{gkh^{-1}}(\be_g(a))$. The tensor product $(A,\tau)$ naturally is a compact quantum group; see the discussion before Proposition \ref{prop.cocycles-products}. Note that all $\al_{(g,h)}$ are quantum group automorphisms. So also the crossed product $M = A \rtimes_\al (\Gamma \times \Gamma)$ naturally becomes a compact quantum group $(M,\Delta)$, as explained before Proposition \ref{prop.cocycle-reduce-to-core}. By definition $\Delta(u_g) = u_g \ot u_g$ for all $g \in \Gamma \times \Gamma$ and $\Delta \circ \pi_e = (\pi_e \ot \pi_e) \circ \Delta_0$.

\begin{theorem}\label{thm.generic-theorem}
Define $(M,\Delta)$ as above. The following statements are equivalent.
\begin{enumlist}
\item\label{thm.generic-theorem.1} $(M,\Delta)$ is quantum W$^*$-superrigid in the sense of Definition \ref{def.quantum-Wstar-superrigid}.
\item\label{thm.generic-theorem.2} $(M,\Delta)$ satisfies the following stronger quantum W$^*$-superrigidity property: if $(Q,\Delta_Q)$ is any other compact quantum group and $p \in M$ is a nonzero projection such that $pMp \cong Q$ as von Neumann algebras, then $p=1$ and $(M,\Delta) \cong (Q,\Delta_Q)$ as compact quantum groups.
\item\label{thm.generic-theorem.3} Our initial data satisfy the following properties.
\begin{enumlist}
\item The group $\Gamma$ is torsion free.
\item Every unitary $2$-cocycle of $(A_0,\Delta_0)$ is a coboundary and every unitary bicharacter on $(A_0,\Delta_0)$ is equal to $1$ (see Proposition \ref{prop.cocycles-products}).
\item $(A_0,\Delta_0)$ is rigid relative to $\Gamma \actson^\be (A_0,\Delta_0)$, in the sense of Definition \ref{def.relative-rigidity-cqg.two}.
\end{enumlist}
\end{enumlist}
Also the following two statements are equivalent.
\begin{enumRom}
\item The compact quantum group $(M,\Delta)$ is strictly quantum W$^*$-superrigid in the sense of Definition \ref{definition:strict-quantum-wstar-superrigidity}.
\item Conditions (iii.a) and (iii.b) hold and $(A_0,\Delta_0)$ is strictly rigid relative to the subgroup $\be(\Gamma)$ of $\Aut(A_0,\Delta_0)$, in the sense of Definition \ref{def.relative-rigidity-cqg.one}.
\end{enumRom}
\end{theorem}

The proof of Theorem \ref{thm.generic-theorem} uses the comultiplication method introduced in \cite{PV09,IPV10}. If a given group von Neumann algebra $(L(G),\Delta_G)$ or compact quantum group von Neumann algebra $(M,\Delta)$ carries another compact quantum group structure, this provides another comultiplication embedding $\psi : M \to M \ovt M$. We use Popa's deformation/rigidity theory to prove that, essentially, $\psi$ and $\Delta$ must be unitarily conjugate. We first establish such a generic classification result for embeddings $M \to M \ovt M$ in Section \ref{sec.coarse-embeddings-co-induced}, and then give the proof of Theorem \ref{thm.generic-theorem} in Section \ref{sec.proof-superrigidity-theorem}.

\begin{remark}\label{rem.no-go-H2}
We consider the restrictions imposed on $n$ and $q$ in Theorem \ref{thm.main.two}. When $(n,q) \in \{(3,2),(3,4),(4,2)\}$, it follows from \cite[Theorem 3.2 in Chapter 16]{Kar93} that $\SL_n(\F_q)$ has a nontrivial Schur multiplier. Then the same holds for $\F_q^n \rtimes \SL_n(\F_q)$. When $n=6,7$, the same holds for $\Atil_n$ by \cite[Theorem 3.2 in Chapter 12]{Kar93}. When $n = 4$, the group $A_4$ is not perfect and therefore, also $\Atil_4$ is not perfect.

When $n=2$, it follows from Lemma \ref{lem.Schur-nontrivial-SL_2-F_2} that $\F_q^2 \rtimes \SL_2(\F_q)$ has a nontrivial Schur multiplier.

So in all the cases mentioned in the previous two paragraphs, by Theorem \ref{thm.generic-theorem} and Remark \ref{rem.compact-and-finite-group-reductions}, the associated compact quantum group $(M,\Delta)$ is not quantum W$^*$-superrigid.

It remains to discuss what happens if we take $K = \SL_2(\F_q)$. When $q=2$ or $q=3$, the group $K$ is not perfect. When $q=4$ or $q=9$, the Schur multiplier of $K$ is nontrivial. When $q$ is prime and $q \geq 5$, W$^*$-superrigidity holds. For all other values of $q$, i.e.\ prime powers that are not prime and not equal to $4$ or $9$, the group $K = \SL_2(\F_q)$ is perfect and has a trivial Schur multiplier. It is plausible that these groups $K$ are rigid relative to their automorphism group so that quantum W$^*$-superrigidity holds, but we were unable to prove this.
\end{remark}

\subsection{Coarse embeddings of co-induced Bernoulli crossed products}\label{sec.coarse-embeddings-co-induced}

Comultiplication embeddings $\psi : M \to M \ovt M$ are always \emph{coarse embeddings} in the sense of \cite[Definition 5.1]{DV24a}; see Lemma \ref{lem.comult-coarse}. By definition, this means that the bimodules $\bim{\psi(M)}{L^2(M \ovt M)}{M \ot 1}$ and $\bim{\psi(M)}{L^2(M \ovt M)}{1 \ot M}$ are coarse.

A key ingredient in the proof of W$^*$-superrigidity is to provide a complete classification of all possible coarse embeddings $\psi : M \to M \ovt M$ when $M$ is a II$_1$ factor such as in Theorem \ref{thm.main}. We prove such a classification theorem in this section, generalizing \cite[Theorem~5.11]{DV24a}.

Recall from \cite[Definition 3.1]{PV21} that a countable group $\Gamma$ is said to belong to class $\cC$ if $\Gamma$ is nonamenable, weakly amenable, biexact, and every nontrivial element of $\Gamma$ has an amenable centralizer. This class of groups contains all torsion free hyperbolic groups, all free groups $\F_n$ with $2 \leq n \leq \infty$, and all free products $\Gamma_1 \ast \Gamma_2$ of nontrivial amenable groups with $|\Gamma_2| \geq 3$.

In \cite[Theorem~5.11]{DV24a}, we classified coarse embeddings $M \to M_1 \ovt \cdots \ovt M_k$ when the II$_1$ factors $M$ and $M_i$ are left-right Bernoulli crossed products $(A_0,\tau_0)^{\Gamma} \rtimes (\Gamma \times \Gamma)$ with $\Gamma$ in $\cC$, possibly twisted with a $2$-cocycle. In this section, in Theorem \ref{theorem:coarse-embeddings-gen}, we consider left-right Bernoulli crossed products that arise as co-inductions of a given trace preserving action $\Gamma \curvearrowright^\beta (A_0, \tau_0)$. The usual left-right Bernoulli action corresponds to taking the trivial action $\be = \id$. Both the formulation and the proof of Theorem \ref{theorem:coarse-embeddings-gen} are very similar to \cite[Theorem~5.11]{DV24a}. We therefore only give a brief presentation, highlighting the necessary modifications, and we do not state the theorem in its highest possible generality, since that is not needed for this paper.

Given a countable group $\Gamma$ and a trace preserving action $\Gamma \actson^\be (B,\tau)$, we thus define $(A, \tau) = (B, \tau)^\Gamma$, with the embeddings $\pi_k : B \to A$ as the $k$'th tensor factor. We then consider $\Gamma \times \Gamma \curvearrowright^\alpha (A, \tau)$ via
\begin{equation*}
\alpha_{(g,h)}(\pi_k(a)) = \pi_{gkh\inv}(\beta_g(a)) \quad\text{for all $(g,h) \in \Gamma$ and $a \in B$.}
\end{equation*}
This is the \emph{co-induced left-right Bernoulli action} associated with $\Gamma \actson^\be (B,\tau)$. Exactly as in \cite{DV24a}, we say that a group homomorphism $\delta: \Gamma \times \Gamma \to \Gamma_1 \times \Gamma_1$ is \emph{symmetric} if $\delta$ is of the form $\delta = \delta_1 \times \delta_1$, or of the form $\delta = \sigma \circ (\delta_1 \times \delta_1)$ where $\sigma$ is the flip map, and $\delta_1 : \Gamma \to \Gamma_1$ is a group homomorphism.

\begin{theorem}[{Variant of Theorem~5.11 in \cite{DV24a}}]\label{theorem:coarse-embeddings-gen}
For $0 \leq i \leq 2$, let $\Gamma_i$ be countable groups in class $\cC$, $(B_i,\tau_i)$ nontrivial amenable tracial von Neumann algebras and $\Gamma_i \actson^{\beta_i} (B_i,\tau_i)$ trace preserving actions. Consider the co-induced left-right Bernoulli action $\Gamma_i \times \Gamma_i \curvearrowright^{\alpha_i} (A_i,\tau_i) := (B_i, \tau_i)^{\Gamma_i}$. Set $M_i = A_i \rtimes_{\alpha_i} (\Gamma_i \times \Gamma_i)$. Assume that $\Ker \be_0$ is infinite.

If $\psi: M_0 \to p(M_1 \ovt M_2)p$ is a coarse embedding for some projection $p$, then $p=1$ and after a unitary conjugacy, $\psi(\pi_e(B_0)) \subset \pi_e(B_1) \ovt \pi_e(B_2)$ and $\psi(u_r) = \om(r) \, u_{\delta_1(r)} \ot u_{\delta_2(r)}$ for all $r \in \Gamma_0 \times \Gamma_0$, where $\delta_i : \Gamma_0 \times \Gamma_0 \to \Gamma_i \times \Gamma_i$ are faithful symmetric homomorphisms and $\om : \Gamma_0 \times \Gamma_0 \to \T$ is a character.

Also, if there is no irreducible infinite index subfactor $P \subset M_1$ such that $\psi(M_0)$ can be unitarily conjugated into $P \ovt M_2$, then $\delta_1$ is an isomorphism.
\end{theorem}

\begin{proof}
Since the proof is very similar to the proof of \cite[Theorem~5.11]{DV24a}, we only give a sketch, indicating the necessary modifications. Write $\Lambda := \Ker \be_0$. Since $\Gamma_0$ belongs to class $\cC$ and $\Lambda$ is a nontrivial normal subgroup of $\Gamma_0$, we get that $\Lambda$ is a nonamenable icc group. Note that the von Neumann subalgebra $N_0 := (B_0,\tau_0)^\Lambda \rtimes_{\al_0} (\Lambda \times \Lambda)$ of $M_0$ is an ordinary left-right Bernoulli crossed product, precisely because $\Lambda = \Ker \be_0$.

We can then literally repeat the first 7 steps of the proof of \cite[Theorem~5.11]{DV24a}, applied to the restriction of $\psi$ to $N_0$, and arrive at the following point: $p=1$ and, after a unitary conjugacy, $\psi(\pi_e(B_0)) \subset \pi_e(B_1) \ovt \pi_e(B_2)$ and
\begin{equation}\label{eq.form-of-psi}
\psi(u_r) \in \T (u_{\delta_1(r)} \ot u_{\delta_2(r)}) \quad\text{for all $r \in \Lambda_0 \times \Lambda_0$,}
\end{equation}
where $\Lambda_0 \lhd \Lambda$ is a finite index normal subgroup and $\delta_i$ is either of the form $\gamma_i \times \gamma_i$, or of the form $\si \circ (\gamma_i \times \gamma_i)$, where $\gamma_i : \Lambda_0 \to \Gamma_i$ is a faithful group homomorphism.

We first claim that $\gamma_i$ uniquely extends to a group homomorphism $\Lambda \to \Gamma_i$. By symmetry, it suffices to consider $i=1$. Since $\Lambda_0$ is nonamenable and $\Gamma_i$ belongs to class $\cC$, the subgroup $\gamma_1(\Lambda_0) < \Gamma_1$ is relatively icc. Assume that $\gamma_1$ does not uniquely extend to a homomorphism $\Lambda \to \Gamma_1$. Consider the left-right action $\Gamma_1 \times \Gamma_1 \actson \Gamma_1$ and define the subgroup $S_k < \Lambda_0 \times \Lambda_0$ by $S_k = \{(g,k^{-1} g k)  \mid g \in \Lambda_0\}$. By \cite[Lemma 2.6(iii)]{DV24a}, $\delta_1(S_k) \cdot s$ is infinite for all $s \in \Gamma_1$.

Since $\delta_1(S_k) < \Gamma_1 \times \Gamma_1$ is relatively icc and since $\psi(\pi_k(B_0))$ commutes with $\psi(u_r)$ for all $r \in S_k$, it then follows from \eqref{eq.form-of-psi} and Lemma \ref{lem.relative-commutant.two} that $\psi(\pi_k(B_0)) \in 1 \ot A_2$. Conjugating with $\psi(u_r)$ for $r \in \Lambda_0 \times \Lambda_0$, it follows that $\psi(B_0^{\Lambda_0 k}) \subset 1 \ot A_2$, which contradicts the coarseness of $\psi$.

So $\gamma_i$ uniquely extends to a group homomorphism $\Lambda \to \Gamma_i$ that we still denote by $\gamma_i$. We still write $\delta_i = \gamma_i \times \gamma_i$, resp.\ $\delta_i = \sigma \circ (\gamma_i \times \gamma_i)$. Take $g \in \Lambda$. Since $\Lambda_0$ is a normal subgroup of $\Lambda$, it follows from \eqref{eq.form-of-psi} that $\psi(u_g)^* (u_{\delta_1(g)} \ot u_{\delta_2(g)})$ commutes up to a scalar with $u_{\delta_1(r)} \ot u_{\delta_2(r)}$ for all $r \in \Lambda_0 \times \Lambda_0$. By Lemma \ref{lem.relative-commutant.two}, this forces $\psi(u_g)^* (u_{\delta_1(g)} \ot u_{\delta_2(g)})$ to be a scalar. We have thus shown that \eqref{eq.form-of-psi} actually holds for all $r \in \Lambda$.

Since $\Lambda$ is a normal subgroup of $\Gamma$, we can repeat the same argument and find that $\gamma_i : \Lambda \to \Gamma_i$ uniquely extends to a group homomorphism $\Gamma_0 \to \Gamma_i$ that we still denote as $\gamma_i$. Extending accordingly $\delta_i = \gamma_i \times \gamma_i$, resp.\ $\delta_i = \sigma \circ (\gamma_i \times \gamma_i)$, the same argument shows that \eqref{eq.form-of-psi} remains valid for all $r \in \Gamma_0 \times \Gamma_0$.

To prove that $\delta_i : \Gamma_0 \times \Gamma_0 \to \Gamma_i \times \Gamma_i$ is faithful, assume the contrary. Since $\Gamma_0$ is icc, the kernel of $\delta_i$ is infinite. Then \eqref{eq.form-of-psi} contradicts the coarseness of $\psi$.

Finally assume that $\delta_1 : \Gamma_0 \times \Gamma_0 \to \Gamma_1 \times \Gamma_1$ is not surjective. Define $T_1 := \gamma_1(\Gamma_0)$. The form of $\psi$ implies that $\psi(M_0) \subset P \ovt M_2$, where $P = (B_1,\tau_1)^{T_1} \rtimes_{\al_1} (T_1 \times T_1)$. If $T_1 \neq \Gamma_1$, $P$ is an irreducible infinite index subfactor of $M_1$.
\end{proof}

We apply Theorem \ref{theorem:coarse-embeddings-gen} to embeddings given by comultiplications on compact quantum groups. We thus need the following lemma, which is a straightforward generalization of \cite[Proposition~7.2(2)]{IPV10}. In the second statement, we make use of Popa's intertwining-by-bimodules, denoted by $\prec$.

\begin{lemma}\label{lem.comult-coarse}
Let $(A, \Delta)$ be a Kac type compact quantum group.
\begin{enumlist}
\item\label{lem.comult-coarse.one} The embedding $\Delta : A \to A \ovt A$ is coarse.
\item\label{lem.comult-coarse.two} If $P \subset A$ is a von Neumann subalgebra and $\Delta(A) \prec_{A \ovt A} A \ovt P$, then $A \prec_A P$.
\end{enumlist}
\end{lemma}

\begin{proof}
We denote by $\vphi$ the Haar state on $(A,\Delta)$. Denote by $L^2(A)$ its GNS Hilbert space.

(i) By Theorem \ref{cqg.2}, there is a unitary $W : L^2(A \ovt A) \to L^2(A \ovt A)$ such that $W^*(b \ot a) = \Delta(a)(b \ot 1)$ for all $a,b \in M$. Then $W$ defines a unitary equivalence between the bimodule $\bim{\Delta(A)}{L^2(A \ovt A)}{A \ot 1}$ and a coarse $A$-$A$-bimodule. By symmetry, also $\bim{\Delta(A)}{L^2(A \ovt A)}{1 \ot A}$ is coarse, so that $\Delta : A \to A \ovt A$ is a coarse embedding.

(ii) Denote by $\cN := \langle A , e_P \rangle$ the basic construction for the inclusion $P \subset A$. So, $\cN \subset B(L^2(A))$ is the commutant of the right action of $P$ on $L^2(A)$. Also, $\cN$ is generated by the left action of $A$ and the orthogonal projection $e_P$ of $L^2(A)$ onto $L^2(P)$. Recall that $\cN$ has a canonical semifinite faithful trace $\Tr$ satisfying $\Tr(a e_P b) = \vphi(ab)$ for all $a,b \in A$.

Since we can view $A \ovt \cN$ as the basic construction for the inclusion $A \ovt P \subset A \ovt A$, by the assumption that $\Delta(A) \prec_{A \ovt A} A \ovt P$, we find a positive element $p \in \Delta(A)' \cap A \ovt \cN$ such that $0 < (\vphi \ot \Tr)(p) < +\infty$. Define $q \in \cN$ by $q = (\vphi \ot \id)(p)$. Then $q$ is a positive element of $\cN$ and $0 < \Tr(q) < +\infty$. We prove that $q$ commutes with $A$. Once this is proven, the intertwining $A \prec_A P$ follows.

By Theorem \ref{cqg.5}, the coefficients of the finite dimensional unitary corepresentations of $(A,\Delta)$ form a dense $*$-subalgebra of $A$. It thus suffices to take a unitary corepresentation $X \in A \ot M_n(\C)$ and prove that $X_{ij} q = q X_{ij}$ for all $i,j$.

We also use the notation $\Tr$ to denote the non normalized trace on $M_n(\C)$. Since $p \otimes 1 = (\Delta \otimes \id)(X) (p \otimes 1) (\Delta \otimes \id ) (X)^*$, applying $\vphi \ot \id \ot \Tr$ gives us
\begin{align*}
n q &= \sum_{i,j,r,s} (\vphi \otimes \id) \bigl( (X_{ir} \otimes X_{rj}) \, p \, (X_{is} \otimes X_{sj})^* \bigr) = \sum_{i,j,r,s} (\vphi \otimes \id) \bigl( (X_{is}^* X_{ir} \otimes X_{rj}) \, p \, (1 \otimes X_{sj}^*) \bigr)\\ &= \sum_{j,r} (\vphi \otimes \id) \bigl((1 \otimes X_{rj}) \, p \, (1 \otimes X_{rj}^*) \bigr) = \sum_{j,r}X_{rj} q X_{rj}^* \; .
\end{align*}
Since
\begin{align*}
\| X_{rj} q - qX_{rj}\|_2^2 &= \vphi \bigl( (X_{rj} q - q X_{rj})(qX_{rj}^* - X_{rj}^* q) \bigr)\\
&= \vphi(X_{rj} q^2 X_{rj}^*) + \vphi(q X_{rj} X_{rj}^* q) - \vphi(q X_{rj}qX_{rj}^*) - \vphi(X_{rj}qX_{rj}^* q) \\
&= \vphi(q^2 X_{rj}^*X_{rj}) + \vphi(q X_{rj} X_{rj}^* q) - \vphi(q X_{rj}qX_{rj}^*) - \vphi(X_{rj}qX_{rj}^* q) \; ,
\end{align*}
summing over $r,j$ and using the equality $\sum_{j,r}X_{rj} q X_{rj}^* = n q$ established above, we find that $\sum_{r,j} \| X_{rj} q - qX_{rj}\|_2^2 = 0$. So, $X_{rj} q = q X_{rj}$ for all $r,j$ and the lemma is proven.
\end{proof}

In the proof of Theorem \ref{theorem:coarse-embeddings-gen}, we used the following elementary result on relative commutants, which we will also use in the proof of Theorem \ref{thm.generic-theorem} below.

\begin{lemma}\label{lem.relative-commutant}
Let $G \actson I$ be an action of a countable group $G$ on a countable set $I$. Let $(B_0,\tau_0)$ be any tracial von Neumann algebra, and consider the tensor power $(D,\tau) = (B_0,\tau_0)^I$ with the embeddings $\pi_i : B_0 \to D$ as the $i$'th tensor factor. Let $G \actson^\gamma (D,\tau)$ be any trace preserving action satisfying $\gamma_g(\pi_i(B_0)) = \pi_{g \cdot i}(B_0)$ for all $g \in G$, $i \in I$.

Let $\Lambda$ be a group and for $i \in \{1,2\}$, let $\delta_i : \Lambda \to G$ be group homomorphisms such that $\delta_i(\Lambda) < G$ is relatively icc. Let $X \in D \rtimes_\gamma G$ such that $u_{\delta_1(s)} X u_{\delta_2(s)}^* = \om(s) X$ for all $s \in \Lambda$, where $\om : \Lambda \to \T$ is a group homomorphism.
\begin{enumlist}
\item\label{lem.relative-commutant.one} If $X \neq 0$, there exists a $g_0 \in G$ such that $\delta_1 = (\Ad g_0) \circ \delta_2$.
\item\label{lem.relative-commutant.two} If $\delta_1 = \delta_2$, define $I_1$ as the set of $i \in I$ for which $\delta_1(\Lambda) \cdot i$ is finite. Then $X \in (B_0,\tau_0)^{I_1}$.
\end{enumlist}
\end{lemma}

If in (ii), $I_1 = \emptyset$, we interpret $(D_0,\tau_0)^{I_1} = \C 1$. Note that if $G_j \actson^{\gamma_j} D_j := (B_0,\tau_0)^{I_j}$ are actions as in the lemma, then also the product action of $G_1 \times \cdots \times G_k$ on $D_1 \ovt \cdots \ovt D_k$ fits into the lemma by taking $I = I_1 \sqcup \cdots \sqcup I_k$.

\begin{proof}
For $x \in D \rtimes_\gamma G$, we denote by $x = \sum_{g \in G} (x)_g u_g$ its Fourier decomposition, with $(x)_g \in D$ and $\sum_{g \in G} \|(x)_g\|_2^2 = \|x\|_2^2$.

(i) Since $(X)_g = \om(s) \, (X)_{\delta_1(s) g \delta_2(s)^{-1}}$ for all $s \in \Lambda$, $g \in G$, the function $g \mapsto |(X)_g|$ is square summable and invariant under the action of $\Lambda$ on $G$ given by $g \mapsto \delta_1(s) g \delta_2(s)^{-1}$. So whenever $(X)_g \neq 0$, the set $\{\delta_1(s) g \delta_2(s)^{-1} \mid s \in \Lambda\}$ is finite. Since $\delta_1(\Lambda) < G$ is relatively icc, this set must then be equal to $\{g\}$, so that $\delta_1 = (\Ad g) \circ \delta_2$.

(ii) From the argument in (i), it already follows that $(X)_g = 0$ for all $g \neq e$. This means that $X \in D$. For every finite subset $\cF \subset I \setminus I_1$, define $H_\cF \subset L^2(D)$ as the closed linear span of $(B_0 \ominus \C 1)^\cF B_0^{I_1}$. Note that if $\cF \neq \cF'$, then $H_\cF \perp H_{\cF'}$. We denote by $P_\cF$ the orthogonal projection of $L^2(D)$ onto $H_\cF$. For every $s \in \Lambda$, we have that $\gamma_{\delta_1(s)}(H_{\cF}) = H_{\delta_1(s) \cdot \cF}$. So, the square summable function $\cF \mapsto \|P_{\cF}(X)\|_2^2$ is $\delta_1(\Lambda)$-invariant. When $\cF$ is a finite nonempty set, by definition, $\delta_1(\Lambda) \cdot \cF$ is infinite, so that $P_{\cF}(X) = 0$. So $X \in H_{\emptyset} = L^2((B_0,\tau_0)^{I_1})$.
\end{proof}

\subsection{\boldmath Proof of Theorem \ref{thm.generic-theorem} and Theorem \ref{thm.main}}\label{sec.proof-superrigidity-theorem}

\begin{proof}[{Proof of Theorem \ref{thm.generic-theorem}}]
{\bf\boldmath (iii) $\Rightarrow$ (ii).} Assume that $(Q, \Delta_Q)$ is a compact quantum group, $p \in M$ a nonzero projection and $\pi : pMp \to Q$ a von Neumann algebra isomorphism. It follows that $Q$ admits a faithful normal tracial state. By \cite[Proposition A.1]{Sol06}, $(Q, \Delta_Q)$ is of Kac type. Since $M$ is a II$_1$ factor, $\pi$ is automatically trace preserving. By Lemma \ref{lem.comult-coarse.one}, the comultiplication $\Delta_Q : Q \to Q \ovt Q$ is a coarse embedding. So, also $\Delta_1 := (\pi^{-1} \ot \pi^{-1}) \circ \Delta_Q \circ \pi$ is a coarse embedding $pMp \to pMp \ovt pMp$ that we amplify to a coarse embedding $\Psi: M \to q(M \ovt M)q$, where $q$ is any projection with the same trace as $p$.

From Theorem \ref{theorem:coarse-embeddings-gen}, it follows that $q=1$, and thus $p=1$. In particular, $\Psi = \Delta_1$. We have that $(M,\Delta_1)$ is a Kac type compact quantum group and that $\pi : M \to Q$ is, tautologically, a quantum group isomorphism. Moreover, Theorem \ref{theorem:coarse-embeddings-gen} gives us a unitary $\Omega \in M \ovt M$, a character $\chi: \Gamma \times \Gamma \to \C$, and faithful symmetric group homomorphisms $\delta_1, \delta_2: \Gamma \times \Gamma \to \Gamma \times \Gamma$ such that $\Omega \Delta_1(\pi_e(A_0)) \Omega^* \subset \pi_e(A_0) \ovt \pi_e(A_0)$, and
\begin{equation}\label{eq.Delta-1-ug}
\Omega \Delta_1(u_g) \Omega^* = \chi(g) (u_{\delta_1(g)} \otimes u_{\delta_2(g)}) \quad\text{for all $g \in \Gamma \times \Gamma$.}
\end{equation}
By Lemma \ref{lem.comult-coarse.two} and the final part of Theorem \ref{theorem:coarse-embeddings-gen}, the $\delta_i$ are symmetric group automorphisms.

Since $\Delta_1$ is co-associative, we write $\Delta_1^{(2)} = (\Delta_1 \ot \id) \circ \Delta_1 = (\id \ot \Delta_1) \circ \Delta_1$. Applying $\Ad(\Om \otimes 1) \circ (\Delta_1 \otimes \id)$ to \eqref{eq.Delta-1-ug} gives
$$(\Omega \ot 1)(\Delta_1 \ot \id)(\Omega) \, \Delta_1^{(2)}(u_g) \, (\Delta_1 \ot \id)(\Omega)^* (\Omega^* \ot 1) = \chi(g) \chi(\delta_1(g)) \, (u_{\delta_1(\delta_1(g))} \otimes u_{\delta_2(\delta_1(g))} \otimes u_{\delta_2(g)}) \; .$$
Applying $\Ad(1 \ot \Om) \circ (\id \ot \Delta_1)$ to \eqref{eq.Delta-1-ug}, we similarly get that
$$(1 \ot \Omega)(\id \ot \Delta_1)(\Omega) \, \Delta_1^{(2)}(u_g) \, (\id \ot \Delta_1)(\Omega)^* (1 \ot \Omega^*) = \chi(g) \chi(\delta_2(g)) \, (u_{\delta_1(g)} \otimes u_{\delta_1(\delta_2(g))} \otimes u_{\delta_2(\delta_2(g))}) \; .$$
Define the unitary $X = (1 \ot \Om)(\id \ot \Delta_1)(\Om) (\Delta_1 \ot \id)(\Om^*) (\Om^* \ot 1)$. It follows that
\begin{equation}\label{eq.comm-X}
X \, \chi(\delta_1(g)) \, (u_{\delta_1(\delta_1(g))} \otimes u_{\delta_2(\delta_1(g))} \otimes u_{\delta_2(g)}) = \chi(\delta_2(g)) \, (u_{\delta_1(g)} \otimes u_{\delta_1(\delta_2(g))} \otimes u_{\delta_2(\delta_2(g))}) \, X
\end{equation}
for all $g \in \Gamma \times \Gamma$. Since $\Gamma \times \Gamma$ is icc, by Lemma \ref{lem.relative-commutant.one}, the automorphisms $\delta_1$ and $\delta_2$ are inner. So after modifying $\Om$, we may assume that $\delta_1 = \delta_2 = \id$. It then follows from \eqref{eq.comm-X} and Lemma \ref{lem.relative-commutant.two} that $X \in \T 1$. By Lemma \ref{lem.2-cocycle-up-to-scalar}, $X=1$ and $\Omega$ is a unitary 2-cocycle for the Kac type quantum group $(M, \Delta_1)$.

By Proposition \ref{coc.3}, defining $\Phi(a) = \Omega \Delta_1(a)\Omega^*$ for all $a \in M$, we obtain the new Kac type compact quantum group $(M, \Phi)$. By construction, $\Phi(u_g) = \chi(g) (u_g \otimes u_g)$ for all $g \in \Gamma \times \Gamma$, and $\Phi(\pi_e(A_0)) \subset \pi_e(A_0) \ovt \pi_e(A_0)$.

Since $\Phi(\pi_e(A_0)) \subset \pi_e(A_0) \ovt \pi_e(A_0)$, we define $\Phi_0 : A_0 \to A_0 \ovt A_0$ such that $(\pi_e \ot \pi_e) \circ \Phi_0 = \Phi \circ \pi_e$. By construction, $(A_0,\Phi_0)$ is a Kac type compact quantum group. Since $\Phi(u_{(g,g)}) = \chi(g,g) (u_{(g,g)} \otimes u_{(g,g)})$ for all $g\in \Gamma$, the automorphisms $(\be_g)_{g \in \Gamma}$ are also quantum group automorphisms of $(A_0,\Phi_0)$. By assumption c), there exists a trace preserving automorphism $\pi_0$ of $A_0$ such that $\pi_0 : (A_0,\Delta_0) \to (A_0,\Phi_0)$ is a quantum group isomorphism, and a $\delta_0 \in \Aut \Gamma$ satisfying $\pi_0 \circ \be_g = \be_{\delta_0(g)} \circ \pi_0$ for all $g \in \Gamma$.

Define $\pi_1$ as the unique automorphism of $M$ satisfying $\pi_1 \circ \pi_e = \pi_e \circ \pi_0$ and $\pi_1(u_{(g,h)}) = \chi(\delta_0(g),\delta_0(h)) \, u_{(\delta_0(g),\delta_0(h))}$ for all $g,h \in \Gamma$. By construction, the equality $\Phi \circ \pi_1 = (\pi_1 \ot \pi_1) \circ \Delta$ holds on $\pi_e(A_0)$ and on $u_{(g,h)}$, so that it holds everywhere and $\pi_1 : (M,\Delta) \to (M,\Phi)$ is a quantum group isomorphism.

Since $\Om^*$ is a unitary $2$-cocycle for $(M,\Phi)$, we get that $(\pi_1 \ot \pi_1)^{-1}(\Om^*)$ is a unitary $2$-cocycle for $(M,\Delta)$. By assumptions a) and b), together with Propositions \ref{prop.cocycle-reduce-to-core.one} and \ref{prop.cocycles-products}, this unitary $2$-cocycle for $(M,\Delta)$ is a coboundary. It follows that $\Om^*$ is a coboundary as a unitary $2$-cocycle for $(M,\Phi)$, so that $\Om$ is a coboundary as a unitary $2$-cocycle for $(M,\Delta_1)$. We thus find a unitary $w \in M$ such that $\Ad w : (M,\Phi) \to (M,\Delta_1)$ is a quantum group isomorphism. Writing $v = \pi(w)$, we have found a unitary $v \in Q$ such that $(\Ad v) \circ \pi \circ \pi_1 : (M,\Delta) \to (Q,\Delta_Q)$ is a quantum group isomorphism. This concludes the proof of (iii) $\Rightarrow$ (ii).

{\bf\boldmath (II) $\Rightarrow$ (I).} We follow the proof of (iii) $\Rightarrow$ (ii), but we replace condition c) by the stronger condition that $(A_0,\Delta_0)$ is strictly rigid relative to $\be(\Gamma) < \Aut(A_0,\Delta_0)$. So in that proof, at the point where we observe that the automorphisms $(\be_g)_{g \in \Gamma}$ are also quantum group automorphisms of $(A_0,\Phi_0)$, we apply Definition \ref{def.relative-rigidity-cqg.one} to the identity isomorphism $\id$ of $A_0$ and conclude that $\id = \pi_0 \circ \lambda_{\om_0}^{-1}$, where $\pi_0$ is a quantum group isomorphism $\pi_0 : (A_0,\Delta_0) \to (A_0,\Phi_0)$ and $\lambda_{\om_0} : A_0 \to A_0$ is the left translation automorphism of $(A_0,\Delta_0)$ given by a $\be(\Gamma)$-invariant character $\om_0$ of $(A_0,\Delta_0)$. More precisely, denote by $\cA_0 \subset A_0$ the dense $*$-subalgebra given by Theorem \ref{cqg.5}, spanned by the coefficients of the finite dimensional unitary representations of $(A_0,\Delta_0)$. As explained in Section \ref{sec.quantum-group-like-aut}, we have that $\om_0 : \cA_0 \to \C$ is a unital $*$-homomorphism and $\lambda_{\om_0}(a) = (\om_0 \ot \id)\Delta_0(a)$ for all $a \in \cA_0$. The $\be(\Gamma)$-invariance of $\om_0$ means that $\om_0 \circ \be_g = \om_0$, and thus $\lambda_{\om_0} \circ \be_g = \be_g \circ \lambda_{\om_0}$, for all $g \in \Gamma$. Since $\id = \pi_0 \circ \lambda_{\om_0}^{-1}$, we get that $\pi_0 = \lambda_{\om_0}$ so that also $\pi_0 \circ \be_g = \be_g \circ \pi_0$ for all $g \in \Gamma$.

In the proof of Proposition \ref{prop.cocycles-products}, we explained that the canonical dense $*$-subalgebra of $(A,\Delta)$ is the algebraic tensor product $\ot^{\text{\rm alg}}_{k \in \Gamma} \cA_0$. By the discussion in the proof of Proposition \ref{prop.cocycle-reduce-to-core}, the canonical dense $*$-subalgebra $\cM$ of $(M,\Delta)$ then is the algebraic crossed product with $\Gamma \times \Gamma$, i.e.\ the $*$-algebra generated by $\pi_e(\cA_0)$ and $u_{(g,h)}$, $g,h \in \Gamma$. We can thus uniquely define the $*$-homomorphism $\om : \cM \to \C$ such that $\om(\pi_e(a)) = \om_0(a)$ and $\om(u_{(g,h)}) = \chi(g,h)$ for all $a \in \cA_0$ and $g,h \in \Gamma$. By construction, the quantum group isomorphism $\pi_1 : (M,\Delta) \to (M,\Phi)$ that we constructed in the proof of (iii) $\Rightarrow$ (ii), then satisfies $\pi_1(b) = (\om \ot \id)\Delta(b)$ for all $b \in \cM$. This means that $\pi_1$ is also a translation automorphism of $(M,\Delta)$.

Since we have shown above that $\theta := (\Ad v) \circ \pi \circ \pi_1$ is a quantum group isomorphism from $(M,\Delta)$ to $(Q,\Delta_Q)$, we have thus written $\pi = (\Ad v^*) \circ \theta \circ \pi_1^{-1}$, where $\theta$ is a quantum group isomorphism and $\pi_1^{-1}$ is a translation automorphism. This concludes the proof of (II) $\Rightarrow$ (I).

{\bf\boldmath (ii) $\Rightarrow$ (i).} This is trivial.

{\bf\boldmath (i) $\Rightarrow$ (iii).} By \cite[Proposition~3.7]{PV21} (or a variant of Theorem \ref{theorem:coarse-embeddings-gen}), every automorphism of the von Neumann algebra $M$ is unitarily conjugate to an automorphism $\theta$ satisfying $\theta(u_g) \in \T u_{\delta(g)}$ for all $g \in \Gamma \times \Gamma$, where $\delta \in \Aut(\Gamma \times \Gamma)$. Since we assume that $(M,\Delta)$ is quantum W$^*$-superrigid, it follows from Proposition \ref{prop.crossed-product-not-rigid} that all unitary $2$-cocycles on $(A,\Delta)$ and on $(L(\Gamma \times \Gamma),\Delta)$ are a coboundary.

If $C_0 < \Gamma$ is a nontrivial finite abelian subgroup, because $H^2(\widehat{C_0} \times \widehat{C_0},\T)$ is nontrivial and $\Gamma$ is icc, it follows from Proposition \ref{prop.not-rigid} that $(L(\Gamma \times \Gamma),\Delta)$ admits a unitary $2$-cocycle that is not a coboundary. So, $\Gamma$ has no nontrivial finite abelian subgroups, i.e.\ a) holds.

Since every unitary $2$-cocycle on $(A,\Delta)$ is a coboundary, it follows from Proposition \ref{prop.cocycles-products} that b) holds.

To prove c), assume that $\Gamma \actson^\theta (A_1,\Delta_1)$ is an action by quantum group automorphisms and that $\pi : A_0 \to A_1$ is a Haar state preserving von Neumann algebra isomorphism such that $\pi \circ \be_g = \theta_g \circ \pi$ for all $g \in \Gamma$. Since the Haar state of $(A_0,\Delta_0)$ is a trace, the same holds for the Haar state of $(A_1,\Delta_1)$, which is thus of Kac type. We have to prove that there exists a quantum group isomorphism $\pi_0 : (A_0,\Delta_0) \to (A_1,\Delta_1)$ and an automorphism $\delta_0 \in \Aut \Gamma$ such that $\pi_0 \circ \be_g = \theta_{\delta_0(g)} \circ \pi_0$ for all $g \in \Gamma$.

Using $\Gamma \actson^\theta (A_1,\Delta_1)$, we consider a similar co-induced left-right Bernoulli action and construct the compact quantum group $(M_1,\Delta_1)$ with underlying von Neumann algebra $M_1 = (A_1,\vphi_1)^\Gamma \rtimes (\Gamma \times \Gamma)$. By definition, $\Delta_1 \circ \pi_e = (\pi_e \ot \pi_e) \circ \Delta_1$ and $\Delta_1(u_g) = u_g \ot u_g$ for all $g \in \Gamma \times \Gamma$. Then $\pi$ induces a von Neumann algebra isomorphism $\Pi : M \to M_1$. Since we assume that (i) holds, we find a quantum group isomorphism $\Pi_0 : (M,\Delta) \to (M_1,\Delta_1)$.

By \cite[Proposition~3.7]{PV21} (or a variant of Theorem \ref{theorem:coarse-embeddings-gen}), there exists a unitary $v \in M_1$, a symmetric automorphism $\delta \in \Aut(\Gamma \times \Gamma)$ and a character $\chi$ on $\Gamma \times \Gamma$ such that $v \Pi_0(u_g) v^* = \chi(g) u_{\delta(g)}$ for all $g \in \Gamma \times \Gamma$. We denote by $\delta_0$ the automorphism of $\Gamma$ such that $\delta = \delta_0 \times \delta_0$ or $\delta = \sigma \circ (\delta_0 \times \delta_0)$. Since $\Pi_0$ is a quantum group isomorphism, we get for all $g \in \Gamma \times \Gamma$ that
\begin{align*}
\chi(g) \, u_{\delta(g)} \ot u_{\delta(g)} &= \Delta_1(\chi(g) u_{\delta(g)}) = \Delta_1(v \Pi_0(u_g) v^*) = \Delta_1(v) \, (\Pi_0 \ot \Pi_0)\Delta(u_g) \, \Delta_1(v^*) \\
&= \chi(g)^2 \, \Delta_1(v)(v^* \ot v^*) \, (u_{\delta(g)} \ot u_{\delta(g)}) \, (v \ot v) \Delta_1(v^*) \; .
\end{align*}
By Lemma \ref{lem.relative-commutant.one}, this commutation property first forces $\Delta_1(v) (v^* \ot v^*)$ to be a scalar and then $\chi = 1$. Since $\Delta_1(v) (v^* \ot v^*)$ is a scalar, $\Ad v$ is a quantum group automorphism of $(M_1,\Delta_1)$. We may thus replace $\Pi_0$ by $(\Ad v) \circ \Pi_0$ and have found a quantum group isomorphism $\Pi_0 : (M,\Delta) \to (M_1,\Delta_1)$ satisfying $\Pi_0(u_g) = u_{\delta(g)}$ for all $g \in \Gamma \times \Gamma$.

In particular, $\Pi_0(u_{(g,g)}) = u_{(\delta_0(g),\delta_0(g))}$ for all $g \in \Gamma$. Considering the relative commutant of the unitaries $u_{(g,g)}$, $g \in \Gamma$, it follows from Lemma \ref{lem.relative-commutant.two} that $\Pi_0(\pi_e(A_0)) = \pi_e(A_1)$. Since $\Pi_0$ is a quantum group isomorphism, we find a quantum group isomorphism $\pi_0 : (A_0,\Delta_0) \to (A_1,\Delta_1)$ such that $\Pi_0 \circ \pi_e = \pi_e \circ \pi_0$. Since $\Pi_0(u_{(g,g)}) = u_{(\delta_0(g),\delta_0(g))}$ for all $g \in \Gamma$, we get that $\pi_0 \circ \be_g = \theta_{\delta_0(g)} \circ \pi_0$ for all $g \in \Gamma$.

{\bf\boldmath (I) $\Rightarrow$ (II).} In the proof of (i) $\Rightarrow$ (iii), we have to reach the stronger conclusion that $\pi$ can be written as $\pi = \pi_0 \circ \pi_1$, where $\pi_0 : (A_0,\Delta_0) \to (A_1,\Delta_1)$ is a quantum group isomorphism and $\pi_1 \in \Aut A_0$ is a left translation automorphism of $(A_0,\Delta_0)$ given by a $\be(\Gamma)$-invariant character on $(A_0,\Delta_0)$.

As in the proof of (i) $\Rightarrow$ (iii), we construct the compact quantum group $(M_1,\Delta_1)$ and the von Neumann algebra isomorphism $\Pi : M \to M_1$ satisfying $\Pi \circ \pi_e = \pi_e \circ \pi$ and $\Pi(u_g) = u_g$ for all $g \in \Gamma \times \Gamma$. Since (I) holds, we can decompose $\Pi = (\Ad w) \circ \Pi_0 \circ \Pi_1$ where $w \in M_1$ is a unitary, $\Pi_0 : (M,\Delta) \to (M_1,\Delta_1)$ is a quantum group isomorphism and $\Pi_1$ is a left translation automorphism of $(M,\Delta)$.

In the proof of (i) $\Rightarrow$ (iii), we analyzed how an arbitrary quantum group isomorphism $\Pi_0 : (M,\Delta) \to (M_1,\Delta_1)$ may look like. We thus find a symmetric automorphism $\delta \in \Aut(\Gamma \times \Gamma)$, either of the form $\delta = \delta_0 \times \delta_0$ or of the form $\delta = \si \circ (\delta_0 \times \delta_0)$, and a unitary $v \in M_1$, such that after replacing $w$ by $wv^*$ and $\Pi_0$ by $(\Ad v) \circ \Pi_0$, we may assume that the quantum group isomorphism $\Pi_0 : (M,\Delta) \to (M_1,\Delta_1)$ satisfies $\Pi_0(u_g) = u_{\delta(g)}$ for all $g \in \Gamma \times \Gamma$ and $\Pi_0(\pi_e(A_0)) = \pi_e(A_1)$. We define the quantum group isomorphism $\pi_0 : (A_0,\Delta_0) \to (A_1,\Delta_1)$ such that $\Pi_0 \circ \pi_e = \pi_e \circ \pi_0$.

In the proof of (II) $\Rightarrow$ (I), we described the canonical dense $*$-subalgebra $\cM \subset M$ as the $*$-algebra generated by $\pi_e(\cA_0)$, where $\cA_0 \subset A_0$ is the canonical dense $*$-subalgebra, and the unitaries $u_g$, $g \in \Gamma \times \Gamma$. Since $\Pi_1$ is a left translation automorphism of $(M,\Delta)$, we can take a unital $*$-homomorphism $\om : \cM \to \C$ such that $\Pi_1(b) = (\om \ot \id) \Delta(b)$ for all $b \in \cM$. Define the character $\chi : \Gamma \times \Gamma \to \T$ by $\chi(g) = \om(u_g)$. Also define the unital $*$-homomorphism $\om_1 : \cA_0 \to \C$ by $\om_1 = \om \circ \pi_e$. For every $g \in \Gamma$ and $a \in \cA_0$, we have that $u_{(g,g)} \pi_e(a) u_{(g,g)}^* = \pi_e(\be_g(a))$. Applying $\om$, it follows that the character $\om_1$ is $\be(\Gamma)$-invariant. Define the left translation automorphism $\pi_1$ of $(A_0,\Delta_0)$ by $\pi_1(a) = (\om_1 \ot \id)\Delta_0(a)$ for all $a \in \cA_0$. By construction, $\Pi_1 \circ \pi_e = \pi_e \circ \pi_1$ and $\Pi_1(u_g) = \chi(g) u_g$ for all $g \in \Gamma \times \Gamma$.

We have thus found that $(\Pi_0 \circ \Pi_1)(u_g) = \chi(g) u_{\delta(g)}$ and $\Pi(u_g) = u_g$ for all $g \in \Gamma \times \Gamma$, while also $\Pi = (\Ad w) \circ \Pi_0 \circ \Pi_1$. It follows that the symmetric automorphism $\delta$ is inner, i.e.\ $\delta_0 = \Ad s$ for some $s \in \Gamma$ and $\delta = \delta_0 \times \delta_0$, and then also that $w$ is a multiple of $u_{(s,s)}^*$. Restricting the equality $\Pi = (\Ad w) \circ \Pi_0 \circ \Pi_1$ to $\pi_e(A_0)$, it follows that $\pi = (\theta_{s^{-1}} \circ \pi_0) \circ \pi_1$. Since $\theta_{s^{-1}} \circ \pi_0$ is a quantum group isomorphism from $(A_0,\Delta_0)$ to $(A_1,\Delta_1)$ and $\pi_1$ is a left translation automorphism of $(A_0,\Delta_0)$, the proof of (I) $\Rightarrow$ (II) is complete.
\end{proof}

Finally, Theorem \ref{thm.main} follows from Theorem \ref{thm.generic-theorem} and the following list of concrete examples satisfying all the assumptions of that theorem.

\begin{examples}\label{ex.main}
In the following cases, all assumptions of Theorem \ref{thm.generic-theorem.3} are satisfied, so that by Theorem \ref{thm.generic-theorem}, the co-induced left-right Bernoulli crossed product construction gives Kac type compact quantum groups $(M,\Delta)$ that are quantum W$^*$-superrigid. In (i) and (ii.a), we even have that $(M,\Delta)$ is strictly quantum W$^*$-superrigid.
\begin{enumlist}
\item Take $(A_0,\Delta_0) = (L(\Lambda),\Delta_\Lambda)$, where $\Lambda$ is any torsion free amenable icc group. Take $\Gamma = \Lambda * \Z$ and the action given by $\be_g = \Ad u_g$ for $g \in \Lambda$ and $\be_a = \id$ for $a \in \Z$.
    This follows from Lemma \ref{lem.trivial-2-cohom-examples.two} and Theorem \ref{thm.rigid-LG-icc}.

\item Take $(A_0,\Delta_0) = (L^\infty(K),\Delta_K)$. Let $\cG < \Aut K$ be one of the following countable subgroups. Define $\Gamma = \F_{\cG}$ as the free group with free generators indexed by $\cG$, and let $\Gamma$ act on $K$ by $\be_\al(a) = \al(a)$ for every $\al \in \cG$.
    \begin{enumlist}
    \item $K = K_0^n$ with $n \geq 3$, $K_0$ a nontrivial abelian compact connected second countable group and $\SL(n,\Z) < \cG < \Aut K$.
    This follows from Theorem \ref{thm.rigid-Kn} and Lemma \ref{lem.trivial-2-cohom-examples.one}.
    \item $K=\SL_2(\F_p)$ with $p$ prime and $p \geq 5$, with $\cG = \Aut K$. This follows from Theorem \ref{thm.rigid-SL2} and Lemma \ref{lem.trivial-2-cohom-examples.one}.
    \item $K = \SL_n(\F_q)$, or $K = \F_q^n \rtimes \SL_n(\F_q)$ with $n \geq 3$, $q$ a prime power, and $(n,q) \not\in \{(3,2),(3,4),(4,2)\}$, with $\cG = \Aut K$.
    This follows from Theorem \ref{thm.rigid-SLn-Fp} and Lemma \ref{lem.trivial-2-cohom-examples.one}.
    \item $K = \Atil_n$, the canonical double cover of the alternating group $A_n$ with $n=5$ or $n \geq 8$, with $\cG = \Aut K$.
This follows from Theorem \ref{thm.relative-rigid-An-tilde} and Lemma \ref{lem.trivial-2-cohom-examples.one}.
    \item A direct product $K_1 \times K_2$ with $K_1$ being one of the groups in a) and $K_2$ one of the groups in b) or c), relative to $\cG \times \Aut K_2$. This follows from Proposition \ref{prop.direct-product-relative-rigid} and Lemma \ref{lem.trivial-2-cohom-examples.one}.
    \end{enumlist}
\end{enumlist}
\end{examples}

\end{document}